\documentclass[aos,preprint]{imsart}
\pdfoutput=1 

\usepackage[hang]{subfigure}
\usepackage{color}
\usepackage{latexsym}
\usepackage{mathtools}
\usepackage{amsmath}
\usepackage{amssymb}
\usepackage[utf8]{inputenc}
\usepackage{textcomp}
\usepackage{graphicx}
\usepackage{epsfig}
\usepackage{dsfont}
\usepackage{hyperref}
\usepackage{bbm}
\usepackage{enumitem}
\usepackage{here}

\usepackage{color}
\usepackage[dvipsnames]{xcolor}

\RequirePackage{amsthm,amsmath,amsfonts,amssymb}
\RequirePackage[authoryear]{natbib}
\RequirePackage[colorlinks,citecolor=blue,urlcolor=blue]{hyperref}
\RequirePackage{graphicx}

\startlocaldefs

\renewcommand{\epsilon}{\varepsilon}

\newcommand{\Z}{{\mathbb Z}}
\newcommand{\N}{\mathbb N}
\newcommand{\Q}{\mathbb Q}
\newcommand{\R}{\mathbb R}
\newcommand{\C}{\mathbb C}

\newcommand{\Prob}{\mathbb P}

\newcommand{\E}{\textnormal{\mbox{I\negthinspace E}}}

\newcommand{\Cov}{\mathrm{Cov}}

\newcommand{\eps}{\varepsilon}
\newcommand{\bea}{\begin{eqnarray*}}
\newcommand{\eea}{\end{eqnarray*}}
\newcommand{\be}{\begin{eqnarray}}
\newcommand{\ee}{\end{eqnarray}}
\newcommand{\ba}{\begin{array}}
\newcommand{\ea}{\end{array}}
\newcommand{\al}[1]{\begin{align*}#1\end{align*}}
\newcommand{\als}[1]{\begin{align}#1\end{align}}
\newcommand{\beq}{\begin{equation}}
\newcommand{\eeq}{\end{equation}}
\newcommand{\bi}{\begin{itemize}}
\newcommand{\ei}{\end{itemize}}
\newcommand{\bpm}{\begin{pmatrix}}
\newcommand{\epm}{\end{pmatrix}}

\newcommand{\proc}{(X_t)_{t\in \Z}}

\newcommand{\spec}{\boldsymbol{\mathfrak{f}}}
\newcommand{\specdis}{\boldsymbol{\mathfrak{F}}}
\newcommand{\specdisred}{\boldsymbol{\mathfrak{G}}}

\newcommand{\copper}{\mathcal{I}}

\newcommand{\cid}{\overset{\mathcal{D}}{\longrightarrow}}

\newcommand{\cum}{\text{\rm cum}}
\newcommand{\flt}[1]{\frac{\floor{(n-|k|)\tau_{#1}}}{n-|k|}}
\newcommand{\cet}[1]{\frac{1+\floor{(n-|k|)\tau_{#1}}}{n-|k|}}
\newcommand{\fltp}[1]{\frac{\floor{(n-|k|)\tau^{\prime}_{#1}}}{n-|k|}}
\newcommand{\cetp}[1]{\frac{1+\floor{(n-|k|)\tau^{\prime}_{#1}}}{n-|k|}}

\DeclarePairedDelimiter{\ceil}{\lceil}{\rceil}
\DeclarePairedDelimiter{\floor}{\lfloor}{\rfloor}

\theoremstyle{plain}
\newtheorem{theorem}{Theorem}[section]
\newtheorem{prop}{Proposition}[section]
\newtheorem{lemma}{Lemma}[section]
\theoremstyle{remark}

\newtheorem{example}{Example}[section]
\newtheorem{ass}{Assumption}[section]
\newtheorem{rem}{Remark}[section]

\newcommand{\uppi}{\pi}

\newcommand{\dd}{\mathrm{d}}
\newcommand{\rrvert}{\vert}
\newcommand{\rrVert}{\Vert}
\newcommand{\llvert}{\vert}
\newcommand{\llVert}{\Vert}

\newcommand{\IP}{\mathds{P}}
\newcommand{\IR}{\mathds{R}}
\newcommand{\IZ}{\mathds{Z}}

\endlocaldefs

\begin{document}

\begin{frontmatter}
	\title{
		The integrated copula spectrum}
	\runtitle{The integrated copula spectrum}
	
	\begin{aug} 
\author[A]{\fnms{Yuichi} \snm{Goto}\ead[label=e1]{yuu510@fuji.waseda.jp}},
\author[B]{\fnms{Tobias} \snm{Kley}\ead[label=e2,mark]{tobias.kley@uni-goettingen.de}}
\author[C]{\fnms{Ria} \snm{Van Hecke}\ead[label=e3,mark]{ria.vanhecke@rub.de}}
\author[D]{\fnms{Stanislav} \snm{Volgushev}\ead[label=e4,mark]{stanislav.volgushev@utoroto.ca}}
\author[E]{\fnms{Holger} \snm{Dette}\ead[label=e5,mark]{holger.dette@rub.de}}
\and 
\author[F]{\fnms{Marc}
\snm{Hallin}\ead[label=e6,mark]{hallinmarc@gmail.com}}

\address[A]{Waseda University, Tokyo, Japan,
	\printead{e1}}

\address[B]{Georg-August-Universit\" at  G\" ottingen, Germany,
	\printead{e2}}

\address[C]{Ruhr-Universit\" at Bochum
	\printead{e3}}

\address[D]{University of Toronto, Canada
	\printead{e4}}

\address[E]{Ruhr-Universit\" at Bochum, Germany
	\printead{e5}}

\address[F]{Universit\' e libre de Bruxelles, Belgium
	\printead{e6}}


\end{aug}

\begin{abstract}
Frequency domain methods form a ubiquitous part of the statistical toolbox for time series analysis. In recent years, considerable interest has been given to the development of new spectral methodology and tools capturing dynamics in the entire joint distributions and thus avoiding the limitations of classical, $L^2$-based spectral methods. Most of the spectral concepts proposed in that literature suffer from one major drawback, though: their estimation requires the choice of a smoothing parameter, which has a considerable impact on estimation quality and poses challenges for statistical inference. In this paper, associated with the concept of copula-based spectrum, we introduce the notion of \textit{copula spectral distribution function} or \textit{integrated copula spectrum}. This integrated copula spectrum retains the advantages of copula-based spectra but can be estimated without the need for smoothing parameters. We provide such estimators, along with a thorough theoretical analysis, based on a functional central limit theorem, of their asymptotic properties. We leverage these results to test various hypotheses that cannot be addressed by classical spectral methods, such as the lack of time-reversibility or asymmetry in tail dynamics. 
\end{abstract}

\begin{keyword}[class=MSC]
	\kwd[Primary ]{62M15}
	\kwd{62G30}
	\kwd[; secondary ]{62G10}
\end{keyword}

\begin{keyword}
	\kwd{Copula}
	\kwd{Ranks}
	\kwd{Time series}
	\kwd{Frequency domain}
	\kwd{Time reversibility}
\end{keyword}

\end{frontmatter}

\section{Introduction}

Spectral methods always have been central in the analysis of time series and  remain  (see \cite{vSachs20} for a recent review) a very active domain of methodological and applied statistical research. Their applications are without number, ranging from econometrics and finance (with classical monographs such as \cite{Granger15}) to geophysics \citep{geoph11},   fluid mechanics \citep{fluid19}, environmetrics, and climate change \citep{climate02}. 

Powerful as they are, classical spectral methods, however, suffer from the significant limitations inherited from their  $L^2$ nature: being covariance-based,  
 they fail to capture important distributional features such as dependence without correlation (as typically observed in financial returns), time-irreversibility,   asymmetric dependence between high and low quantile values, or higher-order dynamics. This has motivated, in the past decades, a rich strand of literature   replacing covariances with  alternative measures of dependence   related to joint distributions, copulas, and characteristic functions. Pioneering contributions in this direction were made by \cite{Hong1999}, who proposes a generalized characteristic function-based concept of  spectral density. In the specific problem of testing pairwise independence (rather than pairwise non-correlation), \cite{hong2000} introduces
a test statistic based on spectra derived from joint distribution functions and copulas at different lags. More recent contributions introduce the notions of \textit{Lap\-lace, quantile-based}, and \textit{copula spectral densities} and  \textit{spectral density kernels}, involving  various quantile-related spectral concepts, along with the
corresponding sample-based (smoothed) periodograms.
That strand of literature includes \cite{Li2008,Li2012,Li2013},
\cite{Hagemann2013}, \cite{DetteEtAl2013, DetteEtAl2016} and \cite{LeeRao2012}. Extensions to 
locally stationary  and multivariate time series are considered in \cite{birr2014} and \cite{BarunikKley2019}, respectively. An analysis of related concepts under long-range dependence can be found in \cite{lim2021quantile}. The utility of quantile and copula spectra for model building and model assessment is demonstrated in \cite{BirrEtAl2019} and \cite{li2021}, while an application of quantile- and copula-based spectral techniques to the analysis of cryptocurrency returns can be found in \cite{su2021quantile}. An extension of these concepts to the analysis of extreme events, which is related in spirit but different in many other respects, was considered by \cite{DavisMikoschZhao2013}. Finally, in the time domain, \cite{linwha2007}, \cite{DavisMikosch2009}, and \cite{hlow2014} introduced the related concepts of \textit{quantilograms} and \textit{extremograms}.

Unfortunately, despite many attractive properties,   spectral densities---whether traditional~$L^2$ or generalized---in practice suffer from several drawbacks; among them,  the need to choose a smoothing parameter to ensure consistent estimation and a lack of process convergence of the resulting estimators when indexed by frequencies. The latter makes it challenging to use them for inferential purposes such as testing for specific time series features.    

In the classical $L^2$ world, this drawback has motivated the  recourse to \textit{$L^2$ spectral distribution functions $\specdis$} resulting from the integration of the spectral density over frequencies. In contrast to spectral densities, such integrated spectra can be estimated without the need for smoothing. Estimation of $\specdis$ along with process convergence of the resulting estimators under increasingly general conditions was discussed in \cite{grenander57}, \cite{ibragimov63}, \cite{brillinger69}, \cite{dahlhaus1985AsymptoticNormality}, and \cite{anderson93} among others. Applications of this process convergence to various testing problems 
 are provided in \cite{priestley87}, Section 6.2.6 and \cite{anderson93}. An extension to related processes indexed by more general classes of functions is considered in \cite{dahlhaus88, mikosch1997}. Integrated versions of certain normalized periodograms were also studied in \cite{kluppelberg1996} under various tail assumptions (including the infinite-variance case) on the underlying time series and extended to long-memory processes in \cite{kokoszka1997integrated}. 

The aim of the present paper is to combine the attractive features of copula--based spectra with the theoretical merits of spectral distributions. To this end, we define the \textit{copula spectral distribution function}, which arises from integrating copula spectral densities over frequencies. We provide estimators which are based on partial sums of copula periodograms and do not require the choice of smoothing parameters. 

The remaining paper is organized as follows. Copula spectral distribution functions are formally defined in Section~\ref{sec:defest} where their estimation is also discussed. Weak convergence (as  stochastic processes) of the estimators from Section~\ref{sec:defest} is established in Section~\ref{sec:th}. Section~\ref{sec:inf} shows how this process convergence can be combined with sub--sampling to construct uniform confidence bands for integrated copula spectra and test various hypotheses about the underlying time series. Section~\ref{sec:sim} demonstrates the finite-sample properties of the methodology from Section~\ref{sec:inf} in an extensive simulation study. All proofs and additional simulation results are deferred to a series of Appendices.

\section{Integrated copula spectra -- definition and estimation}\label{sec:defest}

In what follows, let $\proc$ denote a strictly stationary real-valued time series. Denote by $F$ the marginal distribution function \label{sym:margdistX} of $X_0$ and by $\tau\mapsto q_{\tau}=F^{-1}(\tau) := \inf\{x \in \R: \tau \leq F(x)\},\,\tau\in(0,1)$ the corresponding quantile function. As argued in \cite{DetteEtAl2013, DetteEtAl2016}, a natural way to capture the nonlinear dynamics of   $\proc$ is the analysis of its \textit{copula spectral density} 
\als{\label{copulaspecdensintroduction}
\mathfrak{f}(\omega; \tau_1,\tau_2):=\frac{1}{2\pi}\sum_{k \in \Z}\gamma_{k}^{U}(\tau_1,\tau_2)e^{-i\omega k},\qquad \omega\in\R,\,(\tau_1,\tau_2)\in (0,1)^2
}
where  $U_t := F(X_t)$, 
\[
\gamma_{k}^{U}(\tau_1,\tau_2) := {\text{\rm Cov}}(I\{U_k \leq \tau_1\}, I\{U_0 \leq \tau_2\}) = C_k(\tau_1,\tau_2) - \tau_1\tau_2,
\]
and  $C_k$ denotes the copula of the random vector $(X_k,X_0)$; here $I\{A\}$ denotes the indicator function of $A$. To ensure the existence of $\mathfrak{f}$, it suffices to assume that the $\gamma_{k}^{U}(\tau_1,\tau_2)$ are absolutely summable over $k \in \Z$ for each pair $(\tau_1,\tau_2)$, which we throughout  implicitly assume. As shown in \cite{DetteEtAl2013, DetteEtAl2016, BirrEtAl2019}, copula spectral densities enjoy many attractive properties; see also \cite{Li2013, li2021} for similar findings in the setting of Laplace spectra. They exist without any moment assumptions, are invariant under strictly increasing marginal transformations (hence  are   scale--free), and provide a complete characterization of the pairwise copulas---hence the pairwise dependencies---of the  series at arbitrary lags. The last point is in stark contrast to classical spectral densities which are unable to capture many important properties of time series such as lack of time-reversibility, conditional heteroscedasticity, or asymmetry between upper- and lower-tail  dynamics.   

Yet, despite their flexibility, copula spectral densities are sharing with the traditional ones an important practical drawback: the choice of a smoothing parameter is required to obtain consistent estimators. Selecting this smoothing parameter is difficult in practice and poses substantial challenges for inference. Indeed, larger bandwidths lead to smaller variance but larger (asymptotic) bias and the exact amount of bias depends on unknown smoothness properties of the underlying copula spectral density. The need for local smoothing also leads to difficulties in obtaining results that hold uniformly in frequencies (more formally, no process convergence is possible). This poses a major roadblock for subsequent inference procedures. We note that those drawbacks are not limited to copula spectral densities but also appear  in the estimation of classical, $L^2$--based spectral densities.

Motivated by the above discussion, we propose to consider \textit{copula spectral distribution functions} which are defined as  
\als{\label{specdisestimatorintro}
\specdis(\lambda ;\tau_1,\tau_2):=\int_0^{\lambda}\spec(\omega; \tau_1,\tau_2)\dd\omega , \qquad\lambda\in [0,\pi].
}
Copula spectral distributions inherit the virtues of copula spectral densities and are conveying the same information; at the same time, their estimation (as discussed below) does not involve the choice of smoothing parameters, and process convergence can be established in   quantile levels and frequencies simultaneously (see Theorem~\ref{weakconvintegspectrum} below). 

Before proceeding to estimation, let us  provide two examples of hypotheses about time series dynamics that can be conveniently formulated and tested through the use of spectral distribution functions. 

\begin{example} \label{ex:tr} Testing for time-reversibility. \rm A strictly stationary process $(X_t)_{t \in \Z}$ is called pairwise time-reversible at lag $k$ iff $(X_0,X_{k}) \stackrel{d}{=} (X_0,X_{-k})$. A process is pairwise time-reversible if it is time-reversible for all lags $k \geq 1$. Determining if data can be modeled as  a time-reversible process has important consequences for subsequent modeling:  testing for time-reversibility therefore has attracted substantial interest in the literature---see \cite{br67} for an early contribution, and chapter 8 in \cite{goojier2017} for an overview. Copula spectral distribution functions provide a natural way of assessing time-reversibility since a process is pairwise time-reversible if and only if the imaginary part of the corresponding spectral distribution function is uniformly zero:
	\[
	\Im \specdis(\lambda ;\tau_1,\tau_2) = 0\qquad \text{for all $\lambda$, $\tau_1$, and $\tau_2$.}
	\]
	We will leverage this property of spectral distribution functions in Section~\ref{sec:timerev} to construct a test that has power against the lack of  (pairwise)  time-reversibility at specified or unspecified  lag.
\end{example}

\begin{example}\label{ex:sym}
	Assessing symmetry of tail dynamics. \rm It is well known that financial time series   exhibit asymmetric dependence structures in left- and right-hand tails, respectively---see  \cite{jr2003}, \cite{li2021},  among many others. Copula spectral distributions provide a natural model-free way to access this kind of asymmetry in tail dynamics. From a distributional perspective, asymmetry in tail dynamics corresponds to asymmetry in the lag-$k$ copula~$C_k$ of~$(X_0,X_k)$ for some lag $k$: if $$C_k(\tau_1,\tau_2) -\tau_1\tau_2 \neq C_k(1-\tau_1,1-\tau_2) -(1-\tau_1)(1-\tau_2)$$ for small values of~$\tau_1,\tau_2$, then the tail behavior of $(X_t,X_{t+k})$ is asymmetric. Copula spectral distributions provide a natural way of assessing this type of asymmetry since $$\specdis(\lambda ;\tau_1,\tau_2) = \specdis(\lambda ;1-\tau_1,1-\tau_2)$$ for all $\lambda$ is equivalent to $$C_k(\tau_1,\tau_2) -\tau_1\tau_2 = C_k(1-\tau_1,1-\tau_2) -(1-\tau_1)(1-\tau_2)$$ for all $k$. A more formal discussion of the corresponding null hypothesis and testing procedure is provided in Section~\ref{sec:symmetry}     
\end{example}  

We next discuss estimation. Recall  the definition \citep{DetteEtAl2016}  of the  \textit{copula rank periodogram} (in short, the CR periodogram): 
\als{\label{copperintro}
	\copper_{n,R}^{\tau_1,\tau_2}(\omega):=\frac{1}{2\pi n}d_{n,R}^{\tau_1}(\omega)d_{n,R}^{\tau_2}(-\omega),\qquad \omega\in\R,(\tau_1,\tau_2)\in[0,1]^2
}
with
\als{\label{dnRtaudef}
	d_{n,R}^{\tau}(\omega):=\sum_{t=0}^{n-1}I\{\hat{F}_n(X_t)\leq\tau\}e^{-i\omega t}\text{ and }
	\hat{F}_n(x):=\frac{1}{n}\sum_{t=0}^{n-1}I\{X_t\leq x\}.
}
As shown in \cite{DetteEtAl2016}, the vector $(\copper_{n,R}^{\tau_1,\tau_2}(\omega_1),\dots,\copper_{n,R}^{\tau_1,\tau_2}(\omega_K))$ is approximately multivariate complex normal with expected values $\mathfrak{f}(\omega_1; \tau_1,\tau_2),\dots,\mathfrak{f}(\omega_K; \tau_1,\tau_2)$ and independent entries; see Proposition 3.4 in there for a formal statement. This motivates, for the copula spectral distribution function,  the  estimator 
\als{\label{CRcopulaperiodogram}
	\widehat{\specdis}\phantom{F\!\!\!\!\!}_{n,R}(\lambda ;\tau_1,\tau_2):={}&\frac{2\pi}{n}\sum_{s=1}^{n-1}I\big\{0\leq \frac{2\pi s}{n}\leq\lambda\big\}\copper_{n,R}^{\tau_1,\tau_2}\big(\frac{2\pi s}{n}\big),\qquad\lambda\in [0,\pi].
}
Observe that, in contrast to the copula spectral density estimators considered in~\cite{DetteEtAl2016}, no smoothing parameter is required. In addition, as we shall show in Section~\ref{sec:th}, this estimator converges as a process in all three arguments when properly centered and scaled. This makes it a very attractive choice for testing various hypotheses about distributional dynamics of the underlying time series.

\section{Asymptotic theory}\label{sec:th}


This section is devoted to proving process convergence of the estimator $\widehat{\specdis}\phantom{F\!\!\!\!\!}_{n,R}$ after proper centering and scaling. We begin by stating the main technical conditions which are needed to establish this result. 

\begin{ass}\label{assumptionsspectraldist}
\rm{
\bi
\item[(S)] The real-valued process $\proc$ is strictly stationary;  the marginal distribution  $F$ of $X_0$ is  continuous. 
\item[(C)] There exist constants $\rho\in(0,1)$ and $K<\infty$ such that, for arbitrary intervals $A_1,\dots,A_p$ of $\R$ and arbitrary $t_1,\dots,t_p\in\Z$, 
\begin{equation}\label{assC}
\Big\vert\cum\left(I\{X_{t_1}\in A_1\},\dots,I\{X_{t_p}\in A_p\}\right)\Big\vert\leq K\rho^{\max_{i,j}|t_i-t_j|}.
\end{equation}
\item[(D)] The partial derivatives of the function 
\begin{equation}\label{specdiswithoutzero}
 (\tau_1,\tau_2)\mapsto 
\specdisred(\lambda ;\tau_1,\tau_2):=\frac{1}{2\pi}\sum_{k\in\Z \backslash \{0\}}\gamma_k^U(\tau_1,\tau_2)\frac{i}{k}\Big(e^{-ik\lambda}-1\Big)
\end{equation}
  exist and are continuous for $(\lambda ;\tau_1,\tau_2)\in[0,\pi]\times(0,1)\times (0,1)$. 
\ei
}
\end{ass}

\begin{rem}[Discussion of assumptions] \rm{
Assumption (C) places restrictions on the strength of time dependence in $(X_t)_{t\in \Z}$. This assumption also appears in the asymptotic analysis by \cite{DetteEtAl2016}  of copula spectral densities. In particular, \cite{DetteEtAl2016} show that \eqref{assC}  is implied by several standard assumptions such as exponential $\alpha$- and $\beta$-mixing. The same reference also shows that processes satisfying some geometric moment contraction properties defined in \cite{wu04} fulfill Assumption (C).  
 
Condition (D) is needed to quantify the effect of estimating the marginal cdf $F$ by its empirical version $\hat F_n$. The derivatives of $\specdisred$ also appear in the covariance kernel of the limiting process.}
\end{rem}

In the following Lemma we show that Assumption (D) is satisfied for  strictly stationary centered Gaussian processes $\proc$ with absolutely summable pairwise copula cumulants. The details of the proof are deferred to Section~\ref{sec::proofsintspec}.

\begin{lemma}\label{exampleGaussianprocess}
Let $\proc$ be a stationary centered Gaussian process with auto--covariances $\rho_k$ where $\rho_k \in (-1,1)$ for $k \neq 0$ and $\sum_{k \geq 1}{|\rho_k|}/{k}<\infty$. Then the partial derivatives of the function $(\tau_1,\tau_2)\mapsto \specdisred(\lambda ;\tau_1,\tau_2)$   exist and are continuous on the set~$\{(\lambda ;\tau_1,\tau_2)\in[0,\pi]\times[\eta,1-\eta]\times [\eta,1-\eta]\}$.
\end{lemma}

In order to state our main result we need some additional notation. 
Define the \textit{copula spectral density of order $K$} as
\al{
	\spec(\omega_1,\dots,\omega_{K-1}; \tau_1,\dots,\tau_K):=(2\pi)^{-K+1}\!\!\!\sum_{k_1,\dots,k_{K-1}=-\infty}^{\infty}\!\!\gamma_{k_1,\dots,k_{K-1}}^{U}(\tau_1,\dots,\tau_K)e^{-i\sum_{j=1}^{K-1}k_j\omega_j}
}
with the copula cumulant function of order $K$ 
\al{
	\gamma_{k_1,\dots,k_{K-1}}^{U}(\tau_1,\dots,\tau_K):=\cum(I\{U_{k_1}\leq \tau_1\},\dots,I\{U_{k_{K-1}}\leq \tau_{K-1}\},I\{U_{0}\leq\tau_k\})
}
for $k_1,\dots,k_{K-1}\in\Z$. We are now ready to state our main result---process convergence of the properly centered and scaled estimator $\widehat{\specdis}\phantom{F\!\!\!\!\!}_{n,R}$. Applications of this result to inference will be discussed in the following sections.

\begin{theorem}\label{weakconvintegspectrum}
Let Assumptions \ref{assumptionsspectraldist} hold. Then, for any $0<\eta<\frac{1}{2}$, the process\label{sym:specdistprocrank}
\als{\label{mainprocess}
\mathbb{G}_{n,R}(\lambda ;\tau_1,\tau_2):=\sqrt{n}\Big(\widehat{\specdis}\phantom{F\!\!\!\!\!}_{n,R}(\lambda ;\tau_1,\tau_2)-\specdis(\lambda ;\tau_1,\tau_2)\Big)_{(\lambda ;\tau_1,\tau_2)\in [0,\pi]\times[\eta,1-\eta]^2}
}
converges weakly  to the centered Gaussian process $(\mathbb{G}(\lambda ;\tau_1,\tau_2))_{(\lambda ;\tau_1,\tau_2)\in \big([0,\pi]\times[\eta,1-\eta]\times[\eta,1-\eta]\big)}$  with covariance structure 
\als{\label{covasympintspec}
&\Cov\Big(\mathbb{G}(\lambda_1 ;\tau_1,\tau_2),\mathbb{G}(\lambda_2;\kappa_1,\kappa_2)\Big) \notag
\\
=~&2\pi\int_{0}^{\lambda_1\wedge\lambda_2}\spec(\alpha; \tau_1,\kappa_1)\spec(-\alpha; \tau_2,\kappa_2)\dd\alpha\notag
+2\pi\int_0^{\lambda_1}\int_0^{\lambda_2}\spec(\alpha,-\alpha,-\beta; \tau_1, \tau_2, \kappa_1, \kappa_2)         \dd\alpha \dd\beta\notag\\
&+\sum_{j=1}^2\frac{\partial \specdisred}{\partial \kappa_j}(\lambda_2;\kappa_1,\kappa_2)2\pi\int_0^{\lambda_1}\spec(\alpha,-\alpha ;  \tau_1,\tau_2,\kappa_j)\dd\alpha\notag
\\
&+\frac{\partial \specdisred}{\partial \tau_j}(\lambda_1 ;\tau_1,\tau_2)2\pi\int_0^{\lambda_2}\spec(\alpha,-\alpha ;  \kappa_1,\kappa_2,\tau_j)\dd\alpha\notag\\
&+\sum_{j=1}^2\sum_{k=1}^2\frac{\partial \specdisred}{\partial \tau_j}(\lambda_1 ;\tau_1,\tau_2)\frac{\partial \specdisred}{\partial \tau_k}(\lambda_1 ;\tau_1,\tau_2)2\pi\spec(0; \tau_j, \tau_k),
}
that is, $(\mathbb{G}_{n,R}(\cdot,\cdot,\cdot))\rightsquigarrow (\mathbb{G}(\cdot,\cdot,\cdot))
$ 
where $\rightsquigarrow$ denotes  weak convergence, as $n\to\infty$,  with respect to the uniform metric in the space $\ell^{\infty}_{\C}\big([0,\pi]\times[\eta,1-\eta]\times[\eta,1-\eta]\big)$. 
Moreover, the paths of the process $\mathbb{G}_{n,R}$ are asymptotically uniformly equicontinuous with respect to any norm on $\R^3$.
\end{theorem}

Let us briefly compare this result with related results in the literature. Similarly to estimators of $L^2$ spectral distribution functions, we obtain process convergence in $\lambda$ with a $n^{-1/2}$ convergence rate. However, in contrast to the results in that literature, we have two additional parameters~$(\tau_1,\tau_2)$ and we also obtain process convergence in these, which calls for completely different  proofs. 

Spectral distribution functions without marginal normalization are considered in \cite{hong2000}. The latter author establishes process convergence in $\lambda$ and two parameters which play a similar role as our quantile levels assuming that the time series is a collection of i.i.d. data. This considerably simplifies the entire analysis and the proof technique used there does not extend to the case of general serial dependence. In addition, our analysis differs since we consider marginal normalization by estimating the marginal distribution function, something which is not covered by the results of \cite{hong2000}, even in the special case of i.i.d.\ observations.

Finally, we provide a comparison with  corresponding results for the estimation of copula spectral densities as discussed in \cite{DetteEtAl2016}. 
 There are several key differences in the form of the final result and  the resulting theoretical analysis. First, observe that Theorem~\ref{weakconvintegspectrum} provides process convergence of the integrated copula spectral densities in the quantile levels $\tau_1,\tau_2$ as well as the frequencies $\lambda$. This is in contrast to the copula spectral densities (\ref{copulaspecdensintroduction}) considered in \cite{DetteEtAl2016} where only process convergence in the quantile levels is obtained. This  is the case also for  autocovariance-based spectral densities---due to the fact that the limiting processes,  for distinct frequencies,  
  are mutually  independent, so that no tight element with the right finite-dimensional distributions exists in $\ell^{\infty}_{\C}\big([0,\pi]\times[0,1]^{2}\big)$ [see Remark 3.5 in \cite{DetteEtAl2016}]. 
Second, we obtain an $n^{-1/2}$ convergence rate,  which is strictly faster than the rates obtained in \cite{DetteEtAl2016} for any permissible bandwidth choice. This is due to the need for local smoothing when estimating copula spectral densities, and similar phenomena also occur in the context of $L^2$ spectra and  ``classical'' kernel density estimation. Third, as discussed in more detail in  Remark~\ref{Rem34}, the limiting covariance in Theorem~\ref{weakconvintegspectrum} has several terms that  are due to empirical normalization of the margins. Such terms do not appear in the limiting process when estimating copula spectral densities because  the effect of marginal standardization there is negligible relative to the convergence rate of the estimator with known margins. The fact that we need to account for such terms in our limit considerably complicates our asymptotic analysis compared to the developments in \cite{DetteEtAl2016}.

\begin{rem}[A sketch of the proof]\label{Rem34}
{\rm 
The proof of Theorem~\ref{weakconvintegspectrum} is long and technical;  deferring details to the online supplement,  we only outline here the main  steps. 

 (a) A key ingredient is the weak convergence of  the process 
\begin{multline*}
\Big(\mathbb{G}_{n,U}(\lambda ;\tau_1,\tau_2)\Big)_{(\lambda ;\tau_1,\tau_2)\in [0,\pi]\times[\eta,1-\eta]^2}
\\
:=\sqrt{n}\Big(\widehat{\specdis}\phantom{F\!\!\!\!\!}_{n,U}(\lambda ;\tau_1,\tau_2)-\specdis(\lambda ;\tau_1,\tau_2)\Big)_{(\lambda ;\tau_1,\tau_2)\in [0,\pi]\times[\eta,1-\eta]^2}
\end{multline*}
where $\widehat{\specdis}_{n,U}$ denotes the (infeasible) oracle estimator where the empirical distribution function $\hat F_n$ in $\widehat{\specdis}_{n,R}$ is replaced by $F$. We show that this process converges, in 
$\ell^{\infty}_{\C}\big([0,\pi]\times[\eta,1-\eta]\times[\eta,1-\eta]\big)$, to a centered Gaussian process~$\big(\mathbb{G}_U(\lambda ;\tau_1,\tau_2)\big)_{(\lambda ;\tau_1,\tau_2)\in [0,\pi]\times[\eta,1-\eta]^2}$ with covariance structure
\al{
\Cov\Big(\mathbb{G}_U(\lambda_1 ;\tau_1,\tau_2),\mathbb{G}_U(\lambda_2;\kappa_1,\kappa_2)\Big)={}&2\pi\int_{0}^{\lambda_1\wedge\lambda_2}\spec(\alpha; \tau_1,\kappa_1)\spec(-\alpha; \tau_2,\kappa_2)\dd\alpha\notag\\
&\hspace{0cm}+2\pi\int_0^{\lambda_1}\int_0^{\lambda_2}\spec(\alpha,-\alpha,-\beta; \tau_1, \tau_2, \kappa_1, \kappa_2) \dd\alpha \dd\beta.
}

 (b) Utilizing uniform asymptotic equicontinuity in probability of $(\tau_1,\tau_2)\mapsto \mathbb{G}_{n,U}(\lambda ;\tau_1,\tau_2)$   along with a Taylor expansion of the spectral distribution function ${\specdis}$, we obtain the stochastic representation
\[
\mathbb{G}_{n,R}(\lambda ;\tau_1,\tau_2) = \mathbb{G}_{n,U}(\lambda ;\tau_1,\tau_2) + \sqrt{n} \sum_{j=1}^2 (\tau_j - \hat F_{n}(F^{-1}(\tau_j)))\frac{\partial \specdisred}{\partial \tau_j}(\lambda ;\tau_1,\tau_2) + o_P(1) 
\]   
as $n\to\infty$, where the remainder is uniform in $(\lambda, \tau_1, \tau_2)$. 

  (c) The remaining part of the  proof is devoted to establishing process convergence of the leading term in this representation. The sum~$\sqrt{n} \sum_{j=1}^2 (\tau_j - \hat F_{n}(F^{-1}(\tau_j)))\frac{\partial \specdisred}{\partial \tau_j}(\lambda ;\tau_1,\tau_2)$ captures the impact of estimating the marginal distribution function $F$ by its empirical counterpart. This expression also explains the additional terms in the covariance function of $\mathbb{G}$ when compared to that of $\mathbb{G}_{U}$. Such additional terms also appear in the limiting distribution of empirical copula processes [see, for instance, \cite{fermanian04} or \cite{segers12}]. However, they do not appear in the estimation of copula spectra in \cite{DetteEtAl2016} because   the convergence speed of the estimator  there is strictly slower than $n^{-1/2}$.
}
\end{rem}

\section{Subsampling-based inference } \label{sec:inf}


Theorem~\ref{weakconvintegspectrum} is a very powerful instrument  allowing us to perform copula spectral analysis in a broad range of practical problems. Deriving valid procedures for inference, however,  crucially depends on the limit process $\mathbb{G}$ in Theorem~\ref{weakconvintegspectrum}---that is, on  the covariance kernel defined in~\eqref{covasympintspec}. This covariance kernel in turn depends on second-, third-, and fourth-order copula spectra and  some partial derivatives of the function~$\specdisred$ defined in~\eqref{specdiswithoutzero}. While for some testing problems (e.g., under the null hypothesis of serial independence: cf. \cite{hong2000}) these quantities simplify substantially, they are quite difficult to estimate in general. In this section, we demonstrate how subsampling methods  \citep{PolitisEtAl1999} yield  feasible and asymptotically valid confidence bands and tests for time-reversibility [Example~\ref{ex:tr}] and asymmetry of tail dynamics [Example~\ref{ex:sym}]. 

A key quantity in all subsampling procedures described in this section is  the   estimator 
\begin{equation}\label{eq:hatFnb}
\widehat{\specdis}\phantom{F\!\!\!\!\!}_{n,b,t,R}(\lambda ;\tau_1,\tau_2)
:= \frac{2\pi}{b} \sum_{j=1}^{b-1} I\big\{0\leq \frac{2\pi j}{b}\leq\lambda\big\} \copper_{n,b,t,R}^{\tau_1,\tau_2}\Big( \frac{2\pi j}{b} \Big),
\end{equation}
 of $\specdis$ computed from the subsample $X_t, \ldots, X_{t+b-1}$, where
\als{\label{copperintro2}
	\copper_{n,b,t,R}^{\tau_1,\tau_2}(\omega):=\frac{1}{2\pi b}d_{n,b,t,R}^{\tau_1}(\omega)d_{n,b,t,R}^{\tau_2}(-\omega),\qquad \omega\in\R,(\tau_1,\tau_2)\in[0,1]^2
}
with
\als{\label{dnRtaudef2}
	d_{n,b,t,R}^{\tau}(\omega):=\sum_{j=0}^{b-1}I\{\hat{F}_{n,b,t}(X_{t+j})\leq\tau\}e^{-i\omega j} \text{ and }
	\widehat F_{n,b,t}(x) := \frac{1}{b} \sum_{i=t}^{t+b-1} I\{X_{i} \leq x\}.
}
The block length $b$ is an integer between 1 and $n$; for our asymptotic results to hold, we will choose it such that $b \rightarrow \infty$ and $b = o(n)$ as $n \rightarrow \infty$. 

\subsection{Constructing uniform confidence bands}\label{sec:subsampl}

We now describe how asymptotically valid confidence bands can be obtained via subsampling. We will consider two types of confidence bands: (a) bands that are uniform in $\lambda$ for fixed quantile levels $\tau_1,\tau_2$ and (b) bands that are uniform in all three arguments $\lambda, \tau_1, \tau_2$.

By Theorem~\ref{weakconvintegspectrum} and the Continuous Mapping Theorem, 
\begin{equation*}
\sqrt{n} D_{n}(\tau_1, \tau_2) := \sqrt{n} \max_{\lambda \in [0,\pi]} \Big| \Re \widehat{\specdis}\phantom{F\!\!\!\!\!}_{n,R}(\lambda ;\tau_1,\tau_2) - \Re \specdis(\lambda ;\tau_1,\tau_2)  \Big| \\
\rightsquigarrow \max_{\lambda \in [0,\pi]} \Big| \Re \mathbb{G}(\lambda ;\tau_1,\tau_2) \Big|,
\end{equation*}
in $\ell^{\infty}([0,1]^2)$, as $n \rightarrow \infty$. Further, for any 
 continuous weight function $s: [\eta, 1-\eta]^2 \rightarrow \IR_+$ that is bounded away from $0$, we have
\begin{equation*}
\sqrt{n} E_{n} := \sqrt{n} \max_{(\tau_1,\tau_2) \in [\eta,1-\eta]^2} \frac{D_{n}(\tau_1, \tau_2)}{s(\tau_1, \tau_2)} \rightsquigarrow \max_{(\lambda ;\tau_1,\tau_2)\in [0,\pi]\times[\eta,1-\eta]^2} \Big| \frac{\Re \mathbb{G}(\lambda ;\tau_1,\tau_2)}{s(\tau_1, \tau_2)} \Big|,
\end{equation*}
in distribution, as $n \rightarrow \infty$.

\medskip

For the construction of an asymptotically valid $(1 - \alpha)$-confidence band for $\Re \specdis(\lambda ;\tau_1,\tau_2)$, it is sensible to proceed as follows.
We require 
\[I_{\alpha} := [\Re\widehat{\specdis}\phantom{F\!\!\!\!\!}_{n,R}(\lambda ;\tau_1,\tau_2) - \Delta(\lambda ;\tau_1,\tau_2), \Re\widehat{\specdis}\phantom{F\!\!\!\!\!}_{n,R}(\lambda ;\tau_1,\tau_2) + \Delta(\lambda ;\tau_1,\tau_2)]\]
to satisfy
\begin{equation*}
\begin{split}
\liminf_n \IP\big( \Re \specdis(\lambda ;\tau_1,\tau_2) \in I_{\alpha} \big) \geq 1 - \alpha.
\end{split}
\end{equation*}

For a uniform-in-$\lambda$ confidence band    for fixed $(\tau_1, \tau_2)$, choose $\Delta(\lambda ;\tau_1,\tau_2) \equiv C_D$ and, for a uniform-in-$(\lambda, \tau_1, \tau_2)$ confidence band, choose $\Delta(\lambda ;\tau_1,\tau_2) \equiv C_E \cdot s(\tau_1, \tau_2)$. The use of the weighting function $s$ improves the uniform confidence intervals by allowing the width to depend on $(\tau_1,\tau_2)$; cf.\  \cite{NeumannPaparoditis2008}.
These confidence bands are (asymptotically) valid if $C_D$ and $C_E$ are the $(1-\alpha)$ quantiles of the (limit) distributions of $D_n(\tau_1, \tau_2)$ and $E_n$, respectively. 
In practice, neither these distributions nor their limits are analytically tractable and we therefore propose the following subsampling-based intervals.
	
The $(1 - \alpha)$-confidence band that is uniform in $\lambda$ for fixed $(\tau_1,\tau_2)$ is defined by
	\begin{equation}\label{def:I_D_Re}
	\widehat I_{\alpha,{\rm Re}}^D(\lambda,\tau_1,\tau_2) := \left[\Re\widehat{\specdis}\phantom{F\!\!\!\!\!}_{n,R}(\lambda,\tau_1,\tau_2) - C_{D,\alpha}(\tau_1,\tau_2) , \Re\widehat{\specdis}\phantom{F\!\!\!\!\!}_{n,R}(\lambda,\tau_1,\tau_2) + C_{D,\alpha}(\tau_1,\tau_2) \right],
	\end{equation}
	where
	\begin{equation*}
	C_{D,\alpha}(\tau_1,\tau_2) := (1/n)^{1/2} \inf\{x : L_{n,b}^D(x,\tau_1,\tau_2) \geq 1-\alpha\}
	\end{equation*}
	with
	\begin{align}\label{eqfpc}
	L^D_{n,b}(x) :=& \frac{1}{n-b+1} \sum_{t=1}^{n-b+1} I\{ \sqrt b \tilde D_{n,b,t}(\tau_1, \tau_2) \leq x\}\text{\ and \ }
	\\
	\tilde D_{n,b,t}(\tau_1, \tau_2) :=& (1 - b/n)^{-1/2} \max_{\ell = 0,1,\ldots,\lfloor d/2 \rfloor} \left| \Re \widehat{\specdis}\phantom{F\!\!\!\!\!}_{n,b,t,R}\Big(\frac{2 \pi \ell}{d},\tau_1,\tau_2 \Big) - \Re \widehat{\specdis}\phantom{F\!\!\!\!\!}_{n,R} \Big(\frac{2 \pi \ell}{d},\tau_1,\tau_2 \Big)  \right|. \label{eqfpc}
	\end{align}
	
	Note that $C_{D,\alpha}(\tau_1,\tau_2)$ is the empirical $(1-\alpha)$-quantile of 
	$$\{\tilde D_{n,b,t}(\tau_1, \tau_2), t=1,\ldots, n-b+1\},$$ scaled by a factor  $(b/n)^{1/2}$.
	Intuitively, the proposed interval will be asymptotically valid, be\-cause the distributions of~$\sqrt n D_n(\tau_1, \tau_2)$ and $\sqrt b \tilde D_{n,b,t}(\tau_1, \tau_2)$ converge to the same limit and the distribution of~$\sqrt b \tilde D_{n,b,t}(\tau_1, \tau_2)$ is well approximated by the empirical distribution~$L^D_{n,b}$. 
	
	The factor $(1 - b/n)^{-1/2}$ in \eqref{eqfpc} is an optional finite-population correction and 
	can be replaced by any sequence  converging to one. Such correction is recommended by \citet
	{PolitisEtAl1999};
	 our simulations in Section~\ref{sec:sim} below indicate that it is indeed quite advisable in this context. 
	As for $d$,   a positive integer, it is typically chosen such that
	$$\{0, 1/d, \ldots, \lfloor d/2 \rfloor/d\} \subseteq \{0, 1/b, \ldots, \lfloor b/2 \rfloor/b\},$$
	 which facilitates the evaluation of the estimates.
	
	\medskip
	
	Similarly define the uniform-in-$(\lambda,\tau_1,\tau_2)$ $(1 - \alpha)$-confidence band   as 	\begin{equation}\label{def:I_E_Re}
	\widehat I_{\alpha,{\rm Re}}^E(\lambda,\tau_1,\tau_2) := [\Re\widehat{\specdis}\phantom{F\!\!\!\!\!}_{n,R}(\lambda,\tau_1,\tau_2) - C_{E,\alpha}s(\tau_1, \tau_2) , \Re\widehat{\specdis}\phantom{F\!\!\!\!\!}_{n,R}(\lambda,\tau_1,\tau_2) + C_{E,\alpha} s(\tau_1, \tau_2)],
	\end{equation}
	where
	\begin{equation*}
	C_{E,\alpha} := (1/n)^{1/2} \inf\{x : L_{n,b}^E(x) \geq 1-\alpha\}
	\end{equation*}
	with
	\begin{align*}
	L^E_{n,b}(x) :=& \frac{1}{n-b+1} \sum_{t=1}^{n-b+1} I\{ \sqrt b\tilde E_{n,b,t} \leq x\}
	\text{ and }
	\tilde E_{n,b,t} := \max_{(\tau_1,\tau_2) \in S_n } \frac{\tilde D_{n,b,t}(\tau_1, \tau_2)}{s(\tau_1, \tau_2)}
	\end{align*}
	where $S_n$, the role of which will be made clear in the sequel, is a sequence of finite subsets of the interval~$[\eta,1-\eta]^2$.

	Uniform confidence intervals $\widehat I_{\alpha,{\rm Im}}^D(\lambda,\tau_1,\tau_2)$ and $\widehat I_{\alpha,{\rm Im}}^E(\lambda,\tau_1,\tau_2)$ for the imaginary parts~$\Im \specdis(\lambda ;\tau_1,\tau_2)$ are defined in the same way,  with real parts replaced by imaginary parts.

We now state a result that ensures correct asymptotic coverage for the subsampling-based confidence bands just defined.

\begin{theorem}\label{thm:CB:subs}
Let the assumptions of Theorem~\ref{weakconvintegspectrum} hold and assume moreover that~$(X_t)_{t \in \Z}$ is $\alpha$-mixing such that $\alpha(n) \rightarrow 0$ as $n \rightarrow \infty$. Assume that $b \rightarrow \infty$ and $b = o(n)$ as $n \rightarrow \infty$.
Then, for the confidence band defined in~\eqref{def:I_D_Re},  
\begin{equation*}
\IP\Big( \Re \specdis(\lambda ;\tau_1,\tau_2) \in \widehat I_{\alpha,{\rm Re}}^D(\lambda,\tau_1,\tau_2), \quad \forall \lambda \in [0,\pi]  \Big) \rightarrow 1-\alpha,
\end{equation*}
as $n,d \rightarrow \infty$. Further, assuming that 
\begin{equation} \label{eq:SetConv}
d(S_n,S) := \sup_{y \in S} \inf_{x \in S_n} \|x-y\| \to 0
\end{equation}
for some  $S \subset [\eta,1-\eta]^2$, we have, for the confidence band defined in~\eqref{def:I_E_Re},
\begin{equation*}
\IP\Big( \Re \specdis(\lambda ;\tau_1,\tau_2) \in \widehat I_{\alpha,{\rm Re}}^E(\lambda,\tau_1,\tau_2), \quad \forall (\tau_1, \tau_2) \in S, \lambda \in [0,\pi]  \Big) \rightarrow 1-\alpha.
\end{equation*}
 as $n,d \rightarrow \infty$. The same results hold for the bands for imaginary parts $\Im \specdis(\lambda ;\tau_1,\tau_2)$.
\end{theorem}

\subsection{Testing for time-reversibility}\label{sec:timerev}
An important feature that cannot be captured by second-order moments, hence escapes traditional spectral analysis, is time-(ir)reversibility. Time-irreversibility in time series is the rule rather than the exception (see e.g.\ \cite{hlp88}); it is ubiquitous in some applications such as financial econometrics. Yet, due to the fact that $\text{Cov}(X_t, X_{t-k})= \text{Cov}(X_{t-k}, X_t)$, most classical time-series models generate time-reversible processes while 
 classical spectral analysis, being second-order-based,  is unable to detect time-irreversibility. Copula-based spectral methods can.

Let the stochastic process $(X_t)_{t \in \Z}$   satisfy Assumption~\ref{assumptionsspectraldist}; denote by  $\mathfrak{f}^X$ its copula spectral density,  by  $F_k(x,y) := {\rm P}(X_{k} \leq x, X_0 \leq y)$, $k\in\mathbb Z$, $(x,y)\in\mathbb R^2$ its marginal bivariate distributions. We say that the process $(X_t)_{t \in \Z}$ is {\it pairwise time-reversible} if, for all $k\in\mathbb Z$,  the   distributions of $(X_t,X_{t+k})$ and $(X_t,X_{t-k})$ coincide, i.e., $F_k=F_{-k}$ for all $k\in\mathbb{N}$. The following characterization has been established by \cite{DetteEtAl2013}.

\begin{prop}\label{proptrv}
The process $(X_t)_{t \in \Z}$ is pairwise time-reversible 
 if and only if
$$\Im\mathfrak{f}^X(\lambda;  \tau_1,\tau_2)=0 \quad\text{for all $(\lambda,\tau_1, \tau_2)\in[0,\pi]\times(0,1)^2$}.$$
\end{prop}

A test for (pairwise) time-reversibility thus is a test of the null hypothesis 
\begin{eqnarray}\label{TRtest0}
H_0: F_k(x,y) = F_{-k}(x,y)\quad\text{ for all } (k,x,y)\in\mathbb Z\times\mathbb R^2,
\end{eqnarray}
with alternative  
$$
H_1: F_k(x,y) \neq F_{-k}(x,y)\quad  \text{ for some } (k,x,y)\in\mathbb Z\times\mathbb R^2.
$$
It follows from Proposition~\ref{proptrv} that $H_0$ in~\eqref{TRtest0} also can be written as
\begin{eqnarray}\label{TRtest}
H_0: \sup_{(\lambda,\tau_1, \tau_2) \in  [0,\pi]\times {[\eta,1-\eta]^2}} 
\Big| 
\frac{\Im{\specdis} (\lambda, \tau_1, \tau_2)}{s(\tau_1, \tau_2)}
\Big| = 0,
\end{eqnarray}
 for arbitrarily small $\eta\in(0,1/2),$ where $s: [0,1]^2 \rightarrow [\varepsilon, \infty)$  for some $\varepsilon > 0$.
The function~$s$ is essential to construct the critical region uniformly in $(\lambda ;\tau_1,\tau_2)$ (see the discussion in Section~\ref{sec:subsampl}).
Consider the test statistic (for testing $H_0$ against $H_1$) 
\begin{equation}\label{TTR}
\widetilde{T}_{\rm TR}^{(n)} :=  \sqrt n\sup_{(\lambda,\tau_1, \tau_2) \in  [0,\pi]\times{[\eta,1-\eta]^2}}  \Big| \frac{\Im \widehat {\specdis}_{n,R}(\lambda, \tau_1, \tau_2)}{s(\tau_1, \tau_2)}\Big|.
\end{equation}

The next result is an immediate consequence of Theorem \ref{weakconvintegspectrum}.
\begin{prop}
Let  $(X_t)_{t \in \Z}$   satisfy Assumption \ref{assumptionsspectraldist}. Then, under~$H_0$ defined in  (\ref{TRtest}), as~$n\to\infty$,
$\sqrt n \widetilde{T}_{\rm TR}^{(n)}$ 
converges in distribution  to
\begin{equation}\label{TTRasymdist}
\sup_{(\lambda,\tau_1, \tau_2) \in  [0,\pi]\times[\eta,1-\eta]^2} 
\Big| 
\frac{\Im\mathbb{G} (\lambda, \tau_1, \tau_2)}{s(\tau_1, \tau_2)}
\Big|,
\end{equation}
where $\mathbb{G}(\lambda ;\tau_1,\tau_2)$ is a centered Gaussian process with covariance structure \eqref{covasympintspec}.
\end{prop}

In actual calculations, $\widetilde T_{\rm TR}^{(n)}$ needs to be discretized, and we compute it as 
\begin{equation}\label{Tmax}
T_{\rm TR}^{(n)} := \sqrt n\max_{(\lambda,\tau_1, \tau_2) \in S_n}
\Big| \frac{\Im  \widehat{\specdis}\phantom{F\!\!\!\!\!}_{n,R}(\lambda, \tau_1, \tau_2)}{s(\tau_1, \tau_2)}\Big|,
\end{equation}
where   $S_n$ denotes a sequence of discrete sets the exact choice of which will be discussed in more detail in Section~\ref{sec:sim}. In our theoretical analysis,   we will assume that there exists a  sub-\linebreak set~$S\subseteq[0,\pi]\times[\eta,1-\eta]^2$ such that 
\begin{equation}\label{eq:setconv}
\sup_{x \in S} \inf_{y \in S_n} \|x-y\| \to 0\quad\text{ as $n\to\infty$.}
\end{equation}
Asymptotic $p$--values for this test can be determined based on subsampling:  let
\begin{equation*}
p_{{\rm TR}}:= \frac{1}{n-b+1} \sum_{t=0}^{n-b}I\big\{T_{{\rm TR}1}^{(n,b,t)}> T_{\rm TR}^{(n)}\big\}, 
\end{equation*}
where
\begin{align*}
T_{\rm TR1}^{(n,b,t)}&:=\sqrt{b}\max_{ (\lambda, \tau_1, \tau_2) \in S_n}\Big| \frac{\Im \widehat{\specdis}\phantom{F\!\!\!\!\!}_{n,b,t,R}(\lambda, \tau_1, \tau_2)}{s(\tau_1, \tau_2)}\Big|
\end{align*}
with $\widehat{\specdis}\phantom{F\!\!\!\!\!}_{n,b,t,R}(\lambda, \tau_1, \tau_2)$ defined in~\eqref{eq:hatFnb} denoting the subsampled version of $T_{\rm TR}^{(n)}$ on the\linebreak block~$X_t,\dots,X_{t+b-1}$ of length $b$. The validity of this subsampling procedure is discussed in the next theorem.
\begin{theorem}\label{thm:substr}
Let the assumptions of Theorem~\ref{weakconvintegspectrum} hold and assume moreover that~$(X_t)_{t \in \Z}$ is $\alpha$-mixing such that $\alpha(n) \rightarrow 0$ as $n \rightarrow \infty$. Assume further that~\eqref{eq:setconv} holds and that the weight function $s$ is continuous. Then
\begin{enumerate}
\item[(i)] the test   rejecting $H_0$ in~\eqref{TRtest0} whenever $p_{\rm TR} < \alpha$ has asymptotic level $\alpha$;
\item[(ii)]  the power of this test  converges to one whenever $|\Im \specdis(\lambda, \tau_1, \tau_2)| \neq 0$ for some $(\lambda,\tau_1,\tau_2) \in~\!S$. 
\end{enumerate}	     	
\end{theorem}

\begin{rem}
We also considered the subsampled statistic 
\[
T_{\rm TR2}^{(n,b,t)}:=\sqrt{b}\max_{(\lambda, \tau_1, \tau_2) \in S_n}\Big| \frac{\Im \widehat{\specdis}\phantom{F\!\!\!\!\!}_{n,b,t,R}(\lambda, \tau_1, \tau_2)	-\Im \widehat{\specdis}\phantom{F\!\!\!\!\!}_{n,R}(\lambda, \tau_1, \tau_2)}{s(\tau_1, \tau_2)}\Big|,
\] 
but this did not yield  better results in simulations. 
\end{rem}

\subsection{Assessing asymmetry in tail dynamics} \label{sec:symmetry}

Assessing asymmetry in tail dynamics is of critical importance for, e.g., risk management and investment strategy. Value at risk (VaR) and expected shortfall (ES) are popular risk measures in finance that are related to quantiles.
According to \cite{jr2003}, investors  suspect  that the left tail of stock returns is heavier than the right one. And  \cite{li2021} pointed out asymmetry between lower quantiles and upper quantiles for the S\&P500 index. 
As for copula-based modeling, asymmetry between upper and lower quantiles excludes families of (radially) symmetric copulas such as Gaussian  and $t$-copulas. Misspecified copulas lead to  false conclusions and involve  grave risks (\cite{rj2013, mangold2017}). Hence, the investigation of tail behavior is important. Further discussions can be found in \cite{sc2014} and \cite{kj2019}.

Denote  by $C_k$ the lag--$k$ copula of $(X_0,X_k)$ for some lag $k$. We are interested in the case where 
$$ C_k(\tau_1,\tau_2) -\tau_1\tau_2\neq C_k(1-\tau_1,1-\tau_2) -(1-\tau_1)(1-\tau_2)$$
 for some $(\tau_1,\tau_2)\in(0,\psi)^2$:  the copula~$C_k$ then is called {\it tail asymmetric at a level} $\psi$.
This is not the case when $ C_k(\tau_1,\tau_2) -\tau_1\tau_2= C_k(1-\tau_1,1-\tau_2) -(1-\tau_1)(1-\tau_2)$ for all $k \in \IZ$ and all~$(\tau_1, \tau_2) \in (0,\psi)^2$, where $\psi\in(0,1/2]$: then we say that the copula~$C_k$ is pairwise tail-symmetric at  level $\psi$. 
Note that tail symmetry boils down to radial symmetry when it holds  that 
$$ C_k(\tau_1,\tau_2) -\tau_1\tau_2= C_k(1-\tau_1,1-\tau_2) -(1-\tau_1)(1-\tau_2)$$
for all $\tau_1, \tau_2\in(0,1)$,  see e.g.\  \citet[p.36-p.38]{nelsen06}. 
We call a process $(X_t)_{t \in \Z}$ pairwise tail-symmetric at  level $\psi$ if the copula $C_k$ of $(X_{t+k}, X_t)$ is tail-symmetric at a level $\psi$  for all~$k\in\mathbb Z$.

A test for (pairwise) tail symmetry of $(X_t)_{t \in \Z}$ at given level $\psi\in(0,1/2]$ is a test of the null hypothesis 
\begin{equation}\label{EQtest0}
H_0: C_k(\tau_1,\tau_2) -\tau_1\tau_2= C_k(1-\tau_1,1-\tau_2) -(1-\tau_1)(1-\tau_2) ~~ \forall (k,\tau_1,\tau_2) \in\mathbb Z \times\mathbb (0,\psi)^2
\end{equation}
against the  alternative  
$$
H_1:  C_k(\tau_1,\tau_2) -\tau_1\tau_2 \neq C_k(1-\tau_1,1-\tau_2) -(1-\tau_1)(1-\tau_2) \   \text{ for some } (k,\tau_1,\tau_2)\in\mathbb Z\times(0,\psi)^2\!.
$$
 The null hypothesis $H_0$ can be rewritten  as 
$$\mathfrak{f}(\lambda;\tau_1,\tau_2)=\mathfrak{f}(\lambda;1-\tau_1,1-\tau_2) \ \text{ for all $(\lambda,\tau_1, \tau_2)\in[0,\pi]\times(0,\psi)^2\!$.}$$
 Hence, the following proposition holds true.

\begin{prop}\label{proptEQ}
	The process $(X_t)_{t \in \Z}$ is pairwise tail-symmetric at  level $\psi\in(0,1/2)$
	if and only~if
	$$\specdis(\lambda ;\tau_1,\tau_2) = \specdis(\lambda ;1-\tau_1,1-\tau_2) \quad\text{for all $(\lambda,\tau_1, \tau_2)\in[0,\pi]\times(0,\psi)^2$}.$$
\end{prop}

In view of  Proposition~\ref{proptEQ}, we also consider the following hypothesis, which is slightly weaker  than \eqref{EQtest0}:  
for arbitrary small $\eta\in(0,1/2]$ such that $\eta\leq\psi$,
\begin{eqnarray}\label{EQtest}
H_0: \sup_{(\lambda,\tau_1, \tau_2) \in  [0,\pi]\times {[\eta,\psi]^2}} 
\Big| 
\frac{\specdis (\lambda, \tau_1, \tau_2) - \specdis (\lambda, 1-\tau_1, 1-\tau_2)}{s(\tau_1, \tau_2)}
\Big| = 0,
\end{eqnarray}
where $s: [0,1]^2 \rightarrow [\varepsilon, \infty)$  for some $\varepsilon > 0$. 
For testing $H_0$  against $H_1$, define 
\begin{equation}\label{TEQ}
\widetilde{T}_{\rm EQ}^{(n)} :=  \sqrt n\sup_{(\lambda,\tau_1, \tau_2) \in  [0,\pi]\times{[\eta,\psi]^2}}  \Big| \frac{\widehat {\specdis}_{n,R}(\lambda, \tau_1, \tau_2)- \widehat {\specdis}_{n,R}(\lambda,1- \tau_1, 1-\tau_2)}{s(\tau_1, \tau_2)}\Big|.
\end{equation}
The next result then  is an immediate consequence of Theorem \ref{weakconvintegspectrum}.
\begin{prop}\label{TEQasymdist}
Let  $(X_t)_{t \in \Z}$   satisfy Assumption \ref{assumptionsspectraldist}. Then, under~$H_0$ defined in  (\ref{EQtest}), as $n\to\infty$, 
$\sqrt n \widetilde{T}_{\rm TR}^{(n)}$ 
converges in distribution to
\begin{equation*}
\sup_{(\lambda,\tau_1, \tau_2) \in  [0,\pi]\times{[\eta,\psi]^2}} 
\Big| 
\frac{\mathbb{G} (\lambda, \tau_1, \tau_2)-\mathbb{G} (\lambda,1- \tau_1, 1-\tau_2)}{s(\tau_1, \tau_2)}
\Big|,
\end{equation*}
where $\mathbb{G}(\lambda ;\tau_1,\tau_2)$ is a centered Gaussian process with covariance structure \eqref{covasympintspec}.
\end{prop}

In practice, a  discretisation   
\begin{equation}\label{teststattails}
T_{\rm EQ}^{(n)} := \sqrt n\max_{(\lambda,\tau_1, \tau_2) \in S_n}
\Big| \frac{\widehat {\specdis}_{n,R}(\lambda, \tau_1, \tau_2)- \widehat {\specdis}_{n,R}(\lambda,1- \tau_1, 1-\tau_2)}{s(\tau_1, \tau_2)}\Big|,
\end{equation}
of $\widetilde{T}_{\rm EQ}^{(n)}$ is required, where the sequence $S_n$ is such that 
\begin{equation}\label{eq:setconv_EQ}
\sup_{x \in S} \inf_{y \in S_n} \|x-y\| \to 0\quad \text{for some $S \subseteq [0,\pi]\times{[\eta,\psi]^2}$}.
\end{equation}

\medskip

The $p$-value of the resulting test  for (pairwise) tail symmetry is  
\begin{equation*}
p_{\rm EQ}:= \frac{1}{n-b+1}\sum_{t=0}^{n-b}I\big\{T_{{\rm EQ}}^{(n,b,t)}> T_{\rm EQ}^{(n)}\big\},
\end{equation*}
where
\begin{align*}
&T_{\rm EQ}^{(n,b,t)}
:=\sqrt{{b}} 
\max_{ (\lambda, \tau_1, \tau_2) \in S_n}
\Big|\frac{\widehat{\specdis}\phantom{F\!\!\!\!\!}_{n,b,t,R}^X(\lambda,\tau_1,\tau_2)-\widehat{\specdis}\phantom{F\!\!\!\!\!}_{n,b,t,R}^{X}(\lambda,1-\tau_1,1-\tau_2)}{s(\tau_1, \tau_2)}\Big|
\end{align*}
with $\widehat{\specdis}\phantom{F\!\!\!\!\!}_{n,b,t,R}(\lambda, \tau_1, \tau_2)$ defined in~\eqref{eq:hatFnb}.
The next theorem establishes the  properties of the   testing procedure based on $T_{\rm EQ}^{(n,b,t)}
$.
\begin{theorem}\label{thm:subseq}
Let the assumptions of Theorem~\ref{weakconvintegspectrum} hold and assume moreover that~$(X_t)_{t \in \Z}$ is $\alpha$-mixing such that $\alpha(n) \rightarrow 0$ as $n \rightarrow \infty$. Assume further that~\eqref{eq:setconv_EQ} holds and that the weight function $s$ is continuous. Then
\begin{enumerate}
\item[(i)] the test  rejecting $H_0$ in~\eqref{EQtest} whenever $p_{\rm EQ} < \alpha$ has asymptotic level $\alpha$;
\item[(ii)] the power of that test   converges to one whenever $|\mathfrak{f}^X(\lambda;  \tau_1,\tau_2)-\mathfrak{f}^X(\lambda,1- \tau_1,1-~\!\tau_2)| \neq~\!0$ for some $(\lambda,\tau_1,\tau_2) \in S$. 
\end{enumerate}	     	
\end{theorem}


\section{Simulations} \label{sec:sim}

This section illustrates the finite-sample performance of the  methods proposed in Sections~\ref{sec:subsampl}--\ref{sec:symmetry}.
We consider a range  M0-M15 of fifteen models, which we  describe in detail in the Appendix. These models include linear and nonlinear ones, Gaussian and non-Gaussian ones, models with serial independence, weak serial dependence, and stronger serial dependence. Table~\ref{tb:models} lists  the main features   of these models. 
The R package quantspec  \citep{Kley2016} was used for all simulations.

\begin{table}[h!]
  \caption{Main features of models  M0-M15.  
A check mark indicates that a model violates the null hypothesis $H_0$. A cross mark indicates a null hypothesis $H_0$ we are not interested in for the model.\vspace{4mm}}
  \label{tb:models}
  \centering
  \begin{tabular}{c|c|c|c}\hline
model&$H_0:$ time&$H_0:$ tail &  short  \\ 
&-reversibility&symmetry& description   \\\hline
M0&&&i.i.d.~Gaussian\\
M1&\checkmark&\checkmark&QAR(1)   \citep{Koenker2006}\\
M2& &&AR(2)   \citep{Li2012}\\
M3&\checkmark&&ARCH(1)   \citep{LeeRao2012}.\\
M4&\checkmark&&GARCH(1,1)   \citep{BirrEtAl2019}\\
M5&\checkmark&\checkmark&EGARCH(1,1,1)  \citep{BirrEtAl2019}\\
M6a--c& &&AR(1) with  Gaussian innovation\\
M7a--c&\checkmark&&AR(1) with  Cauchy innovation\\\hline
M8a--g&\checkmark&$\times$&
\begin{tabular}{c}
time series based on an asymmetric Gumbel \\ copula   \citep{bs2014} 
\end{tabular}\\
M9a--g&\checkmark&$\times$& 
\begin{tabular}{c}
time series based on a zero total circulation\\  copula  \citep{bs2014}
\end{tabular}\\
M10a--g&\checkmark&$\times$& the modified models  M8a--g\\
M11a--g&\checkmark&$\times$& the modified models  M9a--g\\\hline
M12a--c&$\times$&\checkmark&
\begin{tabular}{c}
time series based on a Gumbel copula\\   \citep{lg13}
\end{tabular}\\
M13a--c&$\times$&\checkmark&
\begin{tabular}{c}
time series based on a Clayton copula\\   \citep{lg13}\\
\end{tabular}\\
M14&$\times$&\checkmark& 
\begin{tabular}{c}
time series based on   copula 3\\ of  \cite{nelsen93}
\end{tabular}\\
M15&$\times$&\checkmark& 
\begin{tabular}{c}
time series based on   copula 6\\ of \cite{nelsen93}
\end{tabular}\\\hline
  \end{tabular}
\end{table}

 \begin{figure}[ht!]
	\begin{center}
\includegraphics[width = \linewidth]{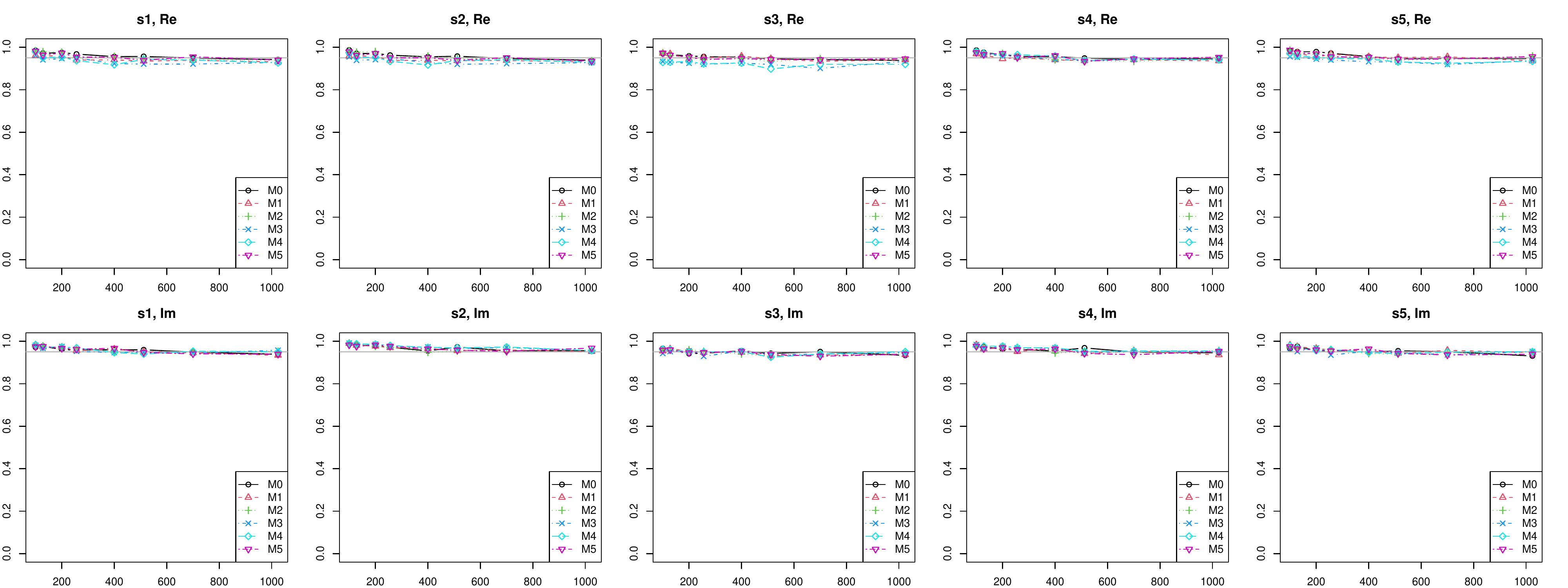}
\includegraphics[width = \linewidth]{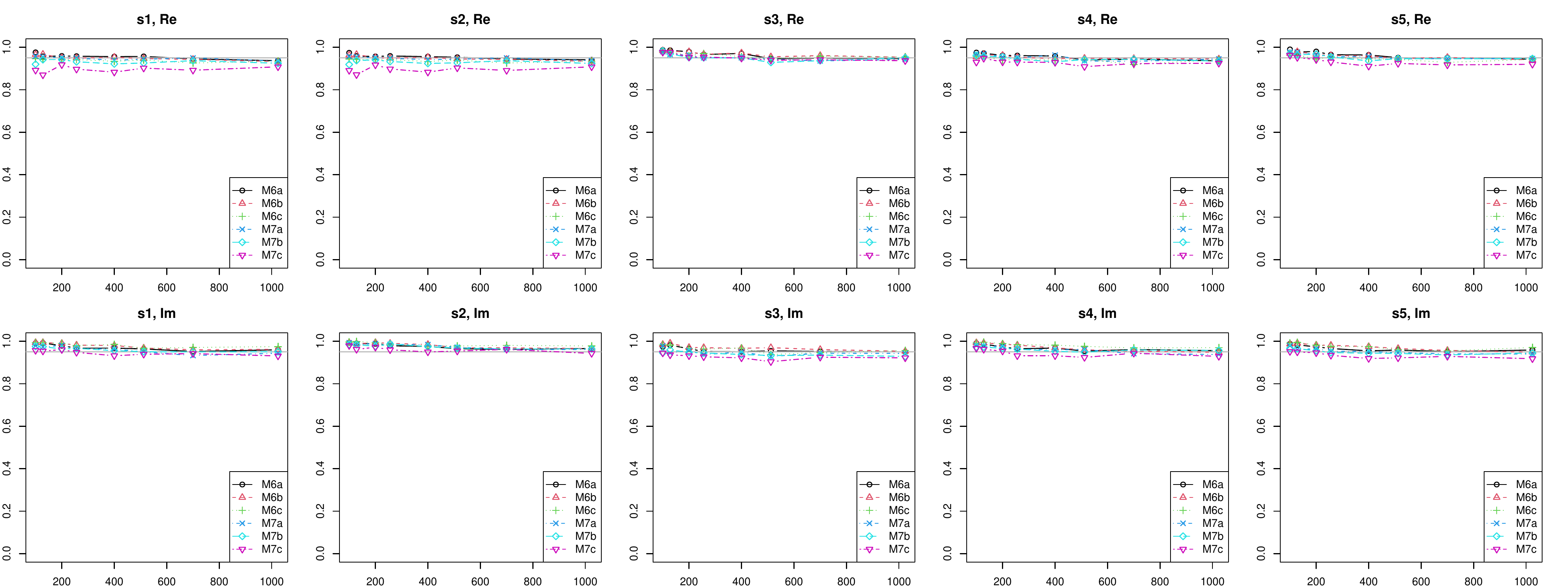}
	\end{center}
	\vspace{-0.4cm}
		\caption{\small Uniform in $(\lambda, \tau_1,\tau_2)$ confidence bands, models M0-M7.  Coverage probabilities with finite-population correction and weight functions $s_1,\ldots,s_5$  as a function of $n$.  Column $i$ corresponds to the weight function $s_i$,  the first and third rows   to the real parts, the 	second and fourth rows   to the imaginary parts, of the integrated copula spectra.\vspace{-3mm} }		
	\label{fig:CB_unif_fpc}
\end{figure}

 \begin{figure}[ht!]
	\begin{center}
\includegraphics[width = \linewidth]{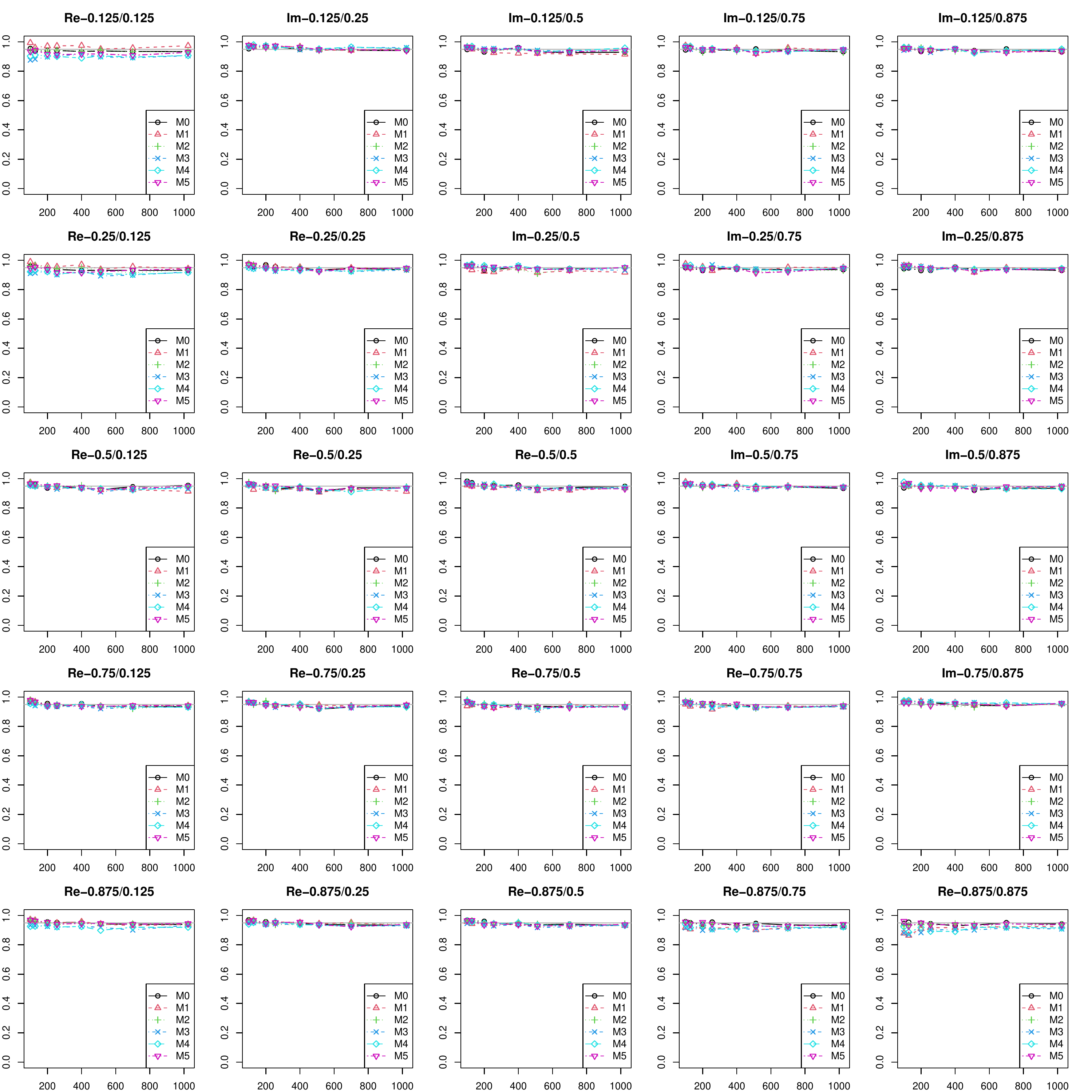}
	\end{center}
	\vspace{-0.4cm}
	    \caption{\small Uniform in $\lambda$ (pointwise in  $(\tau_1,\tau_2)$) confidence bands, models M0-M5.   Coverage probabilities with finite-population correction.    Each subplot has a label  indicating whether it is dealing with the real or imaginary part of the integrated spectrum, and which quantile levels  $(\tau_1,\tau_2)$ were considered: e.g., the subplot with  label Re-0.25/0.125 is about the real part  of the integrated spectrum with $\tau_1 = 0.25$ and $\tau_2 = 0.125$.}
	\label{fig:CB_pw_fpc_a}
\end{figure}
\begin{figure}[ht]
	\begin{center}
\includegraphics[width = \linewidth]{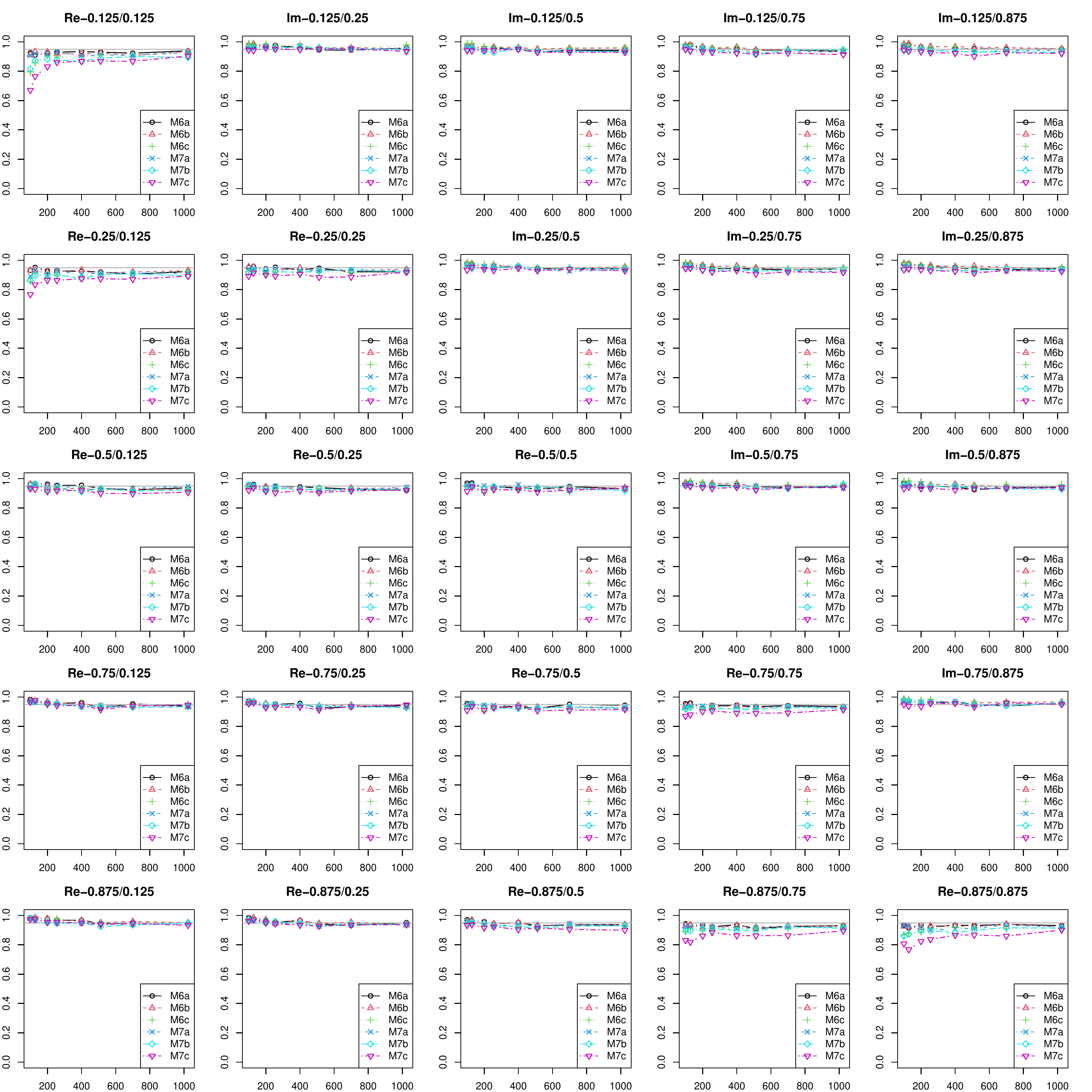}
	\end{center}
	\vspace{-0.4cm}
	    \caption{\small Uniform in $\lambda$ (pointwise in  $(\tau_1,\tau_2)$) confidence bands, models M6-M7.   Coverage probabilities with finite-population correction.    Each subplot has a label  indicating whether it is dealing with the real or imaginary part of the integrated spectrum, and which quantile levels  $(\tau_1,\tau_2)$ were considered: e.g., the subplot with  label Re-0.25/0.125 is about the real part  of the integrated spectrum with $\tau_1 = 0.25$ and $\tau_2 = 0.125$.}
	\label{fig:CB_pw_fpc_b}
\end{figure}

\subsection{Confidence bands}\label{confbandsec}
In this subsection,   models  M0-M7 from Table~\ref{tb:models} are used to study the empirical coverage\footnote{Throughout, with a slight abuse of language, we  write ``coverage probability'' instead of ``coverage frequency'' in order to avoid confusion with $\lambda$.} of the confidence bands described in Section~\ref{sec:subsampl}. 
We consi\-der~$n\!\in~\!\!\{100, 128, 200, 256, 400, 512, 700, 1024\}$ and, for each   $n$ (choosing powers of 2 for $b$ allows for quick computation of the CR periodograms), $b \in B(n) := \{2^4, 2^5, \ldots, \}\cap [0, n/2]$; as a rule of thumb, we selected 
\begin{equation}\label{eqn:rt_bw}
b^{\rm rt}_n := \max\{ 2^j :  2^j \leq 2n^{2/3}, \ j=4,\ldots,8 \},
\end{equation}
yielding $b = 32,  32,  64,  64,  64, 128, 128, 128$ for $n = 100, 128, 200, 256, 400, 512, 700, 1024$, respectively.
As for the Fourier frequencies in \eqref{eqfpc},  
 we put $d=32$.  \phantom{Figure~\ref{fig:CB_unif_fpc}, Figure~\ref{fig:CB_pw_fpc_a}} 
 
 We simulated $R=1000$ independent  series for each configuration. For each of them, we computed  the confidence band as explained in Section~\ref{sec:subsampl}. To obtain their empirical coverage, we compare them  with the actual value of the integrated copula spectral density. The latter    can be computed precisely for~M0; else, it was obtained from $500,000$ simulated CR periodograms.
 
The finite-population correction in~\eqref{eqfpc} 
  was applied; without it, the results
 (not shown here) are significantly worse:   the correction, thus, is essential in numerical applications. 



{We throughout used $\alpha = 0.05$. We simulated pointwise in $(\tau_1, \tau_2)$ coverage for all $\tau_1, \tau_2$ in~$\{1/16,\ldots,15/16\}$. For the uniform procedures, maxima with respect to all 15 quantile levels were used [see Appendix~\ref{app:covprob} for a detailed description of how coverage is computed]. For pointwise coverage, we only display results for $\tau_1, \tau_2 \in \{0.125, 0.25, 0.5, 0.75, 0.875\}$. }

Figure~\ref{fig:CB_unif_fpc} reports, for models M0-M7 and   the $(\lambda, \tau_1, \tau_2)$-uniform  procedure with finite-population correction~\eqref{eqfpc}, the coverage probabilities as functions of the sample size. For weighting, we have used the weights~$s_1,\ldots,s_5$ defined in the Appendix. All results are very close to the nominal~0.95 level; the equal weights function $s_4$  yields the best results. Figures~\ref{fig:CB_pw_fpc_a} (for models M0-M5) and \ref{fig:CB_pw_fpc_b} (for models M6-M7) report the coverage probabilities 
of the $\lambda$-uniform, $(\tau_1, \tau_2)$-pointwise procedure, still    with finite-population correction. Here and in subsequent tables reporting  $(\tau_1, \tau_2)$-pointwise results, we have followed the convention to show the results for real parts on and below the diagonal and the results for imaginary parts above the diagonal. Overall, the method (with finite-population correction) works well. As expected, large sample sizes are required to obtain reasonable coverage probability for extreme quantiles, for example, $\tau_1=\tau_2=0.125, 0.875$. Especially, the construction of confidence bands for extreme quantiles in models  M3, M4, and M7c is  challenging.
For $\tau_1\neq\tau_2$, the results for imaginary parts are better than  for real parts.

\subsection{Time-reversibility}
\label{sim_TR}

In this subsection, we evaluate, based on models M0-M7 and M8-M11,  the finite-sample performance of the  tests  for time-reversibility introduced in Section~\ref{sec:timerev} and compare it to that of their main competitors.  The simulation procedure is essentially the same as in Section~\ref{confbandsec}:  for each value of the sample   size~$n$ \linebreak 
in~$\{100, 128, 150, 200, 256, 400, 512, 700, 1024\}$, a  subsampling block  size $b(n)$
 is chosen via the rule of thumb~\eqref{eqn:rt_bw}.  The maxima  in the test statistic \eqref{Tmax}  are taken over the frequency range~$\{2\pi \ell/32; \ell = 0,1,\ldots,16\}$ and the quantiles   $\{\tau_1,\tau_2=k/8; k = 1,\ldots,7\}$, with the weight functions~$s_1,\ldots,s_5$ defined in the Appendix. The significance level throughout  is~$\alpha=0.05$.  

 For each case, $R=1000$ replications  were generated.   For each replication, two tests were performed, based on $T_{\rm TR1}^{(n,b,t)}$ (no finite-population correction) and $T_{\rm TR1\_fpc}^{(n,b,t)}:=
 \frac{T_{\rm TR1}^{(n,b,t)}}{ (1 - b/n)^{1/2}}$ (finite-population correction), respectively. The resulting  rejection frequencies with weight function $s_4 \equiv 1$ (empirical sizes for M0, M2, M6,  empirical powers for M1, M3, M4, M5, and~M7) are shown in Figure \ref{fig:TR_nofpc} for $T_{\rm TR1}^{(n,b,t)}$  and Figure \ref{fig:TR_fpc} for $T_{\rm TR1\_fpc}^{(n,b,t)}$, respectively.
  \begin{figure}[ht!]
\begin{center}
\includegraphics[width = \linewidth]{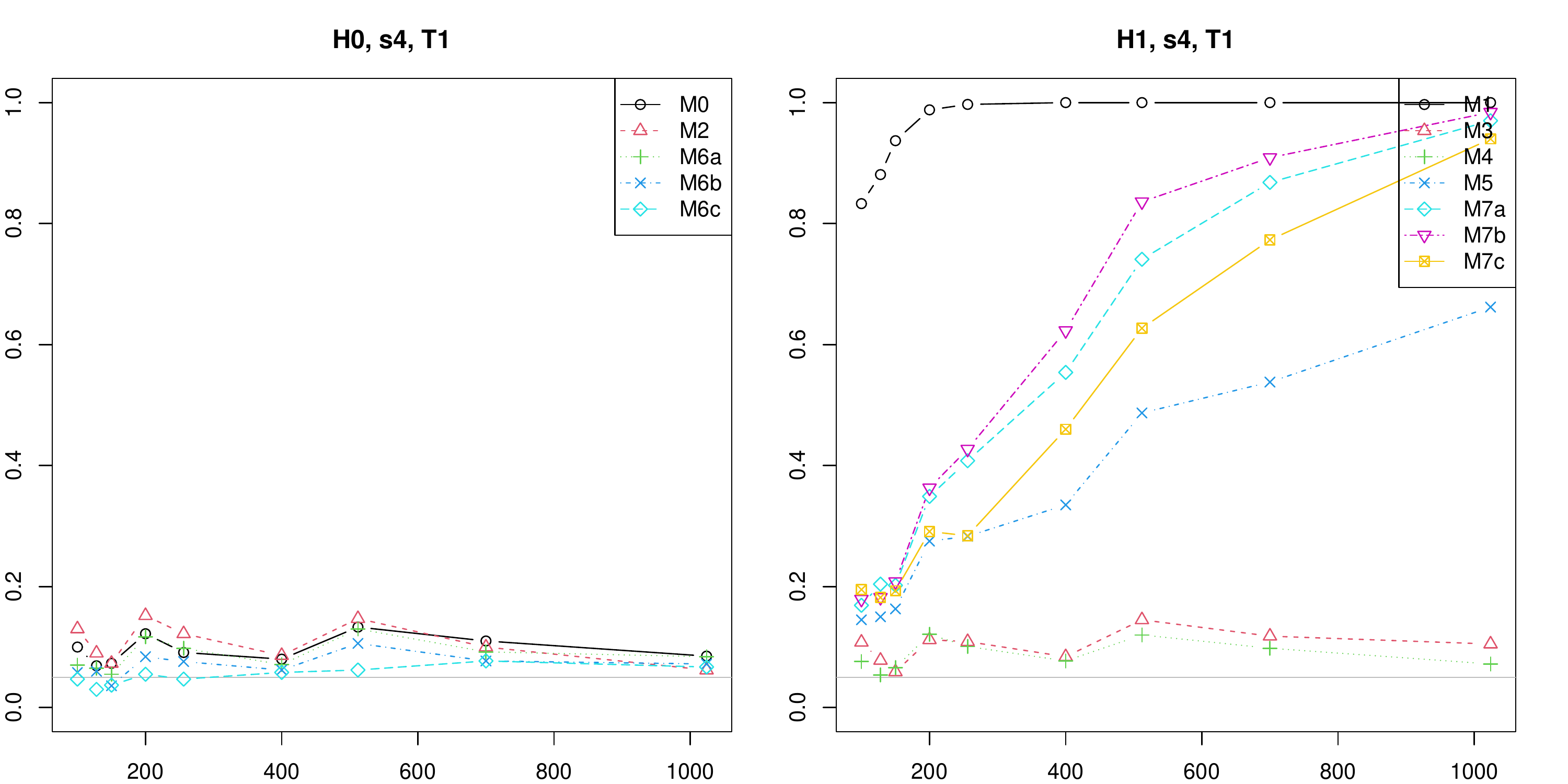}
	\end{center}
	\vspace{-0.3cm}
	\caption{\small Empirical sizes (left, time-reversible models M0, M2, and M6a-c) and powers (right, time-irreversible models M1, M3, M4, M5, and M7a-c) as functions of $n$, of the tests for time-reversibility based on $T_{\rm TR1}^{(n,b,t)}$ 
	(no  finite-population correction). 
}
	\label{fig:TR_nofpc}
\end{figure}

The test based on $T_{\rm TR1}^{(n,b,t)}$ suffers of size distortion (over--rejection) while  the size control, for  the  test  based on $T_{\rm TR1\_fpc}^{(n,b,t)}$, is good. The finite-population correction, thus, is highly re\-commended. We can see that the power of our tests is high for large sample sizes except for~M3-M5. Results for other weight functions are provided in the online supplement. 

\begin{figure}[ht]
	\begin{center}
\includegraphics[width = \linewidth]{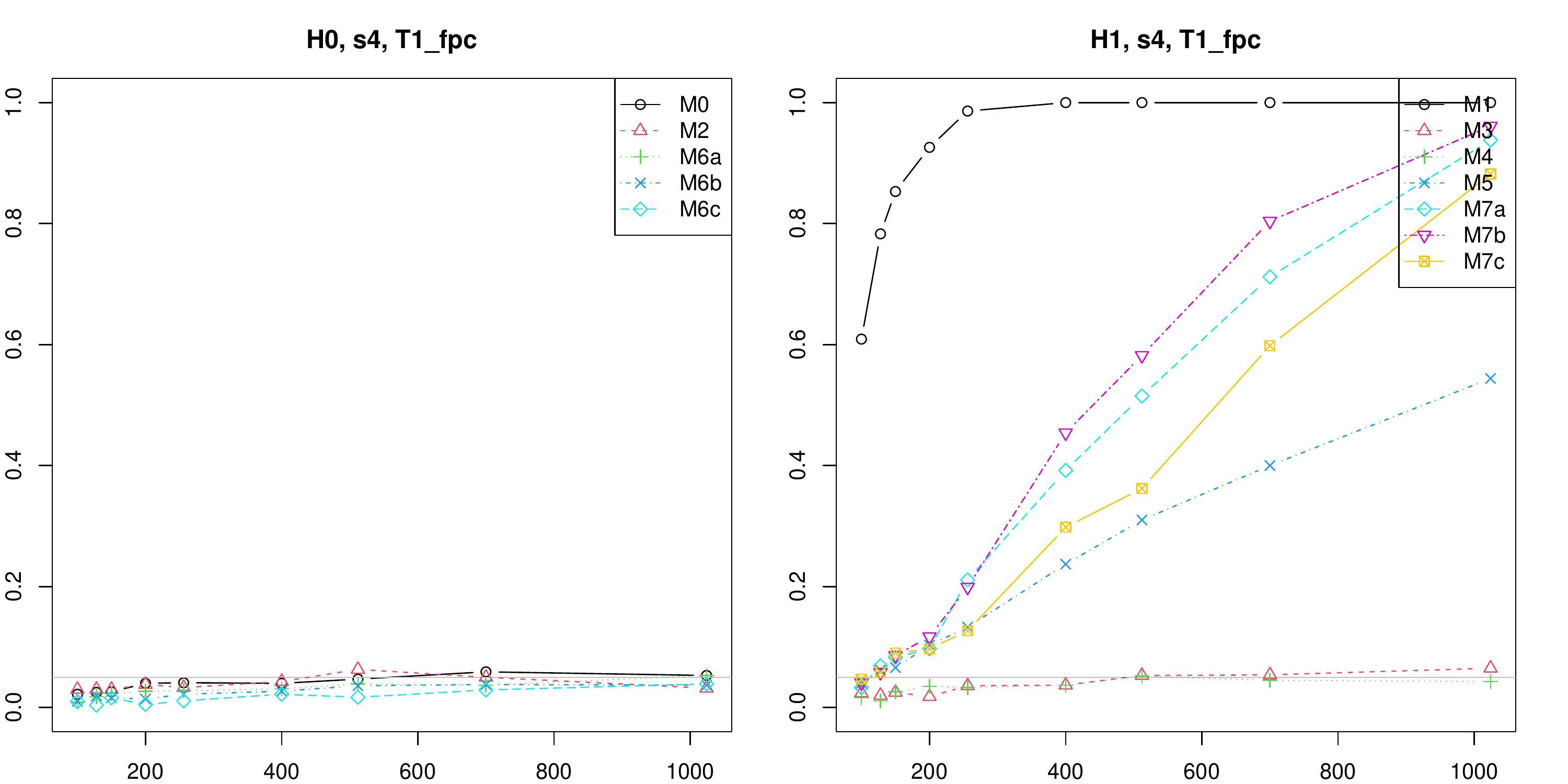}
\end{center}
\caption{\small Empirical sizes (left, time-reversible models M0, M2, and M6a-c) and powers (right, time-irreversible models M1, M3, M4, M5, and M7a-c) as functions of $n$, of the tests for time-reversibility based on $T_{\rm TR1\_fpc}^{(n,b,t)}$ (with  finite-population correction).}
\label{fig:TR_fpc}
\end{figure}

Next, we compare our   tests with the few existing ones, namely,  the tests proposed by 
\cite{rr96}, \cite{cck2000}, \cite{pp2002}, and \cite{bs2014}, based on the test statistics 
\begin{equation}
\begin{array}{rlrcl}
T_{\rm RR}
&:=\dfrac{1}{n-1}\sum_{t=0}^{n-2}(X_{t+1}^2X_{t}-X_{t+1}X_{t}^2),
&T_{\rm CCK}
&:=
&\dfrac{1}{n-1}\sum_{t=0}^{n-2}\dfrac{X_{t+1}-X_t}{1+(X_{t+1}-X_t)^2}
\\
T_{\rm PP}
&:=\dfrac{1}{n-1}\sum_{t=0}^{n-2}I\{X_{t+1}>X_t\}-\dfrac{1}{2},\ \text{and }
&T_{\rm BS}
&:=
&\sup_{(x,y)\in\mathbb R^2}\left|\hat F_n(x,y)-\hat F_n(y,x)\right|,
\end{array}\label{compeq}
\end{equation}
respectively, where $\hat F_n(x,y):=\sum_{t=0}^{n-2}I\{X_{t}\leq x, X_{t+1}\leq y\}/(n-1)$. The critical values of these tests are calculated via   local bootstrap (see Sections 3.2 and 3.3 in~\cite{bs2014}).  The intuition behind   $T_{\rm CCK}$ and $T_{\rm PP}$ is   that time-reversibility of the process $X_t$ implies the  symmetry of~$(X_t-X_{t-1})$  about the origin, while $T_{\rm RR}$ is motivated by the fact that~${\rm E}X_t^2X_{t-1}={\rm E}X_tX_{t-1}^2$ under time-reversibility if $X_t$ has finite third moments. These facts, however, are just necessary conditions for time-reversibility. As for $T_{\rm BS}$, it is based on a property of  Markov processes, which are time-reversible at lag one if and only if  the copula of $(X_0,X_1)$ is. 
\phantom{Figure \ref{fig:TR_M8_M9}}

\begin{figure}[hb!]
	\begin{center}
\includegraphics[width =.6 \linewidth]{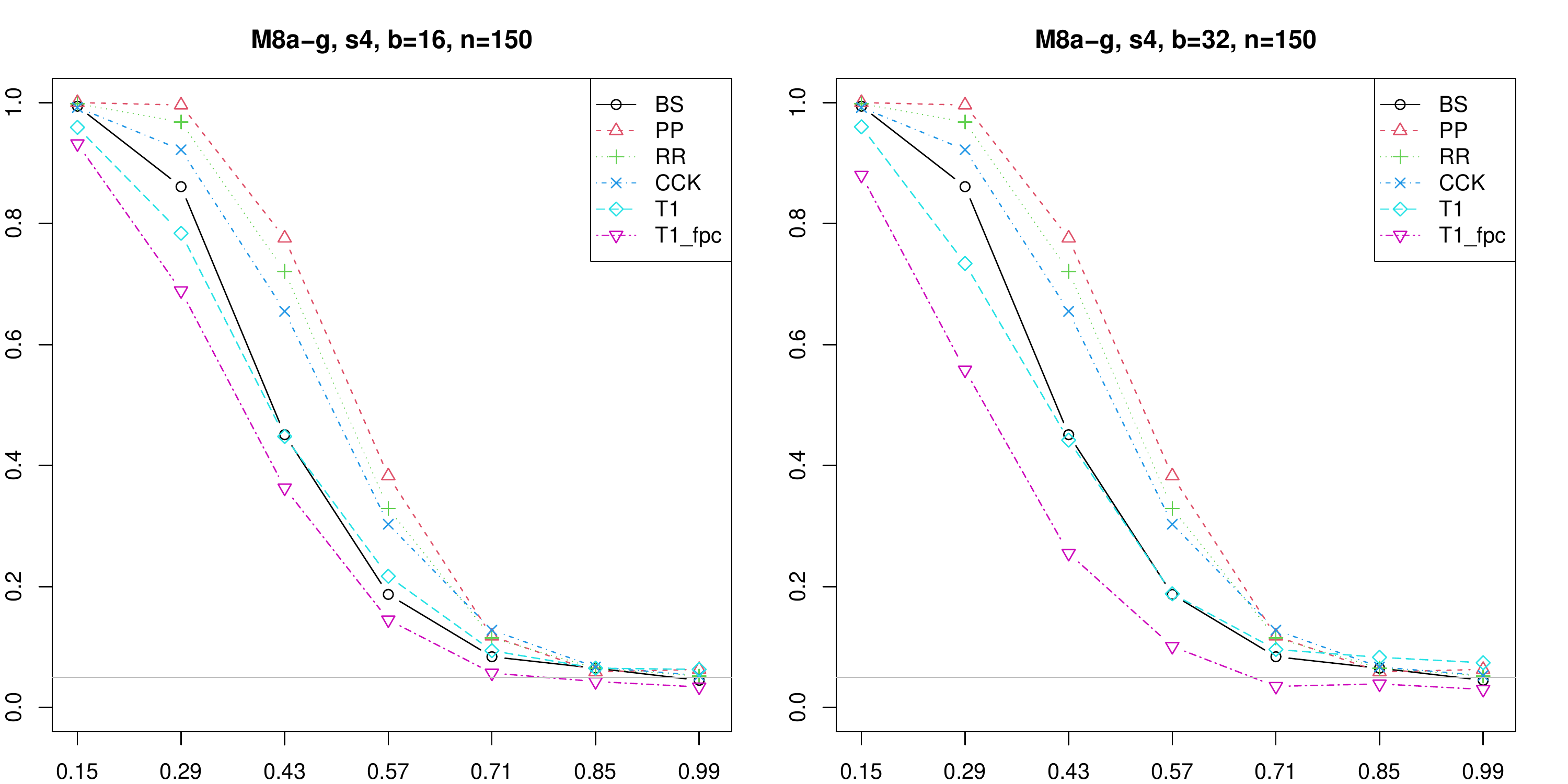}
\includegraphics[width = .6\linewidth]{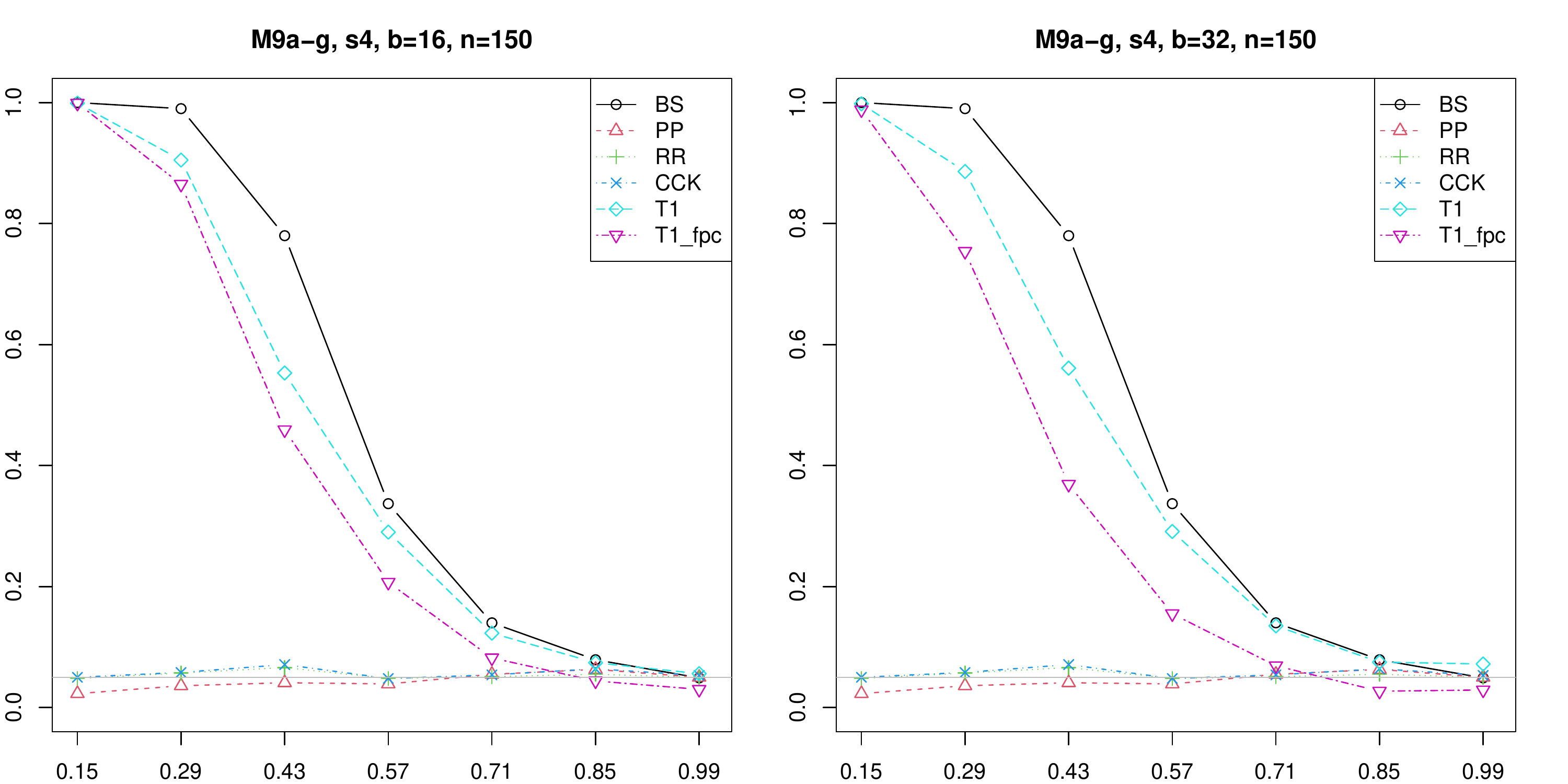}
\end{center}
	\caption{\small
Empirical power of the tests for time-reversibility described in Section \ref{sec:timerev} for $n=150$.
The upper  and lower plots correspond to M8$a$--$g$ and M9$a$--$g$,  the left   and right ones  to   subsampling block sizes $b=16$ and $b=32$, respectively.
}
	\label{fig:TR_M8_M9}
\end{figure}

Our comparison is based on simulations of models M8-M9 with sample size $n=150$   (Figure~\ref{fig:TR_M8_M9}), of  models M10-M11 with sample size $n=512$ (Figure \ref{fig:TR_M10_M11}), with subsampling block sizes  $b=16$ and $b=32$ and   weight function  $s_4 \equiv 1$. Other settings and simulations have been performed, and yield similar results. Empirical power plots are provided in Figures~\ref{fig:TR_M8_M9} and \ref{fig:TR_M10_M11}, with increasing degree of  time-reversibility (measured by the parameters $\lambda$ and $\gamma^{-1}$, respectively, with value one  corresponding to the null hypothesis of time-reversibility) on the horizontal  axis. Model M9 is such  that, among the competitors \eqref{compeq},  only $T_{\rm BS}$ can detect time-irreversibility; models  M10 and M11 are such that none of these competitors can  detect time-irreversibility. Our tests were implemented with and without finite population correction.

Figure \ref{fig:TR_M8_M9} shows the expected result that 
the power of all tests increases with the degree of time-irreversibility  for M8;  the same holds true for M9, but only for our tests and the test based on $T_{\rm BS}$, while $T_{\rm PP}$, $T_{\rm RR}$, and $T_{\rm CCK}$ (which are best under M8) are totally powerless. Our tests behave quite well in all cases, although  outperformed by the test based on $T_{\rm BS}$. Figure \ref{fig:TR_M10_M11}, however,  
\begin{figure}[ht!]
	\begin{center}
\includegraphics[width =.6\linewidth]{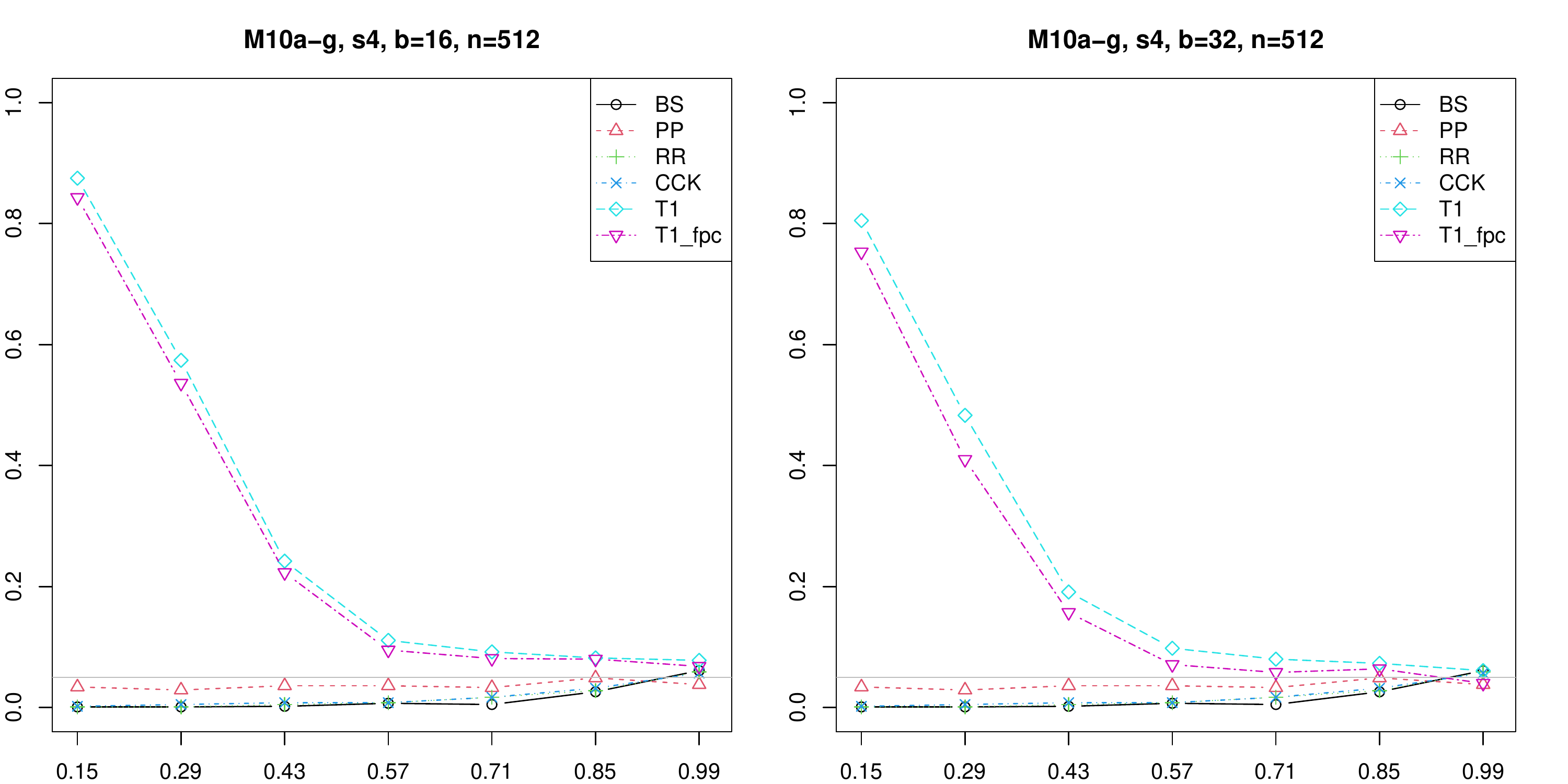}
\includegraphics[width = .6\linewidth]{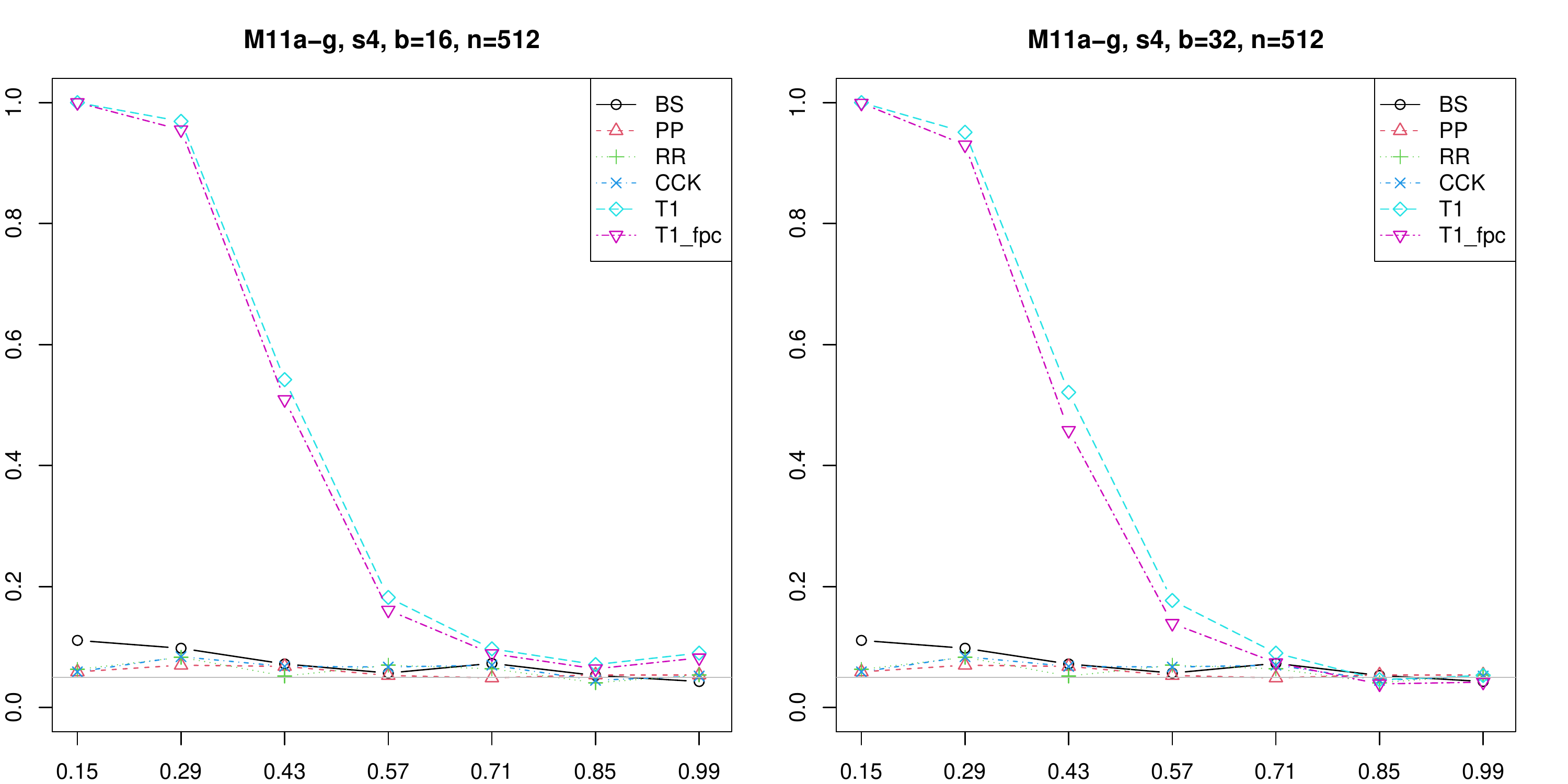}
	\end{center}
	\vspace{-0.4cm}
	\caption{\small Empirical power of the tests for time-reversibility described in Section \ref{sec:timerev}  for $n=512$.
The upper plots and lower plots correspond to M10$a$--$g$ and M11$a$--$g$,  the left   and right ones  to   subsampling block sizes $b=16$ and $b=32$, respectively.
}
	\label{fig:TR_M10_M11}
\end{figure}
establishes that  in models  M10 and M11 with  moderate degree of time-irreversibility, our tests very efficiently do reject time-reversibility while all their competitors, including  the $T_{\rm BS}$-based one, fall short from detecting anything.  The finite population correction and the choice of the subsampling block size apparently have   little impact, irrespective of the model and the sample size. 
Additional simulations can be found in the online Supplement.

\subsection{Asymmetry in tail dynamics}

In order to study the empirical size and power of  the test for quantile symmetry introduced in Section \ref{sec:symmetry}, 
we simulated observations from models M0--M7c and  M12a--M15. For each sample size  $n\in\{100, 128, 200, 256, 400, 512, 700, 1024\}$,   a subsampling block size  $b(n)$ is chosen  via  the rule of thumb   \eqref{eqn:rt_bw}. As in Section~\ref{sim_TR}, the maxima in statistic \eqref{teststattails} were taken over the frequency range   $\{2\pi \ell/32; \ell = 0,1,\ldots,16\}$ and the quantiles   $\{\tau_1,\tau_2=k/16; k = 2,3,4\}$, with weight functions $s_4 \equiv 1$. Significance level throughout  is~$\alpha=0.05$. \phantom{Figure \ref{fig:TR_nofpc}}

 For each case, $R=1000$ replications  were generated.   For each replication, two tests were performed, based on the test statistics $T_{\rm EQ}^{(n,b,t)}$ (as defined in \eqref{teststattails}; no finite-population correction) and~$T_{\rm EQ\_fpc}^{(n,b,t)}:=(1 - b/n)^{-1/2}T_{\rm EQ}^{(n,b,t)}\vspace{1mm}$ (with finite-population correction), respectively. 
  \begin{figure}[hb!]
	\begin{center}
\includegraphics[width = \linewidth]{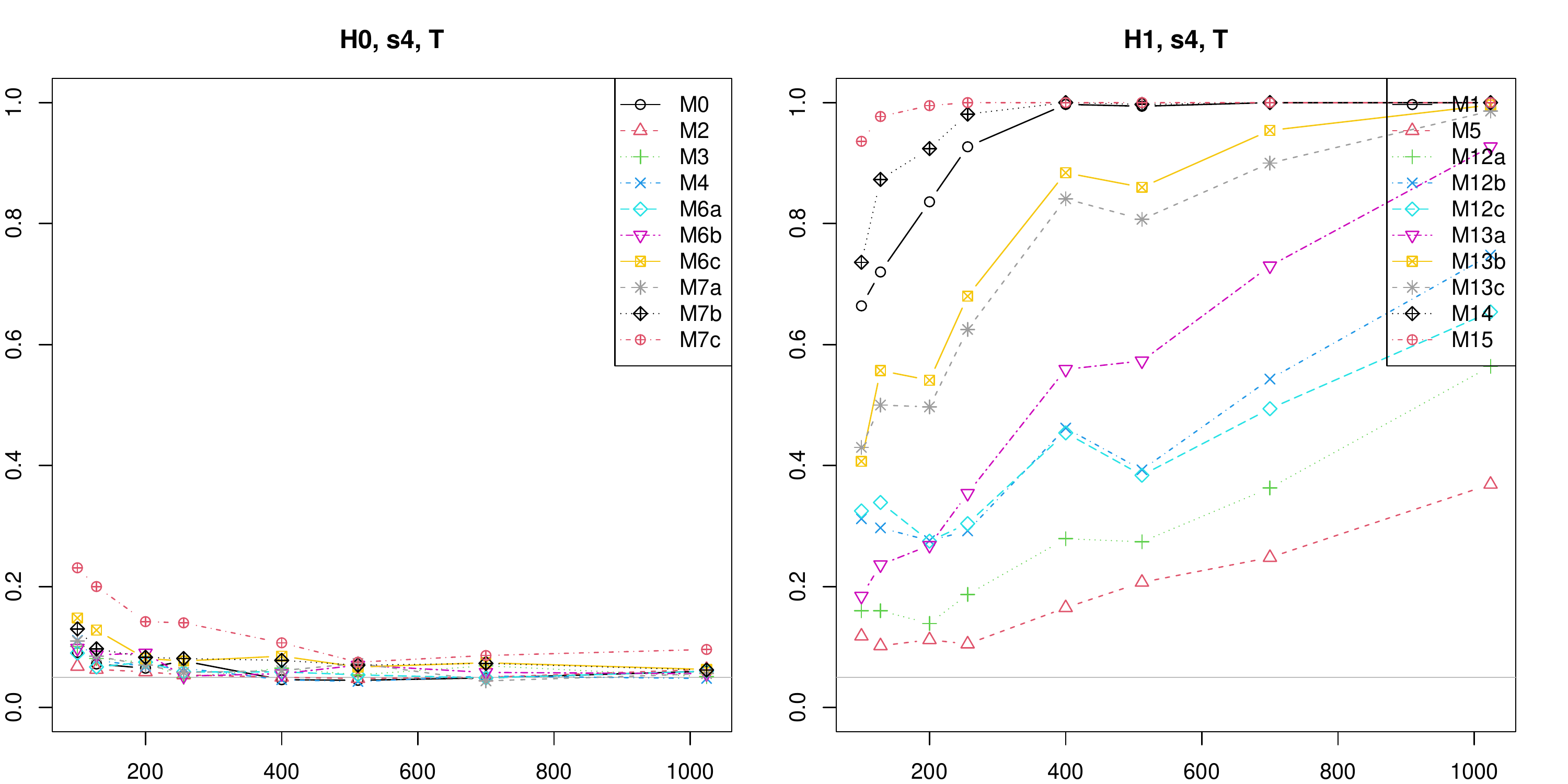}
	\end{center}
	\vspace{-0.4cm}
	\caption{\small 
Empirical sizes (left) and powers (right), as functions of $n$, of the tests for tail symmetry based on $T_{\rm EQ}^{(n,b,t)}$ under various models.  
\vspace{-4mm}}
	\label{fig:EQ_nofpc}
\end{figure}
 
  The resulting  rejection frequencies (empirical sizes for M0, M2, M3, M4, M6a-c, and~M7a-~c, empirical powers for M1, M5, M12a-c, M13a-c, M14, and M15)    are displayed in Figure \ref{fig:EQ_nofpc} for~$T_{\rm EQ}^{(n,b,t)}$  and Figure \ref{fig:EQ_fpc} for $T_{\rm EQ\_fpc}$.

The test based on  $T_{\rm EQ}^{(n,b,t)}\vspace{-.5mm}$ (Figure \ref{fig:EQ_nofpc}) exhibits significant  size distortions for small sample sizes---particularly so under  models  M7b--c and M6c. 
The test based on the corrected statistic~$\vspace{.7mm}T_{\rm EQ\_fpc}^{(n,b,t)}$ provides much better results in that respect, although overrejection is still present under M7b--c. The finite population correction, thus, is still recommended. 
As for   empirical powers, they all increase with the sample size; detecting tail asymmetry in M5 and, to a lesser extent, in M12a remains difficult. Simulation results for additional weight functions are provided in the online Supplement.  


\begin{figure}[h]
\begin{center}
\includegraphics[width = \linewidth]{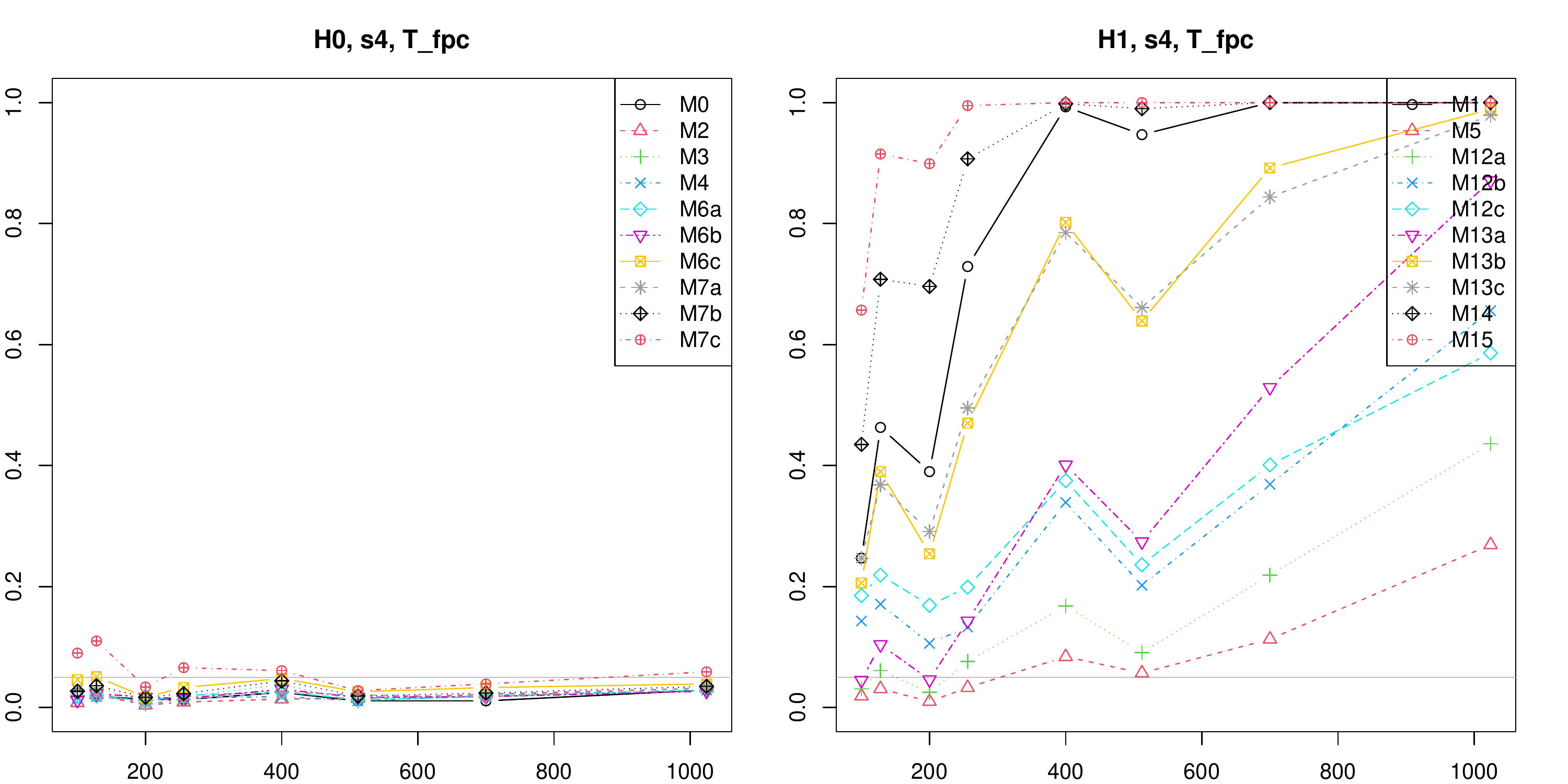}
\end{center}
\vspace{-0.4cm}
\caption{
\small 
Empirical sizes (left) and powers (right), as functions of $n$, of the tests for tail symmetry based on  $T_{\rm TR1\_fpc}$ described in Section \ref{sec:symmetry} under various models. 
The models in the left plots and right plots belong to the null and alternative, respectively. 
\vspace{-4mm}}
	\label{fig:EQ_fpc}
\end{figure}

\begin{appendix}
\section*{Additional details on simulations}\label{appn}

\subsection{Computation of coverage frequencies}\label{app:covprob}
The coverage probability of the procedure that is uniform with respect to $\lambda$ and pointwise with respect to $(\tau_1, \tau_2)$ for a real part
is defined by the empirical probability (with respect to the iterations) of the event, for fixed $\tau_1$ and $\tau_2$,
\begin{align*}
\Big\{\Re{\specdis}\phantom{F\!\!\!\!}\Big(\frac{2\pi \ell}{d},\tau_1,\tau_2\Big) \in \widehat I_{\alpha,{\rm Re}}^D\Big(\frac{2\pi \ell}{d},\tau_1,\tau_2\Big)
\text{ for all $\ell=1,\ldots,d$}\Big\},
\end{align*}
where $\Re{\specdis}$ is the true spectrum derived by the direct calculation for (M0) and the true spectrum simulated by quantspec for the other cases.
In case of the simulated spectra, they are available at the Fourier frequencies $2 \pi k / N$, with $N = 2^{11}$ and we round down to the next available such frequency
$2 \pi \lfloor N \ell /x \rfloor / N$, where $\ell \in 0, \ldots, \lfloor d / 2 \rfloor$.


The coverage probability of the procedure that is uniform with respect to $(\lambda, \tau_1, \tau_2)$ for a real part is defined by the empirical probability (with respect to the iterations) of the event
\begin{align*}
\Big\{\Re{\specdis}\phantom{F\!\!\!\!}\Big(\frac{2\pi \ell}{d},\tau_1,\tau_2\Big)\in
\widehat I_{\alpha,{\rm Re}}^E\Big(\frac{2\pi \ell}{d},\tau_1,\tau_2\Big)
\text{ for all $\ell=1,\ldots,d$ and all $(\tau_1,\tau_2)$ in the range}\Big\}.
\end{align*}

The coverage probabilities of the procedure that is uniform with respect to $\lambda$ and pointwise with respect to $(\tau_1, \tau_2)$ and 
of the procedure that is uniform with respect to $(\lambda, \tau_1, \tau_2)$ for imaginary parts are defined in the same way.

\subsection{Weight functions}
The weight functions $s_1-s_5$ are defined as 
\begin{align*}
s_1(\tau_1, \tau_2):=&\sqrt{\tau_1(1-\tau_1)\tau_2(1-\tau_2)},\\
s_2(\tau_1,\tau_2):=&\max\{\tau_1,\tau_2\}-\tau_1\tau_2,\\
s_3(\tau_1,\tau_2):=&\min\{\tau_1,\tau_2\}-\tau_1\tau_2,\\
s_4(\tau_1,\tau_2):=&1,\\
s_5(\tau_1,\tau_2):=&\sqrt{s_3(\tau_1,\tau_2)}.
\end{align*}

\subsection{Detailed definitions of the models used in simulations} 
Models M0--M15 are defined, for $j=a,b,c$ and $i=a,\ldots,g$, as
\begin{align}
\label{eqnmodel0} 
&X_t \sim \mathcal{N}(0,1) \text{ i.\,i.\,d.}, \tag{M0} 
\\
\tag{M1}\label{eqnmodel1}
&X_t = 0.1 \Phi^{-1}(U_t) + 1.9(U_t-0.5) X_{t-1}, 
\\
\tag{M2}\label{eqnmodel2}
&X_t = - 0.36 X_{t-2} + \varepsilon_t, 
\\
\tag{M3}\label{eqnmodel3}
&X_t = \bigl( 1/1.9 + 0.9 X_{t-1}^2\bigr)^{1/2} \varepsilon_t, \\
\tag{M4}\label{eqnmodel4}
&X_t = \sigma_t \varepsilon_t, \textrm{ where } \sigma_t^2 = 0.01 + 0.4X_{t-1}^2 + 0.5\sigma_{t-1}^2, \\
\tag{M5}\label{eqnmodel5}
&X_t = \sigma_t \varepsilon_t , ~ \ln(\sigma_t^2) = 0.1 + 0.21(|X_{t-1}| - \E|X_{t-1}|) - 0.2X_{t-1} + 0.8\ln(\sigma_{t-1}^2),\\
\tag{M6$j$}\label{M6j}
&X_t = \phi_jX_{t-1}+\varepsilon_t ,\\
\tag{M7$j$}\label{M7j}
&X_t = \phi_jX_{t-1}+\nu_t ,\\
\tag{M8$i$}\label{M8i}
&X_t = C_1^{-1}(U_{t}|X_{t-1})\quad\text{with $\gamma_i$}, \\
\tag{M9$i$}\label{M9i}
&X_t = C_2^{-1}(U_{t}|X_{t-1})\quad\text{with $\lambda_i$}, \\
\tag{M10$i$}\label{M10i}
&X_{2t-1} = Y_{t},\ X_{2t} = Y_{t}^\prime \quad\text{with $Y_t$ and $Y_t^\prime \sim$ (M8$i$),\quad$Y_t\mathop{\perp\!\!\!\!\perp}Y_t^\prime$},\\
\tag{M11$i$}\label{M11i}
&X_{2t-1} = Y_{t},\ X_{2t} = Y_{t}^\prime \quad\text{with $Y_t$ and $Y_t^\prime \sim$ (M9$i$),\quad$Y_t\mathop{\perp\!\!\!\!\perp}Y_t^\prime$},\\
\tag{M12$j$}\label{M12j}
&X_t = C_3^{-1}(U_{t}|X_{t-1})\quad\text{with $\tau_{3j}$},\\
\tag{M13$j$}\label{M13j}
&X_t = C_4^{-1}(U_{t}|X_{t-1})\quad\text{with $\tau_{4j}$},\\
\tag{M14}\label{M14}
&X_t = C_5^{-1}(U_{t}|X_{t-1}),\\
\tag{M15}\label{M15}
&X_t = C_6^{-1}(U_{t}|X_{t-1}).
\end{align}

In~\eqref{eqnmodel1}, $(U_t)$ denotes a sequence of i.i.d. standard uniform random variables, and~$\Phi$ denotes the cdf of $\mathcal{N}(0,1)$. This model is from the class of QAR(1)~processes, which was introduced by \cite{Koenker2006}.
In~\eqref{eqnmodel2}, $(\varepsilon_t)$ denotes a sequence of standard normal white noise. This AR(2)~process was previously considered by \cite{Li2012}.
\eqref{eqnmodel3} is ARCH(1)~process previously considered by \cite{LeeRao2012}.
\eqref{eqnmodel4} and \eqref{eqnmodel5} are GARCH(1,1) and EGARCH(1,1,1) models, respectively, previously considered by \cite{BirrEtAl2019}. 
\eqref{M6j} is AR(1) model with a Gaussian innovation. The AR coefficient of this model is defined as $\phi_j:=0.3,0.5,0.7$ for $j=a,b,c$ in order. 
In \eqref{M7j}, $(\nu_t)$ denotes a sequence of i.i.d. standard Cauchy distribution. This is AR(1) model with a Cauchy innovation. 
The ordinary spectral density of \eqref{M7j} does not exist.
In \eqref{M8i} and \eqref{M9i}, $U_t$ denotes a sequence of i.i.d. standard uniform distribution. The conditional distribution function $C_j^{-1}(u|v)$ is defined, for $(U,V)$ whose joint distribution follows $C_j$, as $C_j^{-1}(u|v):={\rm P}(U\leq u|V=v)$ for~$j=1,2$. The function $C_1(u,v)$ is the asymmetric Gumbel copula, which is defined as
\begin{align*}
C_1(u,v):=u^{1-\alpha}v^{1-\beta}
\exp
\left[-
\left\{
(-\alpha\log u)^\gamma+(-\beta\log v)^\gamma
\right\}^{1/\gamma}
\right],
\end{align*}
where $(\alpha,\beta)=(1,0.5)$ and $\gamma\geq1$.
The function $C_2(u,v)$ is the zero total circulation copula, which is defined, for $\lambda\in[0,1]$, as
\begin{align*}
C_2(u,v):=\int_0^u\int_0^v\lambda+(1-\lambda)c_0(s,t){\rm d}s{\rm d}t,
\end{align*}
where
\begin{align*}
c_{2}(u,v):=
\begin{cases}
1& 
\{0\leq v<1/4,1/4\leq u< 1/2\}
\cup
\{1/4\leq v<1/2,3/4\leq u\leq 1\}
\\
&\cup
\{1/2\leq v<3/4,0\leq u< 1/4\},
\cup
\{3/4\leq v\leq1,1/2\leq u< 3/4\}\\
0& \text{otherwise},
\end{cases}
\end{align*}
and the generalized inverse $C_j^{-1}$ is calculated via a grid of 1000 points equispaced over~$[0,1]$. 
Let $\gamma_i^{-1}$ and $\lambda_i$ take values $0.15, 0.29, 0.43, 0.57, 0.71, 0.85, 0.99$ for $i=a,\ldots,g$, respectively. These models were considered by \cite{bs2014} in their simulation. When~$\gamma_i=~\!1$ and $\lambda_i=1$,  both models reduce to the product copula. Therefore, \eqref{M8i} with~$\gamma_i=1$ and \eqref{M9i} with $\lambda_i=1$ are time-reversible.  
The models \eqref{M10i} and \eqref{M11i} are designed that any time-reversibility tests based on a first-order Markov process cannot detect time-irreversibility.
In \eqref{M12j}, the function $C_3(u,v)$ is the Gumbel copula, which is defined as
\begin{align*}
C_3(u,v):=
\exp
\left[-
\left\{
(-\log u)^\gamma+(-\log v)^\gamma
\right\}^{1/\gamma}
\right],
\end{align*}
where $\gamma=1/(1-\tau_{3j})\geq1$ with Kendall's tau $\tau_{3j}$ for $C_3$. In  \eqref{M13j}, the function $C_4(u,v)$ is the Clayton copula, which is defined as
\begin{align*}
C_4(u,v):=
(u^{-\gamma}+v^{-\gamma}-1)^{-1/\gamma},
\end{align*}
where $\gamma=2\tau_{4j}/(1-\tau_{4j})>0$ with Kendall's tau $\tau_{4j}$ for $C_4$.
The parameters $\tau_{3j}$ and $\tau_{4j}$ are defined as $\tau_{3j},\tau_{4j}:=0.25, 0.5, 0.75$ for $j=a,b,c$, 
 respectively.
These models are considered by \cite{lg13} in their simulation.
In \eqref{M14},
the function  $C_5(u,v)$ is the copula 3 of \citet[Figure 1]{nelsen93}, which is  defined as
\begin{align*}
C_5(u,v):=\int_0^u\int_0^vc_5(s,t){\rm d}s{\rm d}t,
\end{align*}
where
\begin{align*}
c_{5}(u,v):=
\begin{cases}
1& 
\{0\leq v<1/4,0\leq u< 1/4\}
\cup
\{0\leq v<1/4,3/4\leq u< 1\}
\\
&\cup\{1/4\leq v<1/2,1/2\leq u\leq 1\}
\cup
\{1/2\leq v<3/4,1/4\leq u\leq 3/4\}
\\
&\cup
\{3/4\leq v<1,0\leq u< 1/2\},\\
0& \text{otherwise}.
\end{cases}
\end{align*}
In \eqref{M15},
the function $C_6(u,v)$ is the copula 6  of  \citet[Figure 1]{nelsen93}, which is defined~as
\begin{align*}
C_6(u,v):=\int_0^u\int_0^vc_6(s,t){\rm d}s{\rm d}t,
\end{align*}
where
\begin{align*}
c_{6}(u,v):=
\begin{cases}
1& 
\{0\leq v<1/2,1/4\leq u< 3/4\}\cup
\{1/2\leq v\leq 1,0\leq u< 1/4\}\\
&\cup
\{1/2\leq v\leq 1,3/4\leq u\leq 1\}
,\\
0& \text{otherwise}.
\end{cases}
\end{align*} The models M12j--M15 are not radially symmetric.

\end{appendix}

\begin{acks}[Acknowledgments]
The first three authors contributed equally to the paper and are listed alphabetically. The authors are grateful to Brendan K. Beare and Juwon Seo for kindly sharing their Matlab codes for Section \ref{sim_TR}.
\end{acks}

\begin{funding}
This work has been supported in part by the Collaborative Research Center
“Statistical modeling of nonlinear dynamic processes” (SFB 823, Teilprojekt A1,C1) of the German Research Foundation (DFG). Yuichi Goto was supported by JSPS Grant-in-Aid for Research Activity Start-up under Grant Number JP21K20338. Stanislav Volgushev was partially supported by a discovery grant from NSERC of Canada,
\end{funding}

\bibliographystyle{apalike}
\bibliography{SpectralDistribution}

\newpage

\begin{center}
\huge{Online Supplement}
\end{center}

\section{Proofs}\label{sec::proofsintspec}

\subsection{Proof of Lemma \ref{exampleGaussianprocess}}

\begin{proof}
{Throughout the proof, let $\Phi_k$ and $\Phi$ denote the cumulative distribution functions~of 
\al{
(X_k,X_0)^T\sim\mathcal{N}_2\Big(\Big(\begin{array}{c} 0\\0\end{array}\Big), \Big(\begin{array}{cc} 1 & \rho_k \\ \rho_k & 1 \end{array}\Big)\Big)
}
and $X_0\sim\mathcal{N}\big(0,1\big)$, respectively.}

Note that 
\al{
\specdisred(\lambda ;\tau_1,\tau_2)={}&\frac{1}{2\pi}\sum_{k\in\Z\backslash\{0\}}\gamma_k^U(\tau_1,\tau_2)\frac{i}{k}(e^{-ik\lambda}-1)\\
={}&\frac{1}{2\pi}\sum_{k\in\Z\backslash\{0\}}\big(C_k(\tau_1,\tau_2)-\tau_1\tau_2\big)\frac{i}{k}(e^{-ik\lambda}-1).
}
Furthermore, as $\proc$ is Gaussian, by Sklar's theorem [see \cite{sklar59}], 
\al{
C_k(\tau_1,\tau_2):=C\big(\tau_1,\tau_2,\rho(k)\big)=\Phi_k\big(\Phi^{-1}(\tau_1),\Phi^{-1}(\tau_2)\big).
}

{We first provide a bound on $\gamma_k^U(\tau_1,\tau_2) - \tau_1\tau_2$. Observe the following representation:
\begin{align*}
& \gamma_k^U(\tau_1,\tau_2) - \tau_1\tau_2  = \Phi_k\big(\Phi^{-1}(\tau_1),\Phi^{-1}(\tau_2)\big) - \tau_1\tau_2 
\\
& = \int_{-\infty}^{\Phi^{-1}(u)}\int_{-\infty}^{\Phi^{-1}(v)}\frac{1}{2\pi\sqrt{1-\rho_k^{2}}}\exp\Big(-\frac{x^2-2\rho_k x y+y^2}{2(1-\rho_k^2)}\Big) - \frac{1}{2\pi} \exp\Big(-\frac{x^2+y^2}{2}\Big)\dd x\dd y 
\\
& = \int_{-\infty}^{\Phi^{-1}(u)}\int_{-\infty}^{\Phi^{-1}(v)} h(x,y,\rho_k) - h(x,y,0) \dd x\dd y
\end{align*}
where 
\[
h(x,y,\rho) := \frac{1}{2\pi\sqrt{1-\rho^{2}}}\exp\Big(-\frac{x^2-2\rho x y+y^2}{2(1-\rho^2)}\Big).
\]
Now, from a Taylor expansion, we find 
\[
\Big| h(x,y,\rho) - h(x,y,0)\Big| \leq |\rho| \Big| \frac{\partial h(x,y,\rho) }{\partial\rho}\big|_{\rho = \kappa(x,y)} \Big|
\]
where $\kappa(x,y)$ is a value between $0$ and $\rho$. In particular, $|\kappa(x,y)| \leq |\rho|$ for any $x,y$. A straightforward calculation shows that there exists a function $H: \R^2 \to [0,\infty)$, independent of $\kappa(x,y)$ such that for all $\kappa(x,y) \in [-1/2,1/2]$  
\[
\Big| \frac{\partial h(x,y,\rho) }{\partial\rho}\big|_{\rho = \kappa(x,y)} \Big| \leq H(x,y)
\]
for all $x,y \in R^2$ and such that
\[
K := \int_\R\int_\R H(x,y) \dd x\dd y < \infty.
\]
Summarizing, we have shown that for any $\rho_k \in [-1/2,1/2]$ and any $\tau_1,\tau_2 \in (0,1)$
\[
|\gamma_k^U(\tau_1,\tau_2) - \tau_1\tau_2| \leq K |\rho_k|.
\]
Since by assumption $\sum_k |\rho_k|/|k| < \infty$, we can have $|\rho_k| > 1/2$ for at most a finite set of $k$. Thus, 
\begin{align}
&\frac{1}{2\pi}\sum_{k\in\Z \backslash \{0\}}|\gamma_k^U(\tau_1,\tau_2)\frac{i}{k}(e^{-ik\lambda}-1)|\leq \frac{1}{\pi}\sum_{k\in\Z \backslash \{0\}}|\gamma_k^U(\tau_1,\tau_2)|/|k| \notag
\\
\leq &\sum_{k: |\rho_k| \geq 1/2} |\gamma_k^U(\tau_1,\tau_2)|  + K \sum_{k: |\rho_k| < 1/2} \frac{|\rho_k|}{|k|} < \infty . \label{eq:sumpaircum}
\end{align} 
}
Next note that, employing Leibniz's integral rule, we have
\al{
&\frac{\partial \big(C(u,v;\rho)-uv\big)}{\partial u}\\
={}&\frac{\partial}{\partial u}\int_{-\infty}^{\Phi^{-1}(u)}\int_{-\infty}^{\Phi^{-1}(v)}\frac{1}{2\pi\sqrt{1-\rho^{2}}}\exp\Big(-\frac{x^2-2\rho x y+y^2}{2(1-\rho^2)}\Big)\dd x\dd y-v\\
={}&\frac{d \Phi^{-1}(u)}{du}\int_{-\infty}^{\Phi^{-1}(v)}\frac{1}{2\pi\sqrt{1-\rho^{2}}}\exp\Big(-\frac{y^2-2\rho\Phi^{-1}(u)y+(\Phi^{-1}(u))^2}{2(1-\rho^2)}\Big)\dd y-\Phi(\Phi^{-1}(v)).
}
Observe that
\al{
\frac{d \Phi^{-1}(u)}{du}=\frac{1}{\Phi^{\prime}(\Phi^{-1}(u))}=\frac{1}{\frac{1}{\sqrt{2\pi}}\exp\big(-\frac{(\Phi^{-1}(u))^2}{2}\big)}=\sqrt{2\pi}\exp\Big(\frac{(\Phi^{-1}(u))^2}{2}\Big)
}
and, by adding a square, 
\al{
\exp\Big(-\frac{y^2-2\rho\Phi^{-1}(u)y+(\Phi^{-1}(u))^2}{2(1-\rho^2)}\Big)={}&\exp\Big(-\frac{[y-\rho\Phi^{-1}(u)]^2+(1-\rho^2)(\Phi^{-1}(u))^2}{2(1-\rho^2)}\Big)\\
={}&\exp\Big(-\frac{[y-\rho\Phi^{-1}(u)]^2}{2(1-\rho^2)}\Big)\exp\Big(-\frac{(\Phi^{-1}(u))^2}{2}\Big).
}
Thus, altogether,
\al{
\frac{\partial \big(C(u,v;\rho)-uv\big)}{\partial u}={}&\int_{-\infty}^{\Phi^{-1}(v)}\frac{1}{\sqrt{2\pi(1-\rho^{2})}}\exp\Big(-\frac{[y-\rho\Phi^{-1}(u)]^2}{2(1-\rho^2)}\Big)\dd y-\Phi(\Phi^{-1}(v))\\
={}&\Phi\Big(\frac{\Phi^{-1}(v)-\rho\Phi^{-1}(u)}{\sqrt{1-\rho^2}}\Big)-\Phi\big(\Phi^{-1}(v)\big).
}

Next, let $g(u,v;\rho):=\frac{\Phi^{-1}(v)-\Phi^{-1}(u)\rho}{\sqrt{1-\rho^2}}$ and observe that $g(u,v;0)=\Phi^{-1}(v)$. The function~$\rho\mapsto \Phi(g(u,v;\rho))$ is continuous and differentiable on $(-1,1)$. Thus, by the mean value theorem, for any $\rho \in [-1+\epsilon,1-\epsilon]$ with~$0<\epsilon<1$  there exists $\rho_0$ with $|\rho_0| \leq |\rho|$  such that 
\al{
\Phi\big(g(u,v;\rho)\big)-\Phi\big(g(u,v;0)\big)= \frac{\partial \Phi\big(g(u,v;\rho)\big)}{\partial \rho}\Big|_{\rho=\rho_0}\cdot\rho.
}
Since
\al{
&\frac{\partial \Phi\big(g(u,v;\rho)\big)}{\partial \rho}\Big|_{\rho=\rho_0}=\frac{\partial g(u,v;\rho)}{\partial \rho}\Big|_{\rho=\rho_0}\frac{d \Phi(x)}{dx}\Big|_{x=g(u,v,\rho_0)}\\
={}&\frac{-\Phi^{-1}(u)\sqrt{1-\rho_0^2}+\big(\Phi^{-1}(v)-\rho_0\Phi^{-1}(u)\big)\rho_0(1-\rho_0^2)^{-1/2}}{1-\rho_0^2}\frac{1}{\sqrt{2\pi}}e^{-g^2(u,v;\rho_0)/2}
}
and, hence,
\al{
\sup_{u,v\in[\eta,1-\eta],\rho_0\in[-1+\epsilon,1-\epsilon]}\Big|\frac{\partial \Phi\big(g(u,v;\rho)\big)}{\partial \rho}\Big|_{\rho=\rho_0}\leq{}&C_{\eta},
}
we obtain
\al{
\Big|\Phi\big(g(u,v;\rho)\big)-\Phi\big(g(u,v;0)\big)\Big|\leq{}& |\rho|\Big|\frac{\partial \Phi\big(g(u,v;\rho)\big)}{\partial \rho}\Big|_{\rho=\rho_0}\leq C_{\eta}\cdot|\rho|.
}
Therefore, 
\al{
\frac{1}{2\pi}\sum_{k\in\Z\backslash\{0\}}&\sup_{\tau_1,\tau_2\in[\eta,1-\eta]}\Big|\frac{\partial \big(C(u,v;\rho(k))-uv\big)}{\partial u}\Big|_{(u,v)=(\tau_1,\tau_2)}\big|\frac{i}{k}(e^{ik\lambda}-1)\big|\\
\leq{}&\frac{1}{2\pi}\sum_{k\in\Z\backslash\{0\}}\big(C_{\eta}|\rho(k)|\big)\frac{2}{|k|}<\infty,
}
where we have used that, by assumption, $\sum_{k\in\Z\backslash\{0\}}\frac{|\rho(k)|}{|k|}<\infty$. {Combining this with~\eqref{eq:sumpaircum} }
we can apply Theorem 7.17 in \cite{rudin64} to conclude that the partial derivatives 
\[
\frac{\partial \specdisred(\lambda,u,v)}{\partial u}\Big|_{(u,v)=(\tau_1,\tau_2)}
\]
exist and are continuous on $\{(\lambda ;\tau_1,\tau_2)\in[0,\pi]\times[\eta,1-\eta]\times [\eta,1-\eta]\}$.
\end{proof}

\subsection{Proof of Theorem \ref{weakconvintegspectrum}}

We begin by deriving an alternative representation for the copula-based spectral distribution function defined in (\ref{specdisestimatorintro}) and introduce some additional notation. 

Observe that from definitions (\ref{copperintro}) and (\ref{dnRtaudef}) we can derive the following representation of the copula rank periodogram:
\als{\label{exactrepresentationofcopper}
\copper_{n,R}^{\tau_1,\tau_2}\big(\frac{2\pi s}{n}\big)={}&\frac{1}{2\pi n}d_{n,R}^{\tau_1}\big(\frac{2\pi s}{n}\big)d_{n,R}^{\tau_2}\big(-\frac{2\pi s}{n}\big)\notag\\
={}&\frac{1}{2\pi n}\sum_{t_1=0}^{n-1}I\{\hat{F}_{n}(X_{t_1})\leq \tau_1\}e^{-it_1\frac{2\pi s}{n}}\sum_{t_2=0}^{n-1}I\{\hat{F}_{n}(X_{t_2})\leq \tau_2\}e^{it_2\frac{2\pi s}{n}}.
}  
Since $\sum_{t=0}^{n-1}e^{-it2\pi s/n}=0$ for $s\not\in n\Z$, we have, for $\tau\in[0,1]$,
\als{\label{additionalterm}
\sum_{t=0}^{n-1}I\{\hat{F}_{n}(X_{t})\leq \tau\}e^{-it\frac{2\pi s}{n}}={}&\sum_{t=0}^{n-1}\big(I\{\hat{F}_{n}(X_{t})\leq \tau\}-a\big)e^{-it\frac{2\pi s}{n}},
}
where $a\in\R$ can be chosen arbitrarily. Using property (\ref{additionalterm}) in (\ref{exactrepresentationofcopper}), after rearranging sums, we obtain
\als{\label{copperdifferent}
\copper_{n,R}^{\tau_1,\tau_2}\big(\frac{2\pi s}{n}\big)={}&\frac{1}{2\pi n}\sum_{|k|\leq n-1}\sum_{t\in\mathcal{T}_k}\big(I\{\hat{F}_{n}(X_{t+k})\leq \tau_1\}-a\big)\big(I\{\hat{F}_{n}(X_{t})\leq \tau_2\}-b\big)e^{-ik\frac{2\pi s}{n}},
}
with arbitrary $a,b\in\R$,
\al{
\mathcal{T}_k:=\{t\in\{0,\dots,n-1\}|t,t+k\in\{0,\dots,n-1\}\}
}
and $k\in\{-(n-1),\dots,n-1\}$. Next, using (\ref{copperdifferent}) in the definition of the estimator of the spectral distribution function (\ref{specdisestimatorintro}) and rearranging sums  yields
\al{
\widehat{\specdis}\phantom{F\!\!\!\!\!}_{n,R}(\lambda ;\tau_1,\tau_2)
={}&\frac{1}{2\pi}\sum_{|k|\leq n-1}\frac{2\pi}{n}\sum_{s=1}^{n-1}I\big\{0\leq \frac{2\pi s}{n}\leq\lambda\big\}e^{-ik\frac{2\pi s}{n}}\frac{n-|k|}{n}\\
&\cdot\frac{1}{n-|k|}\sum_{t\in\mathcal{T}_k}\big(I\{\hat{F}_{n}(X_{t+k})\leq \tau_1\}-a\big)\big(I\{\hat{F}_{n}(X_{t})\leq \tau_2\}-b\big).
}
Define the weights
\als{\label{defweightsquantspec}
w_{n,\lambda}(k):=\frac{2\pi}{n}\sum_{s=1}^{n-1}I\big\{0\leq \frac{2\pi s}{n}\leq\lambda\big\}e^{-ik\frac{2\pi s}{n}}
}
and the \textit{rank-based copula cumulant function of order $k$}
\als{\label{gammahatranksdef}
\hat{\gamma}_k^R(\tau_1,\tau_2)={}&\frac{1}{n-|k|}\sum_{t\in\mathcal{T}_k}\Big(I\{\hat{F}_{n}(X_{t+k})\leq \tau_1\}-a\Big)\Big(I\{\hat{F}_{n}(X_{t})\leq \tau_2\}-b\Big).
}
Then, we obtain
\als{\label{lagrepresspecdisemp}
\widehat{\specdis}\phantom{F\!\!\!\!\!}_{n,R}(\lambda ;\tau_1,\tau_2):={}&\frac{1}{2\pi}\sum_{|k|\leq n-1}w_{n,\lambda}(k)\frac{n-|k|}{n}\hat{\gamma}_k^R(\tau_1,\tau_2)\notag\\
={}&\frac{1}{2\pi}\sum_{0<|k|\leq n-1}w_{n,\lambda}(k)\frac{n-|k|}{n}\hat{\gamma}_k^R(\tau_1,\tau_2)+\frac{1}{2\pi}w_{n,\lambda}(0)\hat{\gamma}_0^R(\tau_1,\tau_2)
}
as an alternative representation of the estimator of the copula spectral distribution function.
Similarly, the copula spectral distribution function has the alternative representation
\al{
\specdis(\lambda ;\tau_1,\tau_2)={}&\frac{1}{2\pi}\sum_{k\in\Z \backslash \{0\}}\gamma_k^U(\tau_1,\tau_2)\frac{i}{k}(e^{-ik\lambda}-1)+\frac{\lambda}{2\pi}(\tau_1\wedge\tau_2-\tau_1\tau_2).
}

In the subsequent analysis, we  sometimes will  consider versions of $\widehat{\specdis}\phantom{F\!\!\!\!\!}_{n,R}(\lambda ;\tau_1,\tau_2)$ and~$\specdis(\lambda ;\tau_1,\tau_2)$, where the terms corresponding to lag $0$ are removed, that is, 
\al{
\hat{\specdisred}_{n,R}(\lambda ;\tau_1,\tau_2):=\frac{1}{2\pi}\sum_{0<|k|\leq n-1}\frac{2\pi}{n}\sum_{s=1}^{n-1}I\big\{0\leq \frac{2\pi s}{n}\leq\lambda\big\}e^{-ik\frac{2\pi s}{n}}\frac{n-|k|}{n}\hat{\gamma}_k^R(\tau_1,\tau_2)
}
and $\specdisred(\lambda ;\tau_1,\tau_2)$ as defined in~\eqref{specdiswithoutzero}. Also, in the analysis of the asymptotic properties, instead of the process $$\{\widehat{\specdis}\phantom{F\!\!\!\!\!}_{n,R}(\lambda ;\tau_1,\tau_2)\}_{(\lambda ;\tau_1,\tau_2)\in [0,\pi]\times[0,1]^2},$$ we   often prove intermediate results for the process $$\{\widehat{\specdis}\phantom{F\!\!\!\!\!}_{n,U}(\lambda ;\tau_1,\tau_2)\}_{(\lambda ;\tau_1,\tau_2)\in [0,\pi]\times[0,1]^2},\label{sym:estcopspecdisU}$$ where $\widehat{\specdis}\phantom{F\!\!\!\!\!}_{n,U}(\lambda ;\tau_1,\tau_2)$ is defined exactly as $\widehat{\specdis}\phantom{F\!\!\!\!\!}_{n,R}(\lambda ;\tau_1,\tau_2)$ but with the actual distributions function $F$ replacing the 
empirical one  $\hat{F}_{n}$. More precisely, in order to prove the weak convergence of the process $\{\widehat{\specdis}\phantom{F\!\!\!\!\!}_{n,R}(\lambda ;\tau_1,\tau_2)\}_{(\lambda ;\tau_1,\tau_2)\in [0,\pi]\times[\eta,1-\eta]^2}$ for $0<\eta<1/2$, we derive, according to Lemma 2.2.2 in \cite{vandervaart96}, the stochastic equicontinuity for the process $\{\widehat{\specdis}\phantom{F\!\!\!\!\!}_{n,U}(\lambda ;\tau_1,\tau_2)\}_{(\lambda ;\tau_1,\tau_2)\in [0,\pi]\times[\eta,1-\eta]^2}$. The impact of replacing the true distribution functions $F$ by the empirical versions $\hat{F}_n$ in $\{\widehat{\specdis}\phantom{F\!\!\!\!\!}_{n,R}(\lambda ;\tau_1,\tau_2)\}_{(\lambda ;\tau_1,\tau_2)\in [0,\pi]\times[\eta,1-\eta]^2}$ is then seen in the derivation of the covariance structure of the  limiting process.\\

We are now ready to start with the main proof. We first prove several intermediate results for the process
\al{
\mathbb{G}_{n,U}(\lambda ;\tau_1,\tau_2):=\sqrt{n}\Big(\widehat{\specdis}\phantom{F\!\!\!\!\!}_{n,U}(\lambda ;\tau_1,\tau_2)-\specdis(\lambda ;\tau_1,\tau_2)\Big)
}
indexed by $(\lambda ;\tau_1,\tau_2)\in [0,\pi]\times[\eta,1-\eta]^2$, where
\al{
\widehat{\specdis}\phantom{F\!\!\!\!\!}_{n,U}(\lambda ;\tau_1,\tau_2):={}&\frac{2\pi}{n}\sum_{s=1}^{n-1}I\big\{0\leq \frac{2\pi s}{n}\leq\lambda\big\}\copper_{n,U}^{\tau_1,\tau_2}\big(\frac{2\pi s}{n}\big),
}
with $U_t:=F(X_t)$ and
\al{
\copper_{n,U}^{\tau_1,\tau_2}\left(\omega\right)=(2\pi n)^{-1}d_{n,U}^{\tau_1}(\omega)d_{n,U}^{\tau_2}(-\omega),\qquad d_{n,U}^{\tau}(\omega):=\sum_{t=0}^{n-1}I\{U_t\leq \tau\}e^{-i\omega t}.
}
As for $\widehat{\specdis}\phantom{F\!\!\!\!\!}_{n,R}(\lambda ;\tau_1,\tau_2)$, we have
\al{
\widehat{\specdis}\phantom{F\!\!\!\!\!}_{n,U}(\lambda ;\tau_1,\tau_2)={}&\frac{1}{2\pi}\sum_{0<|k|\leq n-1}w_{n,\lambda}(k)\frac{n-|k|}{n}\hat{\gamma}_k^U(\tau_1,\tau_2)+\frac{1}{2\pi}w_{n,\lambda}(0)\hat{\gamma}_0^U(\tau_1,\tau_2),
}
where $w_{n,\lambda}(k)$ is defined in (\ref{defweightsquantspec}),
\als{\label{empautocopU}
\hat{\gamma}_k^U(\tau_1,\tau_2)={}&\frac{1}{n-|k|}\sum_{t\in\mathcal{T}_k}\Big(I\{U_{t+k}\leq \tau_1\}-a\Big)\Big(I\{U_t\leq \tau_2\}-b\Big)
}
with  $\mathcal{T}_k:=\{t\in\{0,\dots,n-1\}|t,t+k\in\{0,\dots,n-1\}\}$, $k\in\{-(n-1),\dots,n-1\}$,  where~$a,b\in\R$ can be chosen arbitrarily since $\sum_{t=0}^{n-1}e^{-it2\pi s/n} =0$ for $s\not\in n\Z$.

Finally, as for the rank-based versions, we define
\al{
\hat{\specdisred}_{n,U}(\lambda ;\tau_1,\tau_2):=\frac{1}{2\pi}\sum_{0<|k|\leq n-1}\frac{2\pi}{n}\sum_{s=1}^{n-1}I\big\{0\leq \frac{2\pi s}{n}\leq\lambda\big\}e^{-ik\frac{2\pi s}{n}}\frac{n-|k|}{n}\hat{\gamma}_k^U(\tau_1,\tau_2)
}
and
\al{
\specdisred(\lambda ;\tau_1,\tau_2):=\frac{1}{2\pi}\sum_{k\in\Z \backslash \{0\}}\gamma_k^U(\tau_1,\tau_2)\frac{i}{k}(e^{-ik\lambda}-1),
}
where the terms corresponding to lag $k=0$ have been removed.

\subsubsection{Proof of Theorem \ref{weakconvintegspectrum} -- Main arguments}\label{proofweakconvintspecoverview}

The proof of Theorem \ref{weakconvintegspectrum} is rather technical and consists of a series of lemmas and intermediate results. To facilitate the reading we give an overview of the most important arguments of the proof. \\

For all $n\in\N$, consider the stochastic process
\als{
\mathbb{G}_{n,R}(\lambda ;\tau_1,\tau_2):=\sqrt{n}\Big(\widehat{\specdis}\phantom{F\!\!\!\!\!}_{n,R}(\lambda ;\tau_1,\tau_2)-\specdis(\lambda ;\tau_1,\tau_2)\Big)
}
indexed by $(\lambda ;\tau_1,\tau_2)\in [0,\pi]\times[\eta,1-\eta]^2$. Observe that since $F$ is assumed to be continuous, the ranks of $X_0,\dots,X_{n-1}$ are almost surely the same as the ranks of $U_0,\dots,U_{n-1}$, 
 i.e., without loss of generality,  we can assume the marginals to be uniformly distributed. In what follows, let $\hat{F}_{n,U}$ denote the empirical distribution function of $U_0,\dots,U_{n-1}$. With~$\hat{\tau}:=\hat{F}_{n,U}^{-1}(\tau)$, we have, by Lemma \ref{replacetaubytauhatlem} (the proof  of which is deferred to Section~\ref{sec::auxiliaryresultsquantspec}),
\als{\label{replacetaubytauhat}
\widehat{\specdis}\phantom{F\!\!\!\!\!}_{n,R}(\lambda ;\tau_1,\tau_2)=\widehat{\specdis}\phantom{F\!\!\!\!\!}_{n,U}(\lambda,\hat{\tau}_1,\hat{\tau}_2)+o_{\Prob}(n^{-1/2}).
} 
Furthermore, by Lemma \ref{withoutlagzero} (which is also proved in Section~\ref{sec::auxiliaryresultsquantspec}),
\al{
\mathbb{G}_{n,R}(\lambda ;\tau_1,\tau_2)={}&\sqrt{n}\Big(\hat{\specdisred}_{n,R}(\lambda ;\tau_1,\tau_2)-\specdisred (\lambda ;\tau_1,\tau_2)\Big)+o_{\Prob}(1),\\
\mathbb{G}_{n,U}(\lambda ;\tau_1,\tau_2)={}&\sqrt{n}\Big(\hat{\specdisred}_{n,U}(\lambda ;\tau_1,\tau_2)-\specdisred (\lambda ;\tau_1,\tau_2)\Big)+o_{\Prob}(1),
}
and, therefore, we have the decomposition
\als{\label{zerlegungGnR}
\mathbb{G}_{n,R}(\lambda ;\tau_1,\tau_2)={}&\mathbb{G}_{n,U}(\lambda,\hat{\tau}_1,\hat{\tau}_2)-\mathbb{G}_{n,U}(\lambda ;\tau_1,\tau_2)\notag\\
&+\mathbb{G}_{n,U}(\lambda ;\tau_1,\tau_2)+\sqrt{n}\Big(\specdisred(\lambda,\hat{\tau}_1,\hat{\tau}_2)-\specdisred(\lambda ;\tau_1,\tau_2)\Big)+o_{\Prob}(1)\notag\\
={}&\mathbb{G}_{n,U}(\lambda,\hat{\tau}_1,\hat{\tau}_2)-\mathbb{G}_{n,U}(\lambda ;\tau_1,\tau_2)+\mathbb{G}_{n,U}(\lambda ;\tau_1,\tau_2)\notag\\
&+\sqrt{n}\Big((\hat{\tau}_1-\tau_1)\frac{\partial \specdisred}{\partial \tau_1}(\lambda ;\tau_1,\tau_2)+(\hat{\tau}_2-\tau_2)\frac{\partial \specdisred}{\partial \tau_2}(\lambda ;\tau_1,\tau_2)\Big)\notag\\
&+o_{\Prob}(1)
}
where, by Assumption (D),
\al{
\sqrt{n}\Big(\specdisred(\lambda,\hat{\tau}_1,\hat{\tau}_2)-\specdisred(\lambda ;\tau_1,\tau_2)\Big)={}\sqrt{n}\sum_{j=1}^2(\hat{\tau}_j-\tau_j)\frac{\partial \specdisred}{\partial \tau_j}(\lambda ;\tau_1,\tau_2)+o_{\Prob}(1),
}
as , by Lemma A.5 in \cite{DetteEtAl2016},  
\[
\sup_{\tau\in[0,1]}|\hat{F}_{n,U}^{-1}(\tau)-\tau|=O_{\Prob}(n^{-1/2}).
\]
Moreover, noting that $\sqrt{n}\big(\hat{F}_{n,U}(\tau)-\tau\big)$ converges to a tight Gaussian limit with continuous sample paths [see the proof of Lemma A.5 in \cite{kleysupp16}], we obtain under the given assumptions by Vervaat's Lemma [see \cite{vervaat72}], 
\als{\label{bahadur}
\hat{\tau}_j-\tau_j=-\Big(\hat{F}_{n,U}(\tau_j)-\tau_j\Big)+o_{\Prob}(n^{-1/2}).
}

Substituting (\ref{bahadur}) into (\ref{zerlegungGnR}) yields the decomposition 
\als{\label{decompositionofGnR}
\mathbb{G}_{n,R}(\lambda ;\tau_1,\tau_2)={}&\mathbb{G}_{n,U}(\lambda,\hat{\tau}_1,\hat{\tau}_2)-\mathbb{G}_{n,U}(\lambda ;\tau_1,\tau_2)+\mathbb{G}_{n,U}(\lambda ;\tau_1,\tau_2)\notag\\
&+\sqrt{n}\sum_{j=1}^2(\tau_j-\hat{F}_{n,U}(\tau_j))G_j(\lambda ;\tau_1,\tau_2)+o_{\Prob}(1),
}

where $G_j(\lambda ;\tau_1,\tau_2):=\frac{\partial \specdisred}{\partial \tau_j}(\lambda ;\tau_1,\tau_2);\,j=1,2$. \\


As a second step, to prove the weak convergence of $\mathbb{G}_{n,R}$, it suffices, by Lemmas 1.5.4 and~1.5.7 in \cite{vandervaart96}, to show that the finite-dimensional distributions converge in distribution and to prove stochastic equicontinuity. That is, we need to establish
\bi
\item[(i)] the convergence of the finite-dimensional distributions of the process (\ref{mainprocess}), i.e.
\als{\label{fidi}
\Big(\mathbb{G}_{n,R}(\lambda_j,\tau_1^{(j)},\tau_2^{(j)})\Big)_{j=1,\dots,L}\cid \Big(\mathbb{G}(\lambda_j,\tau_1^{(j)},\tau_2^{(j)})\Big)_{j=1,\dots,L}
}
for any $(\lambda_j,\tau_1^{(j)},\tau_2^{(j)})\in[0,\pi]\times[\eta,1-\eta]^2,\,j=1,\dots,L$ and $L\in\N$ and 
\item[(ii)] stochastic equicontinuity, i.e., for all $x>0$, 
\als{\label{asympequicont}
\lim_{\delta\downarrow 0}\limsup_{n\rightarrow \infty}\Prob\Big(\sup_{\substack{(\lambda ;\tau_1,\tau_2),(\lambda^{\prime},\tau_1^{\prime},\tau_2^{\prime})\in [0,\pi]\times[\eta,1-\eta]^2\\ \llVert(\lambda ;\tau_1,\tau_2)-(\lambda^{\prime},\tau_1^{\prime},\tau_2^{\prime})\rrVert_1\leq \delta}} |\mathbb{G}_{n,R}(\lambda ;\tau_1,\tau_2)-\mathbb{G}_{n,R}(\lambda^{\prime},\tau_1^{\prime},\tau_2^{\prime})|>x\Big)=0.
}
\ei

We start by proving the stochastic equicontinuity (\ref{asympequicont}). In regard of equation (\ref{decompositionofGnR}), our proof consists of three steps:
\begin{itemize}
\item establish the stochastic equicontinuity of $\big(\mathbb{G}_{n,U}(\lambda ;\tau_1,\tau_2)\big)_{(\lambda ;\tau_1,\tau_2)\in [0,\pi]\times[\eta,1-\eta]^2}$;
\item   establish the stochastic equicontinuity of $\sqrt{n}(\hat{F}_{n,U}(\tau)-\tau)_{\tau\in[0,1]}$;
\item   show that 
\als{\label{diffargandargprime}
\sup_{(\lambda ;\tau_1,\tau_2)\in [0,\pi]\times[\eta,1-\eta]^2}\big|\mathbb{G}_{n,U}(\lambda,\hat{\tau}_1,\hat{\tau}_2)-\mathbb{G}_{n,U}(\lambda ;\tau_1,\tau_2)\big|=o_{\Prob}(1) .
}
\end{itemize}

The   assertion in the second step  has been established in \cite{DetteEtAl2016} and  the third step follows from the first  one   (see Section \ref{vanishingpartsec}). For simplicity of notation, introduce~$a:=(\lambda ;\tau_1,\tau_2)$ and $b:=(\lambda^{\prime},\tau_1^{\prime},\tau_2^{\prime})\in[0,\pi]\times[\eta,1-\eta]^2$. The main part in the proof of  the stochastic equicontinuity of $\mathbb{G}_{n,U}(\lambda ;\tau_1,\tau_2)$
 is the establishment of a uniform bound on the increments of the process $\mathbb{G}_{n,U}$. The derivation of this bound relies on two intermediate bounds. First, we need a general bound on the moments of $\mathbb{G}_{n,U}(a)-\mathbb{G}_{n,U}(b)$ which is obtained in Lemma \ref{lemSixthMomIncrementHn}. Second, we provide in Lemma \ref{LemmaA.7neu} a sharper bound on the same   increments 
  when $a$ and $b$ are ``close.''\\

We now turn to the proof of the weak convergence of the finite-dimensional distributions~(\ref{fidi}). From 
 (\ref{diffargandargprime}), we have
\al{
\mathbb{G}_{n,R}(\lambda ;\tau_1,\tau_2)={}&\mathbb{G}_{n,U}(\lambda ;\tau_1,\tau_2)+\sqrt{n}\sum_{j=1}^2(\tau_j-\hat{F}_{n,U}(\tau_j))G_j(\lambda ;\tau_1,\tau_2)+o_{\Prob}(1)
} 
and hence, it suffices to show the convergence of the finite-dimensional distributions of the process
\al{
\mathbb{K}_n(\lambda ;\tau_1,\tau_2):=\mathbb{G}_{n,U}(\lambda ;\tau_1,\tau_2)+\sqrt{n}\sum_{j=1}^2(\tau_j-\hat{F}_{n,U}(\tau_j))G_j(\lambda ;\tau_1,\tau_2)
}
indexed by $(\lambda ;\tau_1,\tau_2)\in[0,\pi]\times[\eta,1-\eta]^2$. By Lemma P4.5 of \cite{brillinger75}, it suffices to prove that for any $\lambda_1,\dots,\lambda_J\in[0,\pi]$, $J\in\N$ and any $\tau_1^{(1)},\dots,\tau_1^{(J)},\tau_2^{(1)},\dots,\tau_2^{(J)}\in[\eta,1-\eta]$, the cumulants of the vector
\al{
\Big(\mathbb{K}_n(\lambda_1,\tau_1^{(1)},\tau_2^{(1)}),\overline{\mathbb{K}_n(\lambda_1,\tau_1^{(1)},\tau_2^{(1)})},\dots,\mathbb{K}_n(\lambda_L,\tau_1^{(J)},\tau_2^{(J)}),\overline{\mathbb{K}_n(\lambda_J,\tau_1^{(J)},\tau_2^{(J)})}\Big)
}
converge to the corresponding cumulants of the vector
\al{
\Big(\mathbb{G}(\lambda_1,\tau_1^{(1)}\tau_2^{(1)}),\overline{\mathbb{G}(\lambda_1,\tau_1^{(1)}\tau_2^{(1)})},\dots,\mathbb{G}(\lambda_J,\tau_1^{(J)}\tau_2^{(J)}),\overline{\mathbb{G}(\lambda_J,\tau_1^{(J)}\tau_2^{(J)})}\Big).
}

To this end, we proceed again in three steps:
\begin{itemize}
\item show that the first-order moments of $\mathbb{K}_n(\lambda_1 ;\tau_1,\tau_2)$ vanish;
\item show that the second-order moments yield the asymptotic covariance structure (\ref{covasympintspec});
\item show that the moments of order greater than two vanish.
\end{itemize}
The first assertion is proved in Lemma \ref{expF-F};    detailed  proofs of the second and third ones can be found in Section \ref{sec::detailsfidi}.

In the remaining part of this section, we present the proofs of (\ref{fidi}) and (\ref{asympequicont}), where technical details are deferred to Section \ref{technicaldetailsquantspec}.

\subsubsection{Proof of (\ref{asympequicont}) -- stochastic equicontinuity}\label{proofequicont}
%
Assertion (\ref{diffargandargprime}) mainly follows by the stochastic equicontinuity of $\big(\mathbb{G}_{n,U}(\lambda ;\tau_1,\tau_2)\big)_{(\lambda ;\tau_1,\tau_2)\in[0,\pi]\times[\eta,1-\eta]^2}$ which will be proved in the rest of this section. Details of the proof of (\ref{diffargandargprime}) can be found in Section~\ref{vanishingpartsec}.\\

We now prove the stochastic equicontinuity of $\big(\mathbb{G}_{n,U}(\lambda ;\tau_1,\tau_2)\big)_{(\lambda ;\tau_1,\tau_2)\in[0,\pi]\times[\eta,1-\eta]^2}$. By Lemma~\ref{expF-F}, it suffices to consider the process 
\als{\label{auxiliaryprocessdef}
\Big(\overline{\mathbb{G}}_{n,U}(\lambda ;\tau_1,\tau_2)\Big)_{(\lambda ;\tau_1,\tau_2)\in[0,\pi]\times[\eta,1-\eta]^2}:={}&\Big(\sqrt{n}(\widehat{\specdis}\phantom{F\!\!\!\!\!}_{n,U}(\lambda ;\tau_1,\tau_2)\notag\\
&\hspace{1.4cm}-\E[\widehat{\specdis}\phantom{F\!\!\!\!\!}_{n,U}(\lambda ;\tau_1,\tau_2)])\Big)_{(\lambda ;\tau_1,\tau_2)\in[0,\pi]\times[\eta,1-\eta]^2}
}
and we need to prove that ,for all $x>0$,
\al{
\lim_{\delta\downarrow 0}\limsup_{n\rightarrow \infty}\Prob\Big(\sup_{\substack{(\lambda ;\tau_1,\tau_2),(\lambda^{\prime},\tau_1^{\prime},\tau_2^{\prime})\in [0,\pi]\times[\eta,1-\eta]^2\\ \llVert(\lambda ;\tau_1,\tau_2)-(\lambda^{\prime},\tau_1^{\prime},\tau_2^{\prime})\rrVert_1\leq \delta}} |\overline{\mathbb{G}}_{n,U}(\lambda ;\tau_1,\tau_2)-\overline{\mathbb{G}}_{n,U}(\lambda^{\prime},\tau_1^{\prime},\tau_2^{\prime})|>x\Big)=0.
}

This will be achieved by applying Lemma A.1 from \cite{DetteEtAl2016} to the process  (\ref{auxiliaryprocessdef}). Therefore, we will prove in Section \ref{assumpA.1Tobi} that the assumptions for that lemma are fulfilled with the metric 
\[
d\big((\lambda ;\tau_1,\tau_2),(\lambda^{\prime},\tau_1^{\prime},\tau_2^{\prime})\big) := \llVert(\lambda ;\tau_1,\tau_2)-(\lambda^{\prime},\tau_1^{\prime},\tau_2^{\prime}) \rrVert_1^{\gamma/2}
\] 
for a $\gamma > 0$ that will be specified in the proof of \eqref{assumpA.1Tobiequation}. More precisely, for all $(\lambda ;\tau_1,\tau_2),(\lambda^{\prime},\tau_1^{\prime},\tau_2^{\prime})$ in~$[0,\pi]\times[\eta,1-\eta]^2$ with $d\big((\lambda ;\tau_1,\tau_2),(\lambda^{\prime},\tau_1^{\prime},\tau_2^{\prime})\big)\geq \bar{\eta}/2\geq 0$, we have 
\als{\label{assumpA.1Tobiequation}
\llVert \overline{\mathbb{G}}_{n,U}(\lambda ;\tau_1,\tau_2)-\overline{\mathbb{G}}_{n,U}(\lambda^{\prime},\tau_1^{\prime},\tau_2^{\prime})\rrVert_{\Psi}
\leq{}&Kd\big((\lambda ;\tau_1,\tau_2),(\lambda^{\prime},\tau_1^{\prime},\tau_2^{\prime})\big) 
}
where $\Psi$ denotes the Orlicz norm $\llVert X\rrVert_{\Psi}:=\inf \{C>0:\E[\Psi(|X|/C)]\leq 1\}\label{sym:orlicz}$.\\

In particular, (\ref{assumpA.1Tobiequation}) holds for $\Psi(x):=x^{8}$, i.e.\  $L=4$. Denoting by $D(\epsilon,d)$ the packing number of $T:=([0,\pi]\times [\eta,1-\eta]^2,d)$ [cf. \cite{vandervaart96}, page  98], we have $D(\epsilon,d)\asymp \epsilon^{-6/\gamma}$. Therefore, by Lemma A.1 in \cite{DetteEtAl2016}, for all~$x,\delta>0$ and all~$\tilde\eta\geq \bar{\eta}_n$, there exists a random variable $S_1$ and a constant $K<\infty$ such that, \linebreak for~$s:=(\lambda ;\tau_1,\tau_2)$ and~$t:=(\lambda^{\prime},\tau_1^{\prime},\tau_2^{\prime})$,
\al{
\sup_{d(s,t)\leq \delta}|\overline{\mathbb{G}}_{n,U}(s)-\overline{\mathbb{G}}_{n,U}(t)|\leq{}&S_1+2\sup_{d(s,t)\leq \bar{\eta}_n,t\in\tilde{T}}|\overline{\mathbb{G}}_{n,U}(s)-\overline{\mathbb{G}}_{n,U}(t)|
}
with
\al{
\llVert S_1\rrVert_{\Psi}\leq K\Big[\int_{\bar{\eta}_n/2}^{\tilde\eta}\Psi^{-1}\Big(D(\epsilon,d)\Big)\dd\epsilon+(\delta+2\bar{\eta}_n)\Psi^{-1}\Big(D^2(\tilde\eta,d)\Big)\Big]
}
where the set $\tilde{T}$ contains at most $D(\bar{\eta}_n,d)$ points. In particular, by Markov's inequality [cf. \cite{vandervaart96}, page 96], 
\al{
\Prob(|S_1|>x)\leq \Big(\Psi\Big(x[8K\Big(\int_{\bar{\eta}_n/2}^{\tilde\eta}\Psi^{-1}\Big(D(\epsilon,d)\Big)\dd\epsilon+(\delta+2\bar{\eta}_n)\Psi^{-1}\Big(D^2(\tilde\eta,d)\Big)\Big)\Big]^{-1}\Big)\Big)^{-1}.
}

Hence,
\al{
&\Prob\Bigg(\sup_{\substack{(\lambda ;\tau_1,\tau_2),(\lambda^{\prime},\tau_1^{\prime},\tau_2^{\prime})\in [0,\pi]\times[\eta,1-\eta]^2\\ \llVert(\lambda ;\tau_1,\tau_2)-(\lambda^{\prime},\tau_1^{\prime},\tau_2^{\prime})\rrVert_1\leq \delta^{2/\gamma}}} |\overline{\mathbb{G}}_{n,U}(\lambda ;\tau_1,\tau_2)-\overline{\mathbb{G}}_{n,U}(\lambda^{\prime},\tau_1^{\prime},\tau_2^{\prime})|>x\Bigg)\\
={}&\Prob\Bigg(\sup_{\substack{(\lambda ;\tau_1,\tau_2),(\lambda^{\prime},\tau_1^{\prime},\tau_2^{\prime})\in [0,\pi]\times[\eta,1-\eta]^2\\ d((\lambda ;\tau_1,\tau_2),(\lambda^{\prime},\tau_1^{\prime},\tau_2^{\prime}))\leq \delta}} |\overline{\mathbb{G}}_{n,U}(\lambda ;\tau_1,\tau_2)-\overline{\mathbb{G}}_{n,U}(\lambda^{\prime},\tau_1^{\prime},\tau_2^{\prime})|>x\Bigg)\\
\leq{}&\Prob\Bigg(S_1+2\sup_{\substack{(\lambda ;\tau_1,\tau_2),(\lambda^{\prime},\tau_1^{\prime},\tau_2^{\prime})\in [0,\pi]\times[\eta,1-\eta]^2\\ d((\lambda ;\tau_1,\tau_2),(\lambda^{\prime},\tau_1^{\prime},\tau_2^{\prime}))\leq \bar{\eta}_n}} |\overline{\mathbb{G}}_{n,U}(\lambda ;\tau_1,\tau_2)-\overline{\mathbb{G}}_{n,U}(\lambda^{\prime},\tau_1^{\prime},\tau_2^{\prime})|>x\Bigg)\\
\leq{}&\Prob(|S_1|>x/2)+\Prob\Bigg(\sup_{\substack{(\lambda ;\tau_1,\tau_2),(\lambda^{\prime},\tau_1^{\prime},\tau_2^{\prime})\in [0,\pi]\times[\eta,1-\eta]^2\\ d((\lambda ;\tau_1,\tau_2),(\lambda^{\prime},\tau_1^{\prime},\tau_2^{\prime}))\leq \bar{\eta}_n}} |\overline{\mathbb{G}}_{n,U}(\lambda ;\tau_1,\tau_2)-\overline{\mathbb{G}}_{n,U}(\lambda^{\prime},\tau_1^{\prime},\tau_2^{\prime})|>x/4\Bigg)\\
\leq{}&\Big(\Big(\frac{x}{2}[8K\Big(\int_{\bar{\eta}_n/2}^{\tilde\eta}(C_1\epsilon^{-6/\gamma})^{1/8}\dd\epsilon+(\delta+2\bar{\eta}_n)(C_2\tilde\eta^{-12/\gamma})^{1/8}\Big)\Big]^{-1}\Big)^8\Big)^{-1}\\
&+\Prob\Bigg(\sup_{\substack{(\lambda ;\tau_1,\tau_2),(\lambda^{\prime},\tau_1^{\prime},\tau_2^{\prime})\in [0,\pi]\times[\eta,1-\eta]^2\\ d((\lambda ;\tau_1,\tau_2),(\lambda^{\prime},\tau_1^{\prime},\tau_2^{\prime}))\leq \bar{\eta}_n}} |\overline{\mathbb{G}}_{n,U}(\lambda ;\tau_1,\tau_2)-\overline{\mathbb{G}}_{n,U}(\lambda^{\prime},\tau_1^{\prime},\tau_2^{\prime})|>x/4\Bigg)\\
\leq{}&\Big[\frac{8\bar{K}}{x/2}\Big(\int_{\bar{\eta}_n/2}^{\tilde\eta}\epsilon^{-3/(4\gamma)}\dd\epsilon+(\delta+2\bar{\eta}_n)\tilde\eta^{-3/(2\gamma)}\Big)\Big]^8\\
&+\Prob\Bigg(\sup_{\substack{(\lambda ;\tau_1,\tau_2),(\lambda^{\prime},\tau_1^{\prime},\tau_2^{\prime})\in [0,\pi]\times[\eta,1-\eta]^2\\ d((\lambda ;\tau_1,\tau_2),(\lambda^{\prime},\tau_1^{\prime},\tau_2^{\prime}))\leq \bar{\eta}_n}} |\overline{\mathbb{G}}_{n,U}(\lambda ;\tau_1,\tau_2)-\overline{\mathbb{G}}_{n,U}(\lambda^{\prime},\tau_1^{\prime},\tau_2^{\prime})|>x/4\Bigg).
}
Now choose $1>\gamma>3/4$. Letting $n$ tend to infinity, the second term  equals 
\al{
&\Prob\Bigg(\sup_{\substack{(\lambda ;\tau_1,\tau_2),(\lambda^{\prime},\tau_1^{\prime},\tau_2^{\prime})\in [0,\pi]\times[\eta,1-\eta]^2\\ \llVert(\lambda ;\tau_1,\tau_2)-(\lambda^{\prime},\tau_1^{\prime},\tau_2^{\prime})\rrVert_1\leq 2^{2/\gamma}n^{-1/\gamma}}} |\overline{\mathbb{G}}_{n,U}(\lambda ;\tau_1,\tau_2)-\overline{\mathbb{G}}_{n,U}(\lambda^{\prime},\tau_1^{\prime},\tau_2^{\prime})|>x/4\Bigg)
}
and converges to $0$ by Lemma \ref{LemmaA.7neu}. Hence,
\al{
&\lim_{\delta\downarrow 0}\limsup_{n\rightarrow \infty}\Prob\Bigg(\sup_{\substack{(\lambda ;\tau_1,\tau_2),(\lambda^{\prime},\tau_1^{\prime},\tau_2^{\prime})\in [0,\pi]\times[\eta,1-\eta]^2\\ \llVert(\lambda ;\tau_1,\tau_2)-(\lambda^{\prime},\tau_1^{\prime},\tau_2^{\prime})\rrVert_1\leq \delta^{2/\gamma}}} |\overline{\mathbb{G}}_{n,U}(\lambda ;\tau_1,\tau_2)-\overline{\mathbb{G}}_{n,U}(\lambda^{\prime},\tau_1^{\prime},\tau_2^{\prime})|>x\Bigg)\\
\leq{}&\lim_{\delta\downarrow 0}\Big[\frac{8\bar{K}}{x}\Big(\int_{0}^{\tilde\eta}\epsilon^{-3/(4\gamma)}\dd\epsilon+\delta\tilde\eta^{-3/(2\gamma)}\Big)\Big]^8\\
\leq{}&\Big[\frac{8\bar{K}}{x}\int_{0}^{\tilde\eta}\epsilon^{-3/(4\gamma)}\dd\epsilon\Big]^8
}
for every $x,\tilde\eta>0$. Since, $\tilde\eta$ can be chosen arbitrarily small, the integral can be made arbitrarily small and (\ref{asympequicont}) follows.

\subsubsection{Proof of (\ref{fidi}) -- convergence of the finite-dimensional distributions}
In view of (\ref{decompositionofGnR}) and (\ref{diffargandargprime}), it suffices to prove that the finite-dimensional distributions of 
\al{
\mathbb{K}_n(\lambda ;\tau_1,\tau_2):=\mathbb{G}_{n,U}(\lambda ;\tau_1,\tau_2)+\sqrt{n}\sum_{j=1}^2(\tau_j-\hat{F}_{n,U}(\tau_j))G_j(\lambda ;\tau_1,\tau_2)
}
converge, i.e.\  that
\al{
\Big(\mathbb{K}_n(\lambda_j,\tau_1^{(j)},\tau_2^{(j)})\Big)_{j=1,\dots,J}\cid \Big(\mathbb{G}(\lambda_j,\tau_1^{(j)},\tau_2^{(j)})\Big)_{j=1,\dots,J}.
}
for any $(\lambda_j,\tau_1^{(j)},\tau_2^{(j)})\in[0,\pi]\times[\eta,1-\eta]^2,\,j=1,\dots,J$ and $J\in\N$, where the process $\mathbb{G}$ is defined in Theorem \ref{weakconvintegspectrum}. For this purpose, we apply Lemma P4.5 of \cite{brillinger75}, that is we prove that for any $\lambda_1,\dots,\lambda_J\in[0,\pi]$, $J\in\N$ and any $\tau_1^{(1)},\dots,\tau_1^{(J)},\tau_2^{(1)},\dots,\tau_2^{(J)}$ in~$[\eta,1-\eta]$, the cumulants of the vector
\al{
\Big(\mathbb{K}_n(\lambda_1,\tau_1^{(1)},\tau_2^{(1)}),\overline{\mathbb{K}_n(\lambda_1,\tau_1^{(1)},\tau_2^{(1)})},\dots,\mathbb{K}_n(\lambda_J,\tau_1^{(J)},\tau_2^{(J)}),\overline{\mathbb{K}_n(\lambda_J,\tau_1^{(J)},\tau_2^{(J)})}\Big)
}
converge to the corresponding cumulants of the vector
\al{
\Big(\mathbb{G}(\lambda_1,\tau_1^{(1)}\tau_2^{(1)}),\overline{\mathbb{G}(\lambda_1,\tau_1^{(1)}\tau_2^{(1)})},\dots,\mathbb{G}(\lambda_J,\tau_1^{(J)}\tau_2^{(J)}),\overline{\mathbb{G}(\lambda_J,\tau_1^{(J)}\tau_2^{(J)})}\Big).
}

It can easily be shown that $\overline{\mathbb{K}_n(\lambda_j,\tau_1^{(j)},\tau_2^{(j)})}=\mathbb{K}_n(\lambda_j,\tau_2^{(j)},\tau_1^{(j)})$. Hence, it is equivalent to show the convergence of the cumulants of the vector 
\al{
\Big(\mathbb{K}_n(\lambda_1,\tau_1^{(1)},\tau_2^{(1)}),\mathbb{K}_n(\lambda_1,\tau_2^{(1)},\tau_1^{(1)}),\dots,\mathbb{K}_n(\lambda_J,\tau_1^{(J)},\tau_2^{(J)}),\mathbb{K}_n(\lambda_J,\tau_2^{(J)},\tau_1^{(J)})\Big).
}

 It follows from Lemma \ref{expF-F}  that the first-order cumulants vanish as
\al{
|\E[\mathbb{K}_n(\lambda ;\tau_1,\tau_2)]|={}&|\E[\mathbb{G}_{n,U}(\lambda ;\tau_1,\tau_2)+\sqrt{n}\sum_{j=1}^2(\tau_j-\hat{F}_{n,U}(\tau_j))G_j(\lambda ;\tau_1,\tau_2)]|\\
={}&\sqrt{n}|\E[\widehat{\specdis}\phantom{F\!\!\!\!\!}_{n,U}(\lambda ;\tau_1,\tau_2)]-\specdis(\lambda ;\tau_1,\tau_2)|\\
={}&O(n^{-1/2})
}
for any $\lambda\in[0,\pi]$ and $\tau_1,\tau_2\in[\eta,1-\eta]$. Furthermore, for the second-order cumulants we obtain
\als{\label{secondordercumquantspec}
\cum\Big(&\mathbb{K}_n(\lambda ;\tau_1,\tau_2),\mathbb{K}_n(\mu,\xi_1,\xi_2)\Big)={}2\pi\int_{0}^{\lambda}\int_{0}^{\mu}\spec\big(\alpha,-\alpha,\beta; \tau_1,\tau_2,\xi_1,\xi_2 \big)\dd\alpha \dd\beta\notag\\
&+2\pi\int_0^{\lambda\wedge \mu}\spec\big(\alpha ;  \tau_1, \xi_2 \big)\spec\big(-\alpha ;  \tau_2, \xi_1 \big)\dd\alpha\notag\\
&\qquad-\sum_{j=1}^2G_j(\mu,\xi_1,\xi_2)2\pi \int_{0}^{\lambda}\spec\big(\alpha,-\alpha;  \tau_1, \tau_2, \xi_j\big)d\alpha\notag\\
&\qquad\qquad-\sum_{j=1}^2G_j(\lambda ;\tau_1,\tau_2)2\pi \int_0^{\lambda}\spec \big(\alpha,-\alpha ; \xi_1, \xi_2, \tau_j \big)d\alpha\notag\\
&\qquad\qquad\qquad+\sum_{j_1=1}^2\sum_{j_2=1}^2G_{j_1}(\lambda ;\tau_1,\tau_2)G_{j_2}(\mu,\xi_1,\xi_2)2\pi\spec\big(0;  \tau_j, \xi_j \big)+O(n^{-1}).
}
The details of the derivation of (\ref{secondordercumquantspec}) are given in Section \ref{sec::secondmomentsfidiquantspec}.

It remains to show that all cumulants of order $2<l\leq 2J$ vanish. For this purpose we prove in Section \ref{sec::vaninishinghighermomentsfidi} that
\als{\label{vanishinghighermomentsfidi}
\Big|\cum\Big(\mathbb{K}_n(\lambda_1,\tau_1^{(1)},\tau_2^{(1)}),\dots,\mathbb{K}_n(\lambda_l,\tau_1^{(l)},\tau_2^{(l)})\Big)\Big|={}O(n^{-l/2+1}),
}
i.e.\ all cumulants of order greater than $2$ tend to zero. This proves that the limiting process $\mathbb{G}$ is Gaussian and concludes the proof of (\ref{fidi}). \hfill $\Box$


\subsection{Proof of Theorem~\ref{thm:CB:subs}}
We only prove the second part;  the proof of the first part  indeed is similar but simpler, and we only focus on confidence bands for the real part of $\specdis$. Define 
\[
\tilde S_n := \Big\{\frac{2 \pi \ell}{d}, \ell = 0,1,\ldots,\lfloor d/2 \rfloor\Big\}\times S_n, \quad \tilde S := [0,\pi] \times S.
\]
Observe that
\begin{align}
& \IP\Big( \Re \specdis(\lambda ;\tau_1,\tau_2) \in \Big[ \Re \widehat{\specdis}\phantom{F\!\!\!\!\!}_{n,R}(\lambda ;\tau_1,\tau_2) - C_{E,\alpha} s(\tau_1, \tau_2), \nonumber\\
& \hspace{3.5cm}
\Re \widehat{\specdis}\phantom{F\!\!\!\!\!}_{n,R}(\lambda ;\tau_1,\tau_2) + C_{E,\alpha} s(\tau_1, \tau_2) \Big], \quad \forall (\lambda, \tau_1, \tau_2) \in \tilde S_n \Big) \nonumber \\
& = \IP\Big( \sup_{(\lambda, \tau_1, \tau_2) \in \tilde S_n} \frac{|\Re \widehat{\specdis}\phantom{F\!\!\!\!\!}_{n,R}(\lambda ;\tau_1,\tau_2) - \Re \specdis(\lambda ;\tau_1,\tau_2)| }{s(\tau_1, \tau_2)} \leq C_{E,\alpha} \Big) \nonumber \\
& = \IP\big( Y_n \leq G_n^{-1}(1-\alpha) \big), \label{thm34:coverage}
\end{align}
where
\begin{align}
Y_n & := \sqrt{n} \sup_{(\lambda, \tau_1, \tau_2) \in \tilde S_n} \frac{|\Re \widehat{\specdis}\phantom{F\!\!\!\!\!}_{n,R}(\lambda ;\tau_1,\tau_2) - \Re \specdis(\lambda ;\tau_1,\tau_2)| }{s(\tau_1, \tau_2)}, \label{eq:defYn}
\\
G_n(x) & := \frac{1}{n-b+1} \sum_{t=1}^{n-b+1} I\{\sqrt{b} \tilde E_{n,b,t} \leq x\}, \nonumber
\\
G_n^{-1}(1-\alpha) & := \inf\{x : G_n(x) \geq 1-\alpha\}, \quad \alpha \in (0,1). \nonumber
\end{align}

\noindent By Corollary 1.3 and Remark 4.1 in~\cite{gaenssler2007}, the distribution function $G$ of the random variable 
\[
Y := \sup_{(\lambda, \tau_1, \tau_2) \in \tilde S} \frac{| \Re \mathbb{G}(\lambda, \tau_1, \tau_2)|}{s(\tau_1, \tau_2)} = \sup_{(\lambda, \tau_1, \tau_2) \in \tilde S} \max\Big\{\frac{-\Re \mathbb{G}(\lambda, \tau_1, \tau_2)}{s(\tau_1, \tau_2)}, \frac{ \Re \mathbb{G}(\lambda, \tau_1, \tau_2)}{s(\tau_1, \tau_2)} \Big\}
\]
is continuous and strictly increasing on $(0,\infty)$. Let us show that $Y_n$ converges in distribution to $Y$. Defining 
\[
	\widehat Y_n := \sqrt{n} \sup_{(\lambda, \tau_1, \tau_2) \in \tilde  S} \frac{|\Re \widehat{\specdis}\phantom{F\!\!\!\!\!}_{n,R}(\lambda ;\tau_1,\tau_2) - \Re \specdis(\lambda ;\tau_1,\tau_2)| }{s(\tau_1, \tau_2)},
\]
  note that $\widehat Y_n$ converges in distribution to $Y$ by Theorem~\ref{weakconvintegspectrum} and the continuity of the map 
  $$f \mapsto \sup_{(\lambda, \tau_1, \tau_2) \in S} |f(\lambda ;\tau_1,\tau_2)/s(\tau_1, \tau_2)|.$$
   By Slutzky, it suffices to show that $\widehat Y_n - Y_n = o_P(1)$. Note that, for any bounded function $f$ on $\tilde S$ and any $\tilde S_n \subset S$, we have, by the triangle inequality, 
\begin{align*}
	\sup_{x \in \tilde S} |f(x)| &\leq \sup_{x \in \tilde S} \inf_{y \in S_n} \Big(|f(x) - f(y)| + |f(y)|\Big) 
	\\
	&\leq \sup_{x \in \tilde S} \inf_{y \in \tilde S_n} |f(x) - f(y)| + \sup_{z \in \tilde S_n} |f(z)|
\end{align*}
which yields
\[
	0 \leq \sup_{x \in \tilde S} |f(x)| - \sup_{x \in \tilde S_n} |f(x)| \leq \sup_{x \in \tilde S} \inf_{y \in \tilde S_n} |f(x) - f(y)| \leq \sup_{x,y \in S: \|x-y\| \leq d(\tilde S_n,\tilde S)} |f(x) - f(y)|.
\]
Define
\[
	g_n(\lambda,\tau_1,\tau_2) := \sqrt{n} \frac{|\Re \widehat{\specdis}\phantom{F\!\!\!\!\!}_{n,R}(\lambda ;\tau_1,\tau_2) - \Re \specdis(\lambda ;\tau_1,\tau_2)| }{s(\tau_1, \tau_2)}
\]
and apply the above inequality with $x = (\lambda,\tau_1,\tau_2)$ and $f(x) = g_n(\lambda,\tau_1,\tau_2)$ to obtain
\begin{align*}
	0 \leq \widehat Y_n - Y_n \leq  \sup_{x,y \in \tilde S: \|x-y\| \leq d(\tilde S_n,\tilde S)}|g_n(x)-g_n(y)|.
\end{align*}
By a simple calculation and Theorem \ref{weakconvintegspectrum}, the paths of $g_n$ are uniformly asymptotically equicontinuous, whence the right-hand side of the last display is $o_P(1)$. Indeed, for any fixed~$\delta >0$ and $\eps >0 $, we have
\[
		\limsup_{n \to \infty}  P\Big(\sup_{x,y \in S: \|x-y\| \leq d(\tilde S_n,\tilde S) } |g_n(x)-g_n(y)| \geq \eps\Big) \leq \limsup_{n \to \infty}  P\Big(\sup_{x,y \in \tilde S: \|x-y\| \leq \delta} |g_n(x)-g_n(y)| \geq \eps \Big).
\]
Since the left-hand side above does not depend on $\delta$ we can take  the limit $\lim_{\delta \downarrow 0}$ on both sides to obtain
\begin{multline*}
		\lim_{\delta \downarrow 0} \limsup_{n \to \infty}  P\Big(\sup_{x,y \in \tilde S: \|x-y\| \leq d(\tilde S_n,\tilde S) } |g_n(x)-g_n(y)| \geq \eps\Big)
		\\ \leq \lim_{\delta \downarrow 0} \limsup_{n \to \infty} 
		P\Big(\sup_{x,y \in \tilde S: \|x-y\| \leq \delta} |g_n(x)-g_n(y)| \geq \eps \Big) = 0
\end{multline*}
where the last equality follows from uniform asymptotic equicontinuity. Since $\eps > 0$ was arbitrary this implies $\sup_{x,y \in \tilde S: \|x-y\| \leq \delta} |g_n(x)-g_n(y)| = o_P(1)\vspace{1mm}$. 
Thus, by continuity of the distribution of $Y$, we have, for all $x \in \R$, 
\begin{equation}\label{Yn_conv}
\IP( Y_n \leq x ) \rightarrow 
\IP(Y \leq x) = G(x).
\end{equation}

Denote by  $\rho_L$ the bounded Lipschitz metric   on the space of distribution functions on $\R$: the following result will be established later in the proof
\begin{equation}\label{Gn_conv}
\rho_L(G_n, G) \xrightarrow{\IP^*} 0.
\end{equation}

From~\eqref{Gn_conv} and the continuity of $G$, we obtain the following two convergences (note that both suprema are measurable since their value does not change if $\sup_{x \in \R}$ is replaced by~$\sup_{x \in \Q}$ and the latter   is taken over a countable set):
\begin{equation}\label{Gn_conv_a}
\sup_{x \in \R} |G_n(x) - G(x) | \xrightarrow{\IP} 0,
\end{equation}
and
\begin{equation}\label{Gn_conv_b}
\sup_{x \in \R} |G_n(x) - G_n(x-) | \xrightarrow{\IP} 0.
\end{equation}
Here~\eqref{Gn_conv_b} follows from~\eqref{Gn_conv_a} by continuity of $G$. Indeed, any continuous distribution function is also uniformly continuous, and we have, for any $\eps > 0$, 
\begin{align*}
\sup_{x} |G_n(x) - G_n(x-)| &\leq \sup_{x} |G_n(x) - G_n(x-\eps)|
\\
& \leq 2\sup_{x} |G_n(x) - G(x)| + \sup_{x} |G(x) - G(x-\eps)|. 
\end{align*}
Letting $\eps \downarrow 0$, we  obtain, from the uniform continuity of $G$, 
\[
\sup_{x} |G_n(x) - G_n(x-)| \leq 2\sup_{x} |G_n(x) - G(x)|.
\]

To establish~\eqref{Gn_conv_a}, note that, by Problem~23.1 in \cite{vandervaart2000}, \eqref{Gn_conv_a} is equivalent to~$G_n(x) = G(x) + o_\IP(1)$ for every $x$, which  can be established by a standard approximation of indicator functions through Lipschitz continuous functions.  

Then, the assertion of the theorem follows from~\eqref{thm34:coverage}, \eqref{Yn_conv}, the continuity of~$G$, \eqref{Gn_conv_a}, and~\eqref{Gn_conv_b}. The coverage probability in \eqref{thm34:coverage} indeed is bounded from above by 
\begin{equation}\label{hd1}
\begin{split}
  \IP\big( G_n(Y_n) \leq G_n\big(G_n^{-1}(1-\alpha)\big) \big) 
& = \IP\big( G(Y_n) + r_{1,n} + r_{2,n} \leq 1-\alpha\big), \\
& \rightarrow \IP\big( G(Y) \leq 1-\alpha \big) = 1-\alpha,
\end{split}
\end{equation}
where the first equality follows from the fact that  $G_n$ is monotone increasing;  for the second equality,    letting  $r_{1,n} := G_n(Y_n) - G(Y_n)$ and $r_{2,n} := 1-\alpha - G_n\big(G_n^{-1}(1-\alpha)\big)$, note that~$r_{1,n} = o_{\IP}(1)$ and $r_{2,n} = o_{\IP}(1)$ since 
\begin{equation*}
|r_{1,n}| = |G_n(Y_n) - G(Y_n)| \leq \sup_{x \in \R} |G_n(x) - G(x) | \xrightarrow{\IP} 0
\end{equation*}
and, in view of Lemma 21.1 (ii), (iii) in~\cite{vandervaart2000}, 
\begin{equation*}
|r_{2,n}| = |1-\alpha - G_n\big(G_n^{-1}(1-\alpha)\big)| \leq \sup_{x \in \R} |G_n(x) - G_n(x-) | \xrightarrow{\IP} 0.
\end{equation*}
Finally, as $G$ is continuous,  it follows from the Continuous Mapping Theorem and Slutzky's lemma, that 
\[G(Y_n) + o_{\IP}(1) \xrightarrow{\mathcal{D}} G(Y);\] this  completes the proof of \eqref{hd1}. 
Now, the same coverage probability in \eqref{thm34:coverage} is bounded from below  by
\begin{equation*}
\begin{split}
  \IP\big( G_n(Y_n) < G_n\big(G_n^{-1}(1-\alpha)\big) \big) 
& = \IP\big( G(Y_n) + r_{1,n} + r_{2,n} < 1-\alpha\big) \\
& \rightarrow \IP\big( G(Y) < 1-\alpha \big) = 1-\alpha,
\end{split}
\end{equation*}
since $G_n$ is non-decreasing\footnote{Indeed, by contraposition,  $G_n(t) < G_n(s)$ implies $t < s$, so $G_n(Y_n) < G_n(G_n^{-1}(1-\alpha))$ yields $Y_n \leq G_n^{-1}(1-\alpha)$} and since the continuity of $G$  implies that $G(Y) \sim U[0,1]$. Theorem~\ref{thm:CB:subs} follows from combining this with~\eqref{hd1}.

\medskip 

\textbf{Proof of~\eqref{Gn_conv}} The proof of~\eqref{Gn_conv} follows along similar arguments as  in Section~7.3 of \cite{PolitisEtAl1999}. Similar to the notation there, let $\theta(P) = \Re \specdis(\cdot ; \cdot, \cdot)$ and
\begin{equation*}
\begin{split}
R_n\big(X_1, \ldots, X_n; \theta(P)\big) & := Y_n = \sqrt{n} \sup_{(\lambda, \tau_1, \tau_2) \in \tilde S_n} \frac{|\Re \widehat{\specdis}\phantom{F\!\!\!\!\!}_{n,R}(\lambda ;\tau_1,\tau_2) - \Re \specdis(\lambda ;\tau_1,\tau_2)| }{s(\tau_1, \tau_2)}, \\
R_{b,n}(X_t, \ldots, X_{t+b-1}, \theta(P)\big) & := A_t = \sqrt{b} \sup_{(\lambda, \tau_1, \tau_2) \in \tilde S_n} \frac{|\Re \widehat{\specdis}\phantom{F\!\!\!\!\!}_{n,b,t,R}(\lambda ;\tau_1,\tau_2) - \Re \specdis(\lambda ;\tau_1,\tau_2)| }{s(\tau_1, \tau_2)}, \\
R_{b,n}(X_t, \ldots, X_{t+b-1}, \hat\theta_n \big)
& := B_t = \sqrt{b} \tilde E_{n,b,t} \\
& = \sqrt{b} \sup_{(\lambda, \tau_1, \tau_2) \in \tilde S_n} \frac{|\Re \widehat{\specdis}\phantom{F\!\!\!\!\!}_{n,b,t,R}(\lambda ;\tau_1,\tau_2) - \Re \widehat{\specdis}\phantom{F\!\!\!\!\!}_{n,R}(\lambda ;\tau_1,\tau_2)| }{s(\tau_1, \tau_2)}. \\	
\end{split}
\end{equation*}
Denoting  by $J_n$ the cdf of $R_n\big(X_1, \ldots, X_n; \theta(P)\big) = Y_n$  (recall that $Y_n$ was defined in~\eqref{eq:defYn}),  
let $H_{n,b}$ be the empirical cdf of $\{R_{b,n}(X_t, \ldots, X_{t+b-1}, \theta(P)\big): t = 1,\ldots,n-b+1\}$ (recall that $G_n$ denotes the empirical cdf of $\{R_{b,n}(X_t, \ldots, X_{t+b-1}, \hat\theta_n \big) : t = 1,\ldots,n-b+1 \}$). A close look at the proof of Proposition~7.3.1 from \cite{PolitisEtAl1999} reveals that this result continues to hold if $R_b$ in there is replaced by $R_{b,n}$ as in our setting.\footnote{Note that we have an additional dependence on the full sample size $n$ which is not present in \cite{PolitisEtAl1999}.}  It follows that
\[
\rho_L(H_{n,b}, J_n) \xrightarrow{\IP} 0.
\]
By the reverse triangle inequality and some elementary computations, we have
\begin{multline*}
\sup_{t = 1,...,n-b+1}| R_b(X_t, \ldots, X_{t+b-1}, \hat\theta_n \big) - R_b(X_t, \ldots, X_{t+b-1}, \theta(P)\big) | 
\\
\leq \sqrt{b/n} R_n\big(X_1, \ldots, X_n; \theta(P)\big) = O_\IP(\sqrt{b/n}) = o_\IP(1).
\end{multline*}
Let 
\[
BL_1 := \Big\{f: \R \to \R:~|f(x) - f(y)| \leq |x - y|,~\sup_x |f(x)| \leq 1 \Big\}
\]
denote the set of bounded Lipschitz functions from $\R$ to $\R$: we have
\begin{align*}
& \sup_{f \in BL_1} \Big|\int_\R f(x) H_{n,b}(dx) - \int_\R f(x) G_n(dx)\Big|
\\
=~& \sup_{f \in BL_1} \Big| \frac{1}{n-b+1} \sum_{t=1}^{n-b+1} f(A_t) - \frac{1}{n-b+1} \sum_{t=1}^{n-b+1} f(B_t) \Big| 
\\
\leq~& \sup_{t}| R_b(X_t, \ldots, X_{t+b-1}, \hat\theta_n \big) - R_b(X_t, \ldots, X_{t+b-1}, \theta(P)\big) |
\\
=~& o_\IP(1).
\end{align*}
Thus, we have shown that $\rho_L( H_{n,b}, G_n ) = o_\IP(1)$. Note that  \eqref{Yn_conv}   also entails  $\rho_L(J_n,G) =~\!o(1)$. Together with $\rho_L(H_{n,b}, J_n) = o_\IP(1)$ and the triangle inequality, this yields~\eqref{Gn_conv}. \hfill $\Box$

\subsection{Proof of Theorem~\ref{thm:substr}}

We begin with Part 1 of the theorem.
Let us show  that, under the null,
\begin{align*}
T_{\rm TR}^{(n)}\Rightarrow&
T_{\rm TR}:=\sup_{(\lambda, \tau_1, \tau_2) \in S}
\Big|\frac{\Im\mathbb{G}(\lambda, \tau_1, \tau_2)}{s(\tau_1, \tau_2)}
\Big|\quad \text{as $n\to\infty$}.
\end{align*}
More precisely, by employing Theorem~\ref{weakconvintegspectrum} and the Continuous Mapping Theorem, it holds that, under the null,
\[
\sqrt{n}\max_{(\lambda, \tau_1, \tau_2) \in S}
\Big|\frac{\Im \widehat{\specdis}\phantom{F\!\!\!\!\!}_{n,R}(\lambda, \tau_1, \tau_2)}{s(\tau_1, \tau_2)}\Big|\Rightarrow
T_{\rm TR}\quad \text{as $n\to\infty$}.
\]
Further, 
\begin{align*}
0\leq T_{\rm TR}^{(n)}-\max_{(\lambda, \tau_1, \tau_2) \in S}
\Big|\frac{\Im \widehat{\specdis}\phantom{F\!\!\!\!\!}_{n,R}(\lambda, \tau_1, \tau_2)}{s(\tau_1, \tau_2)}\Big|
\leq
\sup_{x,y \in S: \|x-y\| \leq d(S_n,S)}|g_n(x)-g_n(y)|
\end{align*}
where $x = (\lambda,\tau_1,\tau_2)$ and $g_n(x) := \sqrt{n}{|\Im \widehat{\specdis}\phantom{F\!\!\!\!\!}_{n,R}(\lambda ;\tau_1,\tau_2)| }/{s(\tau_1, \tau_2)}$.
Uniform asymptotic equicontinuity of $g_n(x)$ (which follows from Theorem~\ref{weakconvintegspectrum} after a simple computation) implies that
$\sup_{x,y \in S: \|x-y\| \leq d(S_n,S)}|g_n(x)-g_n(y)|\xrightarrow{\IP}0\quad \text{as $n\to\infty$.}$
	
Proposition 7.3.1 in \cite{PolitisEtAl1999} then implies that $\rho_L(H_{n,b}^{\rm TR}, G^{\rm TR}) \xrightarrow{\IP^*} 0$ as~$n\to\infty$,
where $G^{\rm TR}$ is the cdf of $T_{\rm TR}$ and 
\begin{align*}
H_{n,b}^{\rm TR}(x) :=& \frac{1}{n-b+1} \sum_{t=1}^{n-b+1} I\{T_{{\rm TR}1}^{(n,b,t)} \leq x\}.
\end{align*}
Next note that the function $G^{\rm TR}$ is continuous; this can be established similarly to the continuity of $G$ in the proof of Theorem~\ref{weakconvintegspectrum}. Now we obtain, as in  the proof of
\eqref{Gn_conv_a}, that
\[
\sup_{x \in \R}\Big|H_{n,b}^{\rm TR}(x) - G^{\rm TR}(x)\Big| = o_P(1),
\]
which in turn yields
\begin{align*}
{H_{n,b}^{\rm TR}}(T_{\rm TR}^{(n)})
={G^{\rm TR}}(T_{\rm TR}^{(n)})+o_p(1).
\end{align*}
Consequently, it holds that, for $\alpha \in (0,1)$,
\begin{align*}
{\rm P}\bigg(
p_{{\rm TR}}\leq \alpha\bigg) =&{\rm P}\bigg(1-\alpha\leq{H_{n,b}^{\rm TR}}(T_{\rm TR}^{(n)})\bigg)\\
=
&{\rm P}\bigg(1-\alpha\leq{G^{\rm TR}}(T_{\rm TR}^{(n)})+o_p(1)\bigg)\\
\to&
{\rm P}\bigg(1-\alpha\leq{G^{\rm TR}}(T_{\rm TR})\bigg) = \alpha \quad\text{as $n\to \infty$},
\end{align*}
in view of the continuity of $G^{\rm TR}$ which, by the Continuous Mapping Theorem and Slutzky's Lemma, implies ${G^{\rm TR}}(T_{\rm TR}^{(n)})+o_p(1) \Rightarrow {G^{\rm TR}}(T_{\rm TR}) \sim U[0,1] $. This establishes Part 1 of the theorem.


\medskip

We now turn to Part 2 of the same  theorem. Note that it suffices to show that $p_{\rm TR} = o_P(1)$, since then ${\rm P}(p_{\rm TR} \leq \alpha) = 1 - {\rm P}(p_{\rm TR} > \alpha) \rightarrow 1$ for all $\alpha > 0$. Next, since all copulas are continuous and since Assumption~\ref{assumptionsspectraldist}(C) implies uniform convergence of the series defining~$\spec(\omega; \tau_1,\tau_2)$ in \eqref{copulaspecdensintroduction}, we have that $\spec(\omega; \tau_1,\tau_2)$ is continuous as a function of $(\tau_1, \tau_2)$. Now recall the definition in~\eqref{specdisestimatorintro}: $\specdis(\lambda ;\tau_1,\tau_2)=\int_0^{\lambda}\spec(\omega; \tau_1,\tau_2)d\omega$. Thus, $\Im \specdis$ is continuous.
Now, by assumption there exists $(\lambda,\tau_1,\tau_2) \in S$ such that $|\Im \specdis(\lambda, \tau_1, \tau_2)| =: c > 0$. The continuity of $\Im \specdis$ together with~\eqref{eq:SetConv} implies that there exist $n_0$ such that
\begin{equation}\label{eqn:norm_Im_F_lrg_c}
\sup_{(\lambda, \tau_1, \tau_2) \in S_n}
\Big|\frac{\Im \specdis(\lambda, \tau_1, \tau_2)}{s(\tau_1, \tau_2)}\Big| \geq c/(2s_{\max}), \text{ for all $n \geq n_0$,}
\end{equation}
where $s_{\max}:=\sup_{(\tau_1,\tau_2)\in[\eta,1-\eta]^2}s(\tau_1, \tau_2)\vspace{1mm}$.

Let
\[
\bar T_{\rm TR}^{(n)} := \sqrt{n} \max_{(\lambda,\tau_1, \tau_2) \in S_n}
\Big| \frac{\Im  \widehat{\specdis}\phantom{F\!\!\!\!\!}_{n,R}(\lambda, \tau_1, \tau_2) - \Im  \specdis \phantom{F\!\!\!\!}(\lambda, \tau_1, \tau_2)}{s(\tau_1, \tau_2)}\Big|
\]
and
\[
\bar T_{\rm TR1}^{(n,b,t)} := \sqrt{b}\max_{ (\lambda, \tau_1, \tau_2) \in S_n}\Big| \frac{\Im \widehat{\specdis}\phantom{F\!\!\!\!\!}_{n,b,t,R}(\lambda, \tau_1, \tau_2)  - \Im  \specdis \phantom{F\!\!\!\!}(\lambda, \tau_1, \tau_2) }{s(\tau_1, \tau_2)}\Big|. 
\]
We have, under $H_1$, that
\begin{equation}
\label{eqn:wconv_Tnbar}
	\bar T_{\rm TR}^{(n)} \rightsquigarrow T_{\rm TR}:=\max_{(\lambda, \tau_1, \tau_2) \in S}
\Big|\frac{\Im\mathbb{G}(\lambda, \tau_1, \tau_2)}{s(\tau_1, \tau_2)}
\Big|\quad \text{as $n\to\infty$.}
\end{equation}
Denoting by $\bar G^{\rm TR}$ the cdf of $\bar T_{\rm TR}^{(n)}$ and defining 
\[ \bar H_{n,b}^{\rm TR}(x) := \frac{1}{n-b+1} \sum_{t=0}^{n-b}I\big\{\bar T_{{\rm TR}1}^{(n,b,t)} \leq x \big\},
\]
we have, by the   subsampling arguments used in the proof of Part 1, that
\[\sup_{x \in \R} |\bar H_{n,b}^{\rm TR}(x) - \bar G^{\rm TR} (x) | \xrightarrow{\IP} 0.
\]
Finally, letting $\| f\|_{S_n} := \max_{(\lambda, \tau_1, \tau_2) \in S_n} | \frac{f(\lambda, \tau_1, \tau_2)}{s(\tau_1, \tau_2)}|$, we have
\begin{equation*}
\begin{split}
p_{{\rm TR}} & = \frac{1}{n-b+1} \sum_{t=0}^{n-b}I\big\{T_{{\rm TR}1}^{(n,b,t)} > T_{\rm TR}^{(n)}\big\} \\
& = \frac{1}{n-b+1} \sum_{t=0}^{n-b} I\big\{\sqrt{b} \| \Im \widehat{\specdis}\phantom{F\!\!\!\!\!}_{n,b,t,R} - \Im \specdis + \Im \specdis \|_{S_n} >  \sqrt{n} \| \Im \widehat{\specdis}\phantom{F\!\!\!\!\!}_{n,R} - \Im \specdis + \Im \specdis \|_{S_n} \big\} \\
& \leq \frac{1}{n-b+1} \sum_{t=0}^{n-b}I\big\{\bar T_{{\rm TR}1}^{(n,b,t)} + \sqrt{b} \| \Im \specdis \|_{S_n} > \sqrt{n} \| \Im \specdis \|_{S_n} - \bar T_{\rm TR}^{(n)} \big\} \\
& = 1 - \bar H_{n,b}^{\rm TR}\big( (\sqrt{n} - \sqrt{b}) \| \Im \specdis \|_{S_n} - \bar T_{\rm TR}^{(n)} \big)
= 1 - \bar G^{\rm TR}\big( (\sqrt{n} - \sqrt{b}) \| \Im \specdis \|_{S_n} - \bar T_{\rm TR}^{(n)} \big) + o_P(1). 
\end{split}
\end{equation*}
Let us show that this implies $p_{\rm TR} = o_P(1)$. 
From~\eqref{eqn:wconv_Tnbar} we have that $\bar T_{\rm TR}^{(n)} = O_P(1)$; i.\,e., for every $\varepsilon > 0$, there exists $M$ and $n_0$ such that ${\rm P}(\bar T_{\rm TR}^{(n)} > M) < \varepsilon$ for all $n \geq n_0$. 
Hence,
\begin{equation*}
\begin{split}
& \limsup_{n \to \infty} {\rm P}\Big( 1 - \bar G^{\rm TR}\big( (\sqrt{n} - \sqrt{b}) \| \Im \specdis \|_{S_n} - \bar T_{\rm TR}^{(n)} \big) > \kappa \Big) \\
& \leq \limsup_{n \to \infty} {\rm P}\Big( 1 - \bar G^{\rm TR}\big( (\sqrt{n} - \sqrt{b}) c/(2s_{\max}) - M \big) > \kappa \Big)
+ \limsup_{n \to \infty} {\rm P}(\bar T_{\rm TR}^{(n)} > M)
\\
& \leq \varepsilon, 
\end{split}
\end{equation*}
Here we used the fact that $(\sqrt{n} - \sqrt{b}) c/(2s_{\max}) - M \rightarrow \infty$, which in turn implies that~$\bar G^{\rm TR}\big( (\sqrt{n} - \sqrt{b}) c/(2s_{\max}) - M \big) \to 1$ since $\bar G^{\rm TR}$ is a cdf.   
Since $\varepsilon > 0$ is arbitrary, it follows  that $p_{\rm TR} = o_P(1)$, which completes the proof of Part 2.
\hfill $\Box$

\medskip

\subsection{Proof of Theorem~\ref{thm:subseq}}
First, we show that  the proposed test based on $T_{\rm EQ}^{(n)}$ hs asymptotic size $\alpha$.
By the uniform asymptotic equicontinuity of $\mathbb{G}_{n,R}$ 
and Theorem \ref{weakconvintegspectrum},  a simple calculation shows that under the null $H_0$,
\begin{align*}
T_{\rm EQ}^{(n)}\Rightarrow&T_{\rm EQ}:=
\sup_{(\lambda, \tau_1, \tau_2) \in S}
\Big|\frac{\mathbb{G}(\lambda, \tau_1, \tau_2)-\mathbb{G}(\lambda, 1-\tau_1, 1-\tau_2)}{s(\tau_1, \tau_2)}\Big|.
\end{align*} 
Proposition~7.3.1 of \cite{PolitisEtAl1999} entails  
$\rho_L(H_{n,b}^{\rm EQ}, G^{\rm EQ}) \xrightarrow{\IP^*} 0$,
where
\begin{align*}
H_{n,b}^{\rm EQ}(x) :=& \frac{1}{n-b+1} \sum_{t=1}^{n-b+1} I\{ T_{\rm EQ}^{(n,b,t)}\leq x\}
\end{align*} 
 is the empirical distribution  function of $T_{\rm EQ}^{(n)}$  
and $G^{\rm EQ}$ is the distribution function of $T_{\rm EQ}$. The continuity of $G^{\rm EQ}$ follows from the same  arguments as used for the continuity of $G$ in the proof of Theorem~\ref{thm:CB:subs}. This, combined with the arguments used in the proof of~\eqref{Gn_conv_a}, yields 
\[
\sup_{x \in \R} \Big|H_{n,b}^{\rm EQ}(x) - H^{\rm EQ}(x) \Big| = o_P(1).
\]
Therefore, it holds that, under the null $H_0$,
\begin{align*}
{\rm P}\bigg(
p_{{\rm EQ}}\leq \alpha\bigg) =&{\rm P}\bigg(1-\alpha\leq{H_{n,b}^{\rm EQ}}(T_{\rm EQ}^{(n)})\bigg)
\\
=
&{\rm P}\bigg(1-\alpha\leq{G^{\rm EQ}}(T_{\rm EQ}^{(n)}) + o_p(1) \bigg)
\\
\to &
{\rm P}\bigg(1-\alpha\leq{G^{\rm EQ}}(T_{\rm EQ})\bigg) = \alpha \quad\text{as $n\to \infty$},
\end{align*}
where the last line follows from the fact that   the continuity of $G^{\rm EQ}$ implies that
$${G^{\rm EQ}}(T_{\rm EQ}^{(n)}) + o_p(1) \Rightarrow {G^{\rm EQ}}(T_{\rm EQ}) \sim U[0,1].$$ This shows that the proposed test has asymptotic level $\alpha$ and completes the proof of the first part of Theorem~\ref{thm:subseq}.

Next, we show that the test is consistent against fixed alternatives. To this end, let us show that 
${\rm P}(p_{\rm EQ} \leq \alpha) = 1 - {\rm P}(p_{\rm EQ} > \alpha) \rightarrow 1$ for all $\alpha > 0$ follows from the fact that~$p_{\rm TR} =~\!o_P(1)$. By assumption, there exists some $(\lambda,\tau_1,\tau_2) \in S$ such that 
$$| \specdis(\lambda, \tau_1, \tau_2)- \specdis(\lambda, 1-\tau_1, 1-\tau_2)| =: c > 0.$$
From 
  \eqref{eq:SetConv} and the continuity of $\specdis$ with respect to $(\lambda, \tau_1, \tau_2)$, there exists $n_0$ such that
\begin{equation}\label{eqn:norm_F_lrg_c_eq}
\sup_{(\lambda, \tau_1, \tau_2) \in S_n}
\Big|\frac{( \specdis \phantom{F\!\!\!\!}(\lambda, \tau_1, \tau_2)-  \specdis \phantom{F\!\!\!\!}(\lambda, 1- \tau_1, 1- \tau_2))}{s(\tau_1, \tau_2)}\Big| \geq c/(2s_{\max}) \text{ for all $n \geq n_0$}
\end{equation} 
where $s_{\max}:=\sup_{(\tau_1,\tau_2)\in[\eta,1-\eta]^2}s(\tau_1, \tau_2) < \infty$ by continuity of $s$ on a compact set. Defining  
\begin{align*}
&\bar T_{\rm EQ}^{(n)}:=\sqrt{{n}} 
\max_{ (\lambda, \tau_1, \tau_2) \in S_n}
\Big|\frac{\widehat{\specdis}\phantom{F\!\!\!\!\!}_{n,R}^X(\lambda, \tau_1, \tau_2)-\widehat{\specdis}\phantom{F\!\!\!\!\!}_{n,R}^{X}(\lambda, 1-\tau_1, 1-\tau_2)
-( \specdis \phantom{F\!\!\!\!}(\lambda, \tau_1, \tau_2)+  \specdis \phantom{F\!\!\!\!}(\lambda, 1- \tau_1, 1- \tau_2))
}{s(\tau_1, \tau_2)}\Big|,\\
&\text{and }\bar H_{n,b}^{\rm EQ}(x):= \frac{1}{n-b+1}\sum_{t=0}^{n-b}I\big\{\bar T_{{\rm EQ}}^{(n,b,t)}\leq x\big\}
\end{align*}
with 
\begin{align*}
&\bar T_{\rm EQ}^{(n,b,t)}\\
&:=\sqrt{{b}} 
\max_{ (\lambda, \tau_1, \tau_2) \in S_n}
\Big|\frac{\widehat{\specdis}\phantom{F\!\!\!\!\!}_{n,b,t,R}^X(\lambda,\tau_1,\tau_2)-\widehat{\specdis}\phantom{F\!\!\!\!\!}_{n,b,t,R}^{X}(\lambda,1-\tau_1,1-\tau_2)
-( \specdis \phantom{F\!\!\!\!}(\lambda, \tau_1, \tau_2)-  \specdis \phantom{F\!\!\!\!}(\lambda, 1- \tau_1, 1- \tau_2))
}{s(\tau_1, \tau_2)}\Big|,
\end{align*}
 observe that, under the alternative $H_1$, 
\begin{equation}
\label{eqn:wconv_Tnbar_eq}
	\bar T_{\rm EQ}^{(n)} \stackrel{\mathcal{D}}{\longrightarrow} T_{\rm EQ}:=\sup_{(\lambda, \tau_1, \tau_2) \in S}
\Big|\frac{\mathbb{G}(\lambda, \tau_1, \tau_2)-\mathbb{G}(\lambda, 1-\tau_1, 1-\tau_2)}{s(\tau_1, \tau_2)}
\Big|\quad \text{as $n\to\infty$.}
\end{equation}
By similar arguments as in the proof of the first part, it follows that
\begin{equation}\label{Hnb_conv_eq}
\sup_{x \in \R} |\bar H_{n,b}^{\rm EQ}(x) - \bar G^{\rm EQ} (x) | \xrightarrow{\IP} 0,
\end{equation}
where $\bar G^{\rm EQ}$ the cdf of $\bar T_{\rm EQ}$. 
By \eqref{eqn:norm_F_lrg_c_eq}, \eqref{eqn:wconv_Tnbar_eq} and \eqref{Hnb_conv_eq}, it holds that
\begin{equation*}
\begin{split}
p_{{\rm EQ}} & = \frac{1}{n-b+1} \sum_{t=0}^{n-b}I\big\{T_{{\rm EQ}}^{(n,b,t)} > T_{\rm EQ}^{(n)}\big\} \\
& \leq 1 - \bar H_{n,b}^{\rm EQ}\Bigg( (\sqrt{n} - \sqrt{b}) \sup_{(\lambda, \tau_1, \tau_2) \in S_n}
\Big|\frac{( \specdis \phantom{F\!\!\!\!}(\lambda, \tau_1, \tau_2)-  \specdis \phantom{F\!\!\!\!}(\lambda, 1- \tau_1, 1- \tau_2))}{s(\tau_1, \tau_2)}\Big| - \bar T_{\rm EQ}^{(n)} \Bigg)\\
& = 1 - \bar G^{\rm EQ}\Bigg( (\sqrt{n} - \sqrt{b}) \sup_{(\lambda, \tau_1, \tau_2) \in S_n}
\Big|\frac{( \specdis \phantom{F\!\!\!\!}(\lambda, \tau_1, \tau_2)-  \specdis \phantom{F\!\!\!\!}(\lambda, 1- \tau_1, 1- \tau_2))}{s(\tau_1, \tau_2)}\Big| - \bar T_{\rm EQ}^{(n)} \Bigg) + o_P(1),
\end{split}
\end{equation*}
where the first inequality follows by the same arguments as in the proof of Theorem~\ref{thm:substr} and the last line is a consequence of~\eqref{Hnb_conv_eq}. Since $\bar T_{\rm EQ}^{(n)}  = O_P(1)$ and since \[
(\sqrt{n} - \sqrt{b}) \sup_{(\lambda, \tau_1, \tau_2) \in S_n}
\Big|\frac{( \specdis \phantom{F\!\!\!\!}(\lambda, \tau_1, \tau_2)-  \specdis \phantom{F\!\!\!\!}(\lambda, 1- \tau_1, 1- \tau_2))}{s(\tau_1, \tau_2)}\Big| \to \infty ,
\] 
we obtain the desired result that  that $p_{{\rm EQ}} = o_P(1)$ as $n \to \infty$. \hfill $\Box$


\section{Technical details}\label{technicaldetailsquantspec}

\subsection{Details for the proof of (\ref{asympequicont})}

\subsubsection{Proof of (\ref{diffargandargprime})}\label{vanishingpartsec}

Observe that, for any $x>0$ and $\delta_n$ with
$
n^{-1/2}\ll \delta_n=o(1)
$, 
we have 
\al{
&\Prob\Big(\sup_{\substack{\lambda\in[0,\pi]\\\tau_1,\tau_2\in [\eta,1-\eta]}}|\mathbb{G}_{n,U}(\lambda,\hat{\tau}_1,\hat{\tau}_2)-\mathbb{G}_{n,U}(\lambda ;\tau_1,\tau_2)|>x\Big)\\
\leq{}&\Prob\Big(\sup_{\substack{\lambda\in[0,\pi]\\\tau_1,\tau_2\in [\eta,1-\eta]}}\sup_{\substack{\llVert(u,v)-(\tau_1,\tau_2)\rrVert_{\infty}\\ \leq \sup_{\tau\in[0,1]}|\hat{F}_{n,U}^{-1}(\tau)-\tau|}}|\mathbb{G}_{n,U}(\lambda,u,v)-\mathbb{G}_{n,U}(\lambda ;\tau_1,\tau_2)|>x\Big)\\
\leq{}&\Prob\Big(\sup_{\substack{\lambda\in[0,\pi]\\\tau_1,\tau_2\in [0,1]}}\sup_{\substack{|u-\tau_1|\leq \delta_n\\ |v-\tau_2|\leq \delta_n}}|\mathbb{G}_{n,U}(\lambda,u,v)-\mathbb{G}_{n,U}(\lambda ;\tau_1,\tau_2)|>x,\sup_{\tau\in[0,1]}|\hat{F}_{n,U}^{-1}(\tau)-\tau|\leq \delta_n\Big)\\
&+\Prob\Big(\sup_{\tau\in[0,1]}|\hat{F}_{n,U}^{-1}(\tau)-\tau|>\delta_n\Big)\\
=:{}&P_{1,n}+P_{2,n}, \text{ say.}
}
It follows from Lemma A.5 in the online appendix of \cite{DetteEtAl2016} that 
$$\sup_{\tau\in[0,1]}|\hat{F}_{n,U}^{-1}(\tau)-\tau|=O_{\Prob}(n^{-1/2});$$
 since  $n^{-1/2}\ll\delta_n$,  this implies $P_{2,n}=o(1)$. As for $P_{1,n}$ we have
\al{
P_{1,n}\leq{}&\Prob\Bigg(\sup_{\substack{\lambda\in[0,\pi]\\\tau_1,\tau_2\in [\eta,1-\eta]}}\sup_{\llVert(u,v)-(\tau_1,\tau_2)\rrVert_1\leq 2\delta_n}|\mathbb{G}_{n,U}(\lambda,u,v)-\mathbb{G}_{n,U}(\lambda ;\tau_1,\tau_2)|>x\Bigg)\\
\leq{}&\Prob\Bigg(\sup_{\substack{(\lambda,u,v),(\lambda ;\tau_1,\tau_2)\in [0,\pi]\times[\eta,1-\eta]^2\\ \llVert(\lambda,u,v)-(\lambda ;\tau_1,\tau_2)\rrVert_1\leq 2\delta_n}}|\mathbb{G}_{n,U}(\lambda,u,v)-\mathbb{G}_{n,U}(\lambda ;\tau_1,\tau_2)|>x\Bigg)
}
which vanishes asymptotically  for $n^{-1/2}\ll \delta_n=o(1)$ by the stochastic equicontinuity
\al{
\lim_{\delta\downarrow 0}\limsup_{n\rightarrow \infty}\Prob\Bigg(\sup_{\substack{(\lambda ;\tau_1,\tau_2),(\lambda^{\prime},\tau_1^{\prime},\tau_2^{\prime})\in [0,\pi]\times[\eta,1-\eta]^2\\ \llVert(\lambda ;\tau_1,\tau_2)-(\lambda^{\prime},\tau_1^{\prime},\tau_2^{\prime})\rrVert_1\leq \delta}} |\mathbb{G}_{n,U}(\lambda ;\tau_1,\tau_2)-\mathbb{G}_{n,U}(\lambda^{\prime},\tau_1^{\prime},\tau_2^{\prime})|>x\Bigg)=0
}
of the process $\Big(\mathbb{G}_{n,U}(\lambda ;\tau_1,\tau_2)\Big)_{(\lambda ;\tau_1,\tau_2)\in[0,\pi]\times[\eta,1-\eta]^2}$  proved in Section \ref{proofequicont}.

\subsubsection{Proof of \eqref{assumpA.1Tobiequation} -- convergence of higher order cumulants}\label{assumpA.1Tobi}

Let $\Psi(x):=x^{2L}$, $L\in\N$. In this case, the Orlicz norm coincides with the $L_{2L}$-norm $\llVert X\rrVert_{2L}=(\E[|X|^{2L}])^{1/(2L)}$ so that 
%
\als{\label{psitotal}
\llVert \overline{\mathbb{G}}_{n,U}(\lambda ;\tau_1,\tau_2)-&\overline{\mathbb{G}}_{n,U}(\lambda^{\prime},\tau_1^{\prime},\tau_2^{\prime})\rrVert_{\Psi}\notag\\
\leq{}&2^{(2L-1)/(2L)}\Big(\E[|\overline{\mathbb{G}}_{n,U}(\lambda ;\tau_1,\tau_2)-\overline{\mathbb{G}}_{n,U}(\lambda^{\prime},\tau_1,\tau_2)|^{2L}]\notag\\
&+\E[|\overline{\mathbb{G}}_{n,U}(\lambda^{\prime},\tau_1,\tau_2)-\overline{\mathbb{G}}_{n,U}(\lambda^{\prime},\tau_1^{\prime},\tau_2^{\prime})|^{2L}]\Big)^{1/(2L)}\notag\\
=:{}&2^{(2L-1)/(2L)}\Big(R_n^{(1)}+R_{n}^{(2)}\Big)^{1/(2L)}, \text{ say.}
}
%

 In order to   bound for $R_n^{(2)}$, observe that $\overline{\mathbb{G}}_{n,U}(\lambda^{\prime},\tau_1,\tau_2)-\overline{\mathbb{G}}_{n,U}(\lambda^{\prime},\tau_1^{\prime},\tau_2^{\prime})$ can be written~as
\al{
\overline{\mathbb{G}}_{n,U}(\lambda^{\prime},\tau_1,\tau_2)-\overline{\mathbb{G}}_{n,U}(\lambda^{\prime},\tau_1^{\prime},\tau_2^{\prime})
={}&\begin{cases}C_{\lambda^{\prime}}\overline{\mathbb{H}}_n^U(\tau,\tau^{\prime};\lambda^{\prime}),&\text{ if } \lambda^{\prime}\in(0,\pi],\\
0,&\text{ if } \lambda^{\prime}=0,
\end{cases}
} 
where $\tau=(\tau_1,\tau_2)$, $\tau^{\prime}=(\tau_1^{\prime},\tau_2^{\prime})$ and
\al{
\overline{\mathbb{H}}_n^U(\tau,\tau^{\prime};\lambda^{\prime}):=\sqrt{nb_{\lambda^{\prime}}}(\tilde{\mathbb{H}}_n^U(\tau,\tau^{\prime};\lambda^{\prime})-\E[\tilde{\mathbb{H}}_n^U(\tau,\tau^{\prime};\lambda^{\prime})])
}
with 
\al{
&\tilde{\mathbb{H}}_n^U(\tau,\tau^{\prime};\lambda^{\prime})=\frac{2\pi}{n}\sum_{s=1}^{n-1}W_{n,\lambda^{\prime}}(\frac{\lambda^{\prime}}{2}-2\pi s/n)\Big\{I_{n,U}^{\tau_1,\tau_2}(2\pi s/n)-I_{n,U}^{\tau_1^{\prime},\tau_2^{\prime}}(2\pi s/n)\Big\},\\
&W_{n,\lambda^{\prime}}(u)=\sum_{j=-\infty}^{\infty}b_{\lambda^{\prime}}^{-1}W(b_{\lambda^{\prime}}^{-1}(u+2\pi j)), \ \text{ and}\\
&W(\cdot)=\frac{1}{2\pi}I\{-\pi\leq\cdot\leq\pi\}
}
for $C_{\lambda^{\prime}}=\sqrt{2\pi \lambda^{\prime}}$ and $b_{\lambda^{\prime}}=\frac{\lambda^{\prime}}{2\pi}$. 
Furthermore, by Lemma A.4 in \cite{DetteEtAl2016}, there exist constants $K$ and $d$, independent of $\omega
_1,\ldots,\omega_p \in\R, n$ and $A_1,\ldots,A_p$, such that
\al{
\Big| \cum\Big(d_n^{A_1}(
\omega_1), \ldots, d_n^{A_p}(
\omega_p)\Big) \Big| \leq K \Big( \Big|
\Delta_n \Big(\sum_{i=1}^p
\omega_i \Big) \Big| + 1 \Big) \varepsilon(|\log \varepsilon|+1)^d
}
for any Borel sets $A_1, \ldots, A_p$ with $\min_j \Prob(X_0 \in A_j)\leq\varepsilon$.\\

Lemma \ref{lemSixthMomIncrementHn} in Section \ref{sec::auxiliaryresultsquantspec} below yields
\al{
\E[|\overline{\mathbb{G}}_{n,U}(\lambda^{\prime},\tau_1,\tau_2)-\overline{\mathbb{G}}_{n,U}(\lambda^{\prime},\tau_1^{\prime},\tau_2^{\prime})|^{2L}]\leq K_1 \llVert W\rrVert_{\infty}^{2L}C_{\lambda^{\prime}}^{2L}\sum_{l=0}^{L-1}
\frac{g^{L-l}(\llVert \tau-\tau^{\prime}\rrVert_1)}{(nb_{\lambda^{\prime}})^{l}}
}
for $\llVert \tau-\tau^{\prime}\rrVert_1>0$ sufficiently small and $g(\varepsilon)=\varepsilon(|\log \varepsilon|+1)^d$. Observing that for $\varepsilon$ sufficiently small, $g(\varepsilon)=\varepsilon(|\log \varepsilon|+1)^d< \varepsilon^{\kappa}$ for any $\kappa\in(0,1)$, we obtain
\als{\label{(II)}
\E[|\overline{\mathbb{G}}_{n,U}(\lambda^{\prime},\tau_1,\tau_2)-\overline{\mathbb{G}}_{n,U}(\lambda^{\prime},\tau_1^{\prime},\tau_2^{\prime})|^{2L}]
\leq{}&
\tilde{K}_1\sum_{l=0}^{L-1}\frac{\llVert \tau-\tau^{\prime}\rrVert_1^{(L-l)\kappa}}{n^{l}}.
}

%
%
%
%

Similarly, for $R_n^{(1)}$,
\als{\label{zuwachslambda}
\overline{\mathbb{G}}_{n,U}(\lambda ;\tau_1,\tau_2)-\overline{\mathbb{G}}_{n,U}(\lambda^{\prime},\tau_1,\tau_2)
={}&\begin{cases}C_{|\lambda-\lambda^{\prime}|}\overline{\mathbb{H}}_n^U(\tau,\tau^{\prime};|\lambda-\lambda^{\prime}|),&\text{ if } |\lambda-\lambda^{\prime}|\in(0,\pi],\\
0,&\text{ if } |\lambda-\lambda^{\prime}|=0,
\end{cases}
} 

where $\tau=(\tau_1,\tau_2)=\tau^{\prime}$ and
\al{
\overline{\mathbb{H}}_n^U(\tau,\tau^{\prime};|\lambda-\lambda^{\prime}|):=\sqrt{nb_{|\lambda-\lambda^{\prime}|}}(\tilde{\mathbb{H}}_n^U(\tau,\tau^{\prime};|\lambda-\lambda^{\prime}|)-\E[\tilde{\mathbb{H}}_n^U(\tau,\tau^{\prime};|\lambda-\lambda^{\prime}|)])
}
with 
\al{
&\tilde{\mathbb{H}}_n^U(\tau,\tau^{\prime};|\lambda-\lambda^{\prime}|)=\frac{2\pi}{n}\sum_{s=1}^{n-1}W_{n,|\lambda-\lambda^{\prime}|}\left(\frac{\lambda+\lambda^{\prime}}{2}-2\pi s/n\right)\Big\{I_{n,U}^{\tau_1,\tau_2}(2\pi s/n)\\
&\hspace{10cm}-I_{n,U}^{\tau_1^{\prime},\tau_2^{\prime}}(2\pi s/n)I\{\tau\neq \tau^{\prime}\}\Big\},\\
&W_{n,|\lambda-\lambda^{\prime}|}(u)=\sum_{j=-\infty}^{\infty}b_{|\lambda-\lambda^{\prime}|}^{-1}W(b_{|\lambda-\lambda^{\prime}|}^{-1}(u+2\pi j)), \ \text{ and}\\
&W(\cdot)=\frac{1}{2\pi}I\{-\pi\leq\cdot\leq\pi\}
}
for $C_{|\lambda-\lambda^{\prime}|}=\sqrt{2\pi |\lambda-\lambda^{\prime}|}$ and $b_{|\lambda-\lambda^{\prime}|}=\frac{|\lambda-\lambda^{\prime}|}{2\pi}\vspace{1mm}$.

Similar arguments imply 
\al{
\E[|\overline{\mathbb{G}}_{n,U}(\lambda ;\tau_1,\tau_2)-\overline{\mathbb{G}}_{n,U}(\lambda^{\prime},\tau_1,\tau_2)|^{2L}]\leq K_1 \llVert W\rrVert_{\infty}^{2L}C_{|\lambda-\lambda^{\prime}|}^{2L}\sum_{l=0}^{L-1}
\frac{K_2}{(nb_{|\lambda-\lambda^{\prime}|})^{l}},
}
and hence,
\als{\label{(I)}
\E[|\overline{\mathbb{G}}_{n,U}(\lambda ;\tau_1,\tau_2)-\overline{\mathbb{G}}_{n,U}(\lambda^{\prime},\tau_1,\tau_2)|^{2L}]\leq{}& \bar{K}_1\sum_{l=0}^{L-1}\frac{|\lambda-\lambda^{\prime}|^{L-l}}{n^l}.
}

Plugging (\ref{(I)}) and (\ref{(II)}) into (\ref{psitotal}) yields
\al{
&\llVert \overline{\mathbb{G}}_{n,U}(\lambda ;\tau_1,\tau_2)-\overline{\mathbb{G}}_{n,U}(\lambda^{\prime},\tau_1^{\prime},\tau_2^{\prime})\rrVert_{\Psi}\\
\leq{}&2^{(2L-1)/(2L)}\Big(K_1\sum_{l=0}^{L-1}\frac{|\lambda-\lambda^{\prime}|^{L-l}}{n^l}+K_2\sum_{l=0}^{L-1}\frac{\llVert \tau-\tau^{\prime}\rrVert_1^{(L-l)\kappa}}{n^{l}}\Big)^{1/(2L)}.
}

Furthermore, if $|\lambda-\lambda^{\prime}|<1$ then $|\lambda-\lambda^{\prime}|^{q}\leq |\lambda-\lambda^{\prime}|^{q\kappa}$ for all $q>0,\,\kappa\in(0,1)$ so that 
\al{
&2^{(2L-1)/(2L)}\Big(K_1\sum_{l=0}^{L-1}\frac{|\lambda-\lambda^{\prime}|^{L-l}}{n^l}+K_2\sum_{l=0}^{L-1}\frac{\llVert \tau-\tau^{\prime}\rrVert_1^{(L-l)\kappa}}{n^{l}}\Big)^{1/(2L)}\\
\leq{}&2^{(2L-1)/(2L)}\Big(K_1\sum_{l=0}^{L-1}\frac{|\lambda-\lambda^{\prime}|^{(L-l)\kappa}}{n^l}+K_2\sum_{l=0}^{L-1}\frac{\llVert \tau-\tau^{\prime}\rrVert_1^{(L-l)\kappa}}{n^{l}}\Big)^{1/(2L)}\\
\leq{}&K_3\Big(\sum_{l=0}^{L-1}\frac{|\lambda-\lambda^{\prime}|^{(L-l)\kappa}+\llVert \tau-\tau^{\prime}\rrVert_1^{(L-l)\kappa}}{n^l}\Big)^{1/(2L)}\\
\leq{}&K_3\Big(\sum_{l=0}^{L-1}\frac{2(|\lambda-\lambda^{\prime}|\vee\llVert \tau-\tau^{\prime}\rrVert_1)^{(L-l)\kappa}}{n^l}\Big)^{1/(2L)}\\
\leq{}&2^{1/(2L)}K_3\Big(\sum_{l=0}^{L-1}\frac{(|\lambda-\lambda^{\prime}|+\llVert \tau-\tau^{\prime}\rrVert_1)^{(L-l)\kappa}}{n^l}\Big)^{1/(2L)}\\
={}&2^{1/(2L)}K_3\Big(\sum_{l=0}^{L-1}\frac{\llVert(\lambda ;\tau_1,\tau_2)-(\lambda^{\prime},\tau_1^{\prime},\tau_2^{\prime}) \rrVert_1^{(L-l)\kappa}}{n^l}\Big)^{1/(2L)}.\\
}

It follows that, for all $(\lambda ;\tau_1,\tau_2),(\lambda^{\prime},\tau_1^{\prime},\tau_2^{\prime})$ with $\llVert(\lambda ;\tau_1,\tau_2)-(\lambda^{\prime},\tau_1^{\prime},\tau_2^{\prime}) \rrVert_1$ sufficiently small and   $\llVert(\lambda ;\tau_1,\tau_2)-(\lambda^{\prime},\tau_1^{\prime},\tau_2^{\prime}) \rrVert_1\geq n^{-1/\gamma}$ for all $\gamma\in(0,1)$ such that $\gamma<\kappa$,
\al{
\llVert \overline{\mathbb{G}}_{n,U}(\lambda ;\tau_1,\tau_2)-\overline{\mathbb{G}}_{n,U}(\lambda^{\prime},\tau_1^{\prime},\tau_2^{\prime})\rrVert_{\Psi}\leq{}&K_4\Big(\llVert(\lambda ;\tau_1,\tau_2)-(\lambda^{\prime},\tau_1^{\prime},\tau_2^{\prime}) \rrVert_1^{L\kappa}\\
&+\llVert(\lambda ;\tau_1,\tau_2)-(\lambda^{\prime},\tau_1^{\prime},\tau_2^{\prime}) \rrVert_1^{(L-1)\kappa+\gamma}\\
&+\dots+\llVert(\lambda ;\tau_1,\tau_2)-(\lambda^{\prime},\tau_1^{\prime},\tau_2^{\prime}) \rrVert_1^{\kappa+(L-1)\gamma}\Big)^{1/(2L)}\\
\leq{}&K_5\llVert(\lambda ;\tau_1,\tau_2)-(\lambda^{\prime},\tau_1^{\prime},\tau_2^{\prime}) \rrVert_1^{\gamma/2}.
}
Observing that $\llVert(\lambda ;\tau_1,\tau_2)-(\lambda^{\prime},\tau_1^{\prime},\tau_2^{\prime}) \rrVert_1\geq n^{-1/\gamma}$ if and only if 
\[
d((\lambda ;\tau_1,\tau_2),(\lambda^{\prime},\tau_1^{\prime},\tau_2^{\prime})) =\llVert(\lambda ;\tau_1,\tau_2)-(\lambda^{\prime},\tau_1^{\prime},\tau_2^{\prime}) \rrVert_1^{\gamma/2}\geq n^{-1/2} =:\bar{\eta}_n/2,
\]
we have  
\al{
\llVert \overline{\mathbb{G}}_{n,U}(\lambda ;\tau_1,\tau_2)-\overline{\mathbb{G}}_{n,U}(\lambda^{\prime},\tau_1^{\prime},\tau_2^{\prime})\rrVert_{\Psi}\leq{}&Kd((\lambda ;\tau_1,\tau_2),(\lambda^{\prime},\tau_1^{\prime},\tau_2^{\prime}))
}
for all $(\lambda ;\tau_1,\tau_2),(\lambda^{\prime},\tau_1^{\prime},\tau_2^{\prime})$ with $d((\lambda ;\tau_1,\tau_2),(\lambda^{\prime},\tau_1^{\prime},\tau_2^{\prime}))\geq \bar{\eta}_n/2$. This establishes (\ref{assumpA.1Tobiequation}). \hfill $\Box$

\subsection{Details for the proof of (\ref{fidi})}\label{sec::detailsfidi}

All results in this section rely on the assumption

\begin{itemize}
\item[(CS)] Assume that assumption (S) holds and that, for given $p\geq 2,\, l\geq 0$, a constant~$K<~\!\infty$ exists such that the  summability condition 
\al{
\sum_{k_1,\dots,k_{p-1}\in\Z}(1+|k_j|^l)|\cum(I\{X_{k_1}\in A_1\},\dots,I\{X_{k_{p-1}}\in A_{p-1}\},I\{X_{0}\in A_p\})|< K
}
holds for arbitrary intervals $A_1,\dots, A_p\subset \R$ and all $j=1,\dots, p-1$.
\end{itemize}
This condition is a consequence of Assumption (C), but is slightly weaker and, therefore, mentioned seperately.

\subsubsection{Proof of (\ref{secondordercumquantspec}) }\label{sec::secondmomentsfidiquantspec}

Note that
\al{
&\cum\Big(\mathbb{K}_n(\lambda ;\tau_1,\tau_2),\mathbb{K}_n(\mu,\xi_1,\xi_2)\Big)\\
={}&\cum\Big(\mathbb{G}_{n,U}(\lambda ;\tau_1,\tau_2)+\sqrt{n}\sum_{j=1}^2(\tau_j-\hat{F}_{n,U}(\tau_j))G_j(\lambda ;\tau_1,\tau_2),\\
&\hspace{2cm}\mathbb{G}_{n,U}(\mu,\xi_1,\xi_2)+\sqrt{n}\sum_{j=1}^2(\xi_j-\hat{F}_{n,U}(\xi_j))G_j(\mu,\xi_1,\xi_2)\Big)\\
=:{}&U_n^{(1)}-U_n^{(2)}-U_n^{(3)}+U_n^{(4)}
}
where 
\al{
&U_n^{(1)}=\cum\Big(\mathbb{G}_{n,U}(\lambda ;\tau_1,\tau_2),\mathbb{G}_{n,U}(\mu,\xi_1,\xi_2)\Big),\\
&U_n^{(2)}=\sqrt{n}\sum_{j=1}^2G_j(\mu,\xi_1,\xi_2)\cum\Big(\mathbb{G}_{n,U}(\lambda ;\tau_1,\tau_2),\hat{F}_{n,U}(\xi_j)\Big),\\
&U_n^{(3)}=\sqrt{n}\sum_{j=1}^2G_j(\lambda ;\tau_1,\tau_2)\cum\Big(\hat{F}_{n,U}(\tau_j)-\tau_j,\mathbb{G}_{n,U}(\mu,\xi_1,\xi_2)\Big), \ \text{ and}\\
&U_n^{(4)}=n\sum_{j_1=1}^2\sum_{j_2=1}^2G_{j_1}(\lambda ;\tau_1,\tau_2)G_{j_2}(\mu,\xi_1,\xi_2)\cum\Big(\hat{F}_{n,U}(\tau_j)-\tau_j,\hat{F}_{n,U}(\xi_j)-\xi_j\Big).
}

First consider $U_n^{(1)}$. We have
\al{
\cum\Big(\mathbb{G}_{n,U}&(\lambda ;\tau_1,\tau_2),\mathbb{G}_{n,U}(\mu,\xi_1,\xi_2)\Big)
={}n^{-3}\sum_{r=1}^{n-1}\sum_{s=1}^{n-1}I\big\{0\leq \frac{2\pi r}{n} \leq \lambda\big\}I\big\{0\leq \frac{2\pi s}{n}\leq \mu\big\}\\
&\times \cum\Big(d_{n,U}^{\tau_1}\big(\frac{2\pi r}{n}\big)d_{n,U}^{\tau_2}\big(-\frac{2\pi r}{n}\big),d_{n,U}^{\xi_1}\big(\frac{2\pi s}{n}\big)d_{n,U}^{\xi_2}\big(-\frac{2\pi s}{n}\big)\Big).
}
By Theorem 2.3.2 in \cite{brillinger75}, as $\E[d_{n,U}^{\tau}\left(2\pi r/n\right)]=0$ for any $r=1,\dots,n-1$,
\al{
&\cum\Big(d_{n,U}^{\tau_1}\left(2\pi r/n\right)d_{n,U}^{\tau_2}\left(-2\pi r/n\right),d_{n,U}^{\xi_1}\left(2\pi s/n\right)d_{n,U}^{\xi_2}\left(-2\pi s/n\right)\\
&\quad={}\cum\Big(d_{n,U}^{\tau_1}\left(2\pi r/n\right),d_{n,U}^{\tau_2}\left(-2\pi r/n\right),d_{n,U}^{\xi_1}\left(2\pi s/n\right),d_{n,U}^{\xi_2}\left(-2\pi s/n\right)\Big)\\
&\qquad+\cum\Big(d_{n,U}^{\tau_1}\left(2\pi r/n\right),d_{n,U}^{\xi_1}\left(2\pi s/n\right)\Big)\cum\Big(d_{n,U}^{\tau_2}\left(-2\pi r/n\right),d_{n,U}^{\xi_2}\left(-2\pi s/n\right)\Big)\\
&\qquad+\cum\Big(d_{n,U}^{\tau_1}\left(2\pi r/n\right),d_{n,U}^{\xi_2}\left(-2\pi s/n\right)\Big)\cum\Big(d_{n,U}^{\tau_2}\left(-2\pi r/n\right),d_{n,U}^{\xi_1}\left(2\pi s/n\right)\Big)
}
and, from Theorem 1.3 in the online appendix of \cite{DetteEtAl2016}, we know that under Assumption (CS) with $p=2,4$ and $l\geq 1$, for all $\tau_1,\dots,\tau_k\in[\eta,1-\eta]$ and $\lambda_1,\dots,\lambda_K\in\R$,
\als{\label{tobilemma1.1bis1.3}
\cum\Big(d_{n,U}^{\tau_1}(\lambda_1),\dots,d_{n,U}^{\tau_K}(\lambda_K)\Big)={}&(2\pi)^{K-1}\Delta_n\Big(\sum_{j=1}^K\lambda_j\Big)\spec_{q_{\tau_1},\dots,q_{\tau_K}}(\lambda_1,\dots,\lambda_{K-1})\notag\\
&+\varepsilon_n(\tau_1,\dots,\tau_k,\lambda_1,\dots,\lambda_K),
}
where $\Delta_n(\cdot):=\sum_{t=0}^{n-1}e^{-it\cdot}$ and 
\al{
\sup_{n}\sup_{\substack{\tau_1,\dots,\tau_K\in[0,1]\\\lambda_1,\dots,\lambda_K\in[0,\pi]}}|\varepsilon_n(\tau_1,\dots,\tau_k,\lambda_1,\dots,\lambda_K)|<\infty.
}
Observe that 
\al{
0\leq \Delta_n\left(\frac{2\pi}{n}s\right):=\begin{cases}
n,&\text{ if }s\in n\Z;\\
0,&\text{ if }s\notin n\Z.
\end{cases}
}
Hence, the functions $\Delta_n(\cdot)$ impose linear restrictions on the summation indices and we obtain
\al{
U_n^{(1)}={}&n^{-3}\sum_{r=1}^{n-1}\sum_{s=1}^{n-1}I\big\{0\leq \frac{2\pi r}{n}\leq\lambda\}I\big\{0\leq \frac{2\pi s}{n}\leq \mu\big\}\\
&\hspace{1cm}\times\Big\{(2\pi)^{3}\Delta_n\left(0\right)\spec\Big(\frac{2\pi r}{n},-\frac{2\pi r}{n},\frac{2\pi s}{n};  \tau_1,\tau_2,\xi_1,\xi_2\Big)+O(1)\\
&\hspace{1cm}+\Big((2\pi)\Delta_n\Big(\frac{2\pi (r+s)}{n}\Big)\spec\Big(\frac{2\pi r}{n}; \tau_1, \xi_1\Big)+O(1)\Big)\\
&\hspace{1.5cm}\times\Big((2\pi)\Delta_n\Big(-\frac{2\pi (r+s)}{n}\Big)\spec\Big(-\frac{2\pi r}{n}; \tau_2, \xi_2\Big)+O(1)\Big)\\
&\hspace{1cm}+\Big((2\pi)\Delta_n\Big(\frac{2\pi (r-s)}{n}\Big)\spec\Big(\frac{2\pi r}{n}; \tau_1, \xi_2\Big)+O(1)\Big)\\
&\hspace{1.5cm}\times\Big((2\pi)\Delta_n\Big(\frac{2\pi (s-r)}{n}\Big)\spec\Big(-\frac{2\pi r}{n}; \tau_2, \xi_1\Big)+O(1)\Big)\Big\}\\
={}&n^{-3}\sum_{r=1}^{n-1}\sum_{s=1}^{n-1}I\big\{0\leq \frac{2\pi r}{n}\leq\lambda\}I\big\{0\leq \frac{2\pi s}{n}\leq \mu\big\}\Big((2\pi)^{3}n\\
&\hspace{4cm}\times\spec\Big(\frac{2\pi r}{n},-\frac{2\pi r}{n},\frac{2\pi s}{n}+O(1);  \tau_1, \tau_2, \xi_1, \xi_2\Big)\Big)\\
&+n^{-3}\sum_{r=1}^{n-1}I\big\{0\leq\frac{2\pi r}{n} \leq\lambda\big\}I\big\{0\leq 2\pi -\frac{2\pi r}{n}\leq \mu\big\}\\
&\hspace{1.5cm}\times \Big(2\pi n\,\spec\Big(\frac{2\pi r}{n}; \tau_1, \xi_1\Big)+O(1)\Big) \Big(2\pi n\,\spec\Big(-\frac{2\pi r}{n}; \tau_2, \xi_2\Big)+O(1)\Big)\\
&+n^{-3}\sum_{r=1}^{n-1}I\big\{0\leq \frac{2\pi r}{n}\leq\lambda\big\}I\big\{0\leq\frac{2\pi r}{n}\leq\mu\big\}\\
&\hspace{1.5cm}\times \Big(2\pi n\,\spec\Big(\frac{2\pi r}{n}; \tau_1, \xi_2\Big)+O(1)\Big)\Big(2\pi n\,\spec\Big(-\frac{2\pi r}{n}; \tau_2, \xi_1\Big)+O(1)\Big).
} 
Similar arguments as   in the proof of Lemma \ref{expF-F} in Section \ref{sec::auxiliaryresultsquantspec} below yield
\al{
U_n^{(1)}={}&2\pi\int_{0}^{\lambda}\int_{0}^{\mu}\spec\big(\alpha,-\alpha,\beta; \tau_1,\tau_2,\xi_1,\xi_2 \big)d\alpha d\beta+O(n^{-1})\\
&+2\pi\int_0^{2\pi}I\{0\leq\alpha\leq\lambda\}I\{0\leq2\pi-\alpha\leq\mu\}\spec\big(\alpha ;  \tau_1, \xi_1 \big)\spec\big(-\alpha ;  \tau_2, \xi_2 \big)d\alpha+O(n^{-1})\\
&+2\pi\int_0^{\lambda\wedge \mu}\spec\big(\alpha ;  \tau_1, \xi_2 \big)\spec\big(-\alpha ;  \tau_2, \xi_1 \big)d\alpha+O(n^{-1})
} 
and, as 
$$\int_0^{2\pi}I\{0\leq\alpha\leq\lambda\}I\{0\leq2\pi-\alpha\leq\mu\}\spec\big(\alpha ;  \tau_1, \xi_1 \big)\spec\big(-\alpha ;  \tau_2, \xi_2 \big)d\alpha=0,$$
 because $\lambda,\mu\in[0,\pi]$,
\begin{align}
U_n^{(1)}={}&2\pi\int_{0}^{\lambda}\int_{0}^{\mu}\spec\big(\alpha,-\alpha,\beta; \tau_1,\tau_2,\xi_1,\xi_2 \big)d\alpha d\beta \notag\\
&+2\pi\int_0^{\lambda\wedge \mu}\spec\big(\alpha ;  \tau_1, \xi_2 \big)\spec\big(-\alpha ;  \tau_2, \xi_1 \big)d\alpha+O(n^{-1}). \label{eq:Un1-repr}
\end{align}

As for $U_n^{(2)}$, we have
\al{
U_n^{(2)}
={}&\sum_{j=1}^2G_j(\mu,\xi_1,\xi_2)n^{-2}\sum_{r=1}^{n-1}I\big\{0\leq \frac{2\pi r}{n} \leq \lambda\big\}\cum\Big(d_{n,U}^{\tau_1}\big(\frac{2\pi r}{n}\big)d_{n,U}^{\tau_2}\big(-\frac{2\pi r}{n}\big),d_{n,U}^{\xi_j}(0)\Big),
}
where, in view of Theorem 2.3.2 in \cite{brillinger75} and the fact that $\E\Big[d_{n,U}^{\tau_1}\left(2\pi r/n\right)\Big]=0$\linebreak  for~$r=1,\dots,n-1$,
\al{
\cum\Big(d_{n,U}^{\tau_1}\left(2\pi r/n\right)&d_{n,U}^{\tau_2}\left(-2\pi r/n\right),d_{n,U}^{\xi_j}(0)\Big)\\
={}&\cum\Big(d_{n,U}^{\tau_1}\left(2\pi r/n\right),d_{n,U}^{\tau_2}\left(-2\pi r/n\right),d_{n,U}^{\xi_j}(0)\Big).
}
Hence, with similar arguments as in the derivation of~\eqref{eq:Un1-repr}, we obtain
\al{
U_n^{(2)}={}&\sum_{j=1}^2G_j(\mu,\xi_1,\xi_2)n^{-2}\sum_{r=1}^{n-1}I\big\{0\leq \frac{2\pi r}{n} \leq \lambda\big\}\Big\{(2\pi)^2\Delta_n(0)\spec \Big(2\pi r/n,-2\pi r/n ;  \tau_1, \tau_2, \xi_j \Big)\\
&\hspace{5cm}+\varepsilon_n(\tau_1,\tau_2,\xi_j,2\pi r/n,-2\pi r/n,0)\Big\}\\
={}&\sum_{j=1}^2G_j(\mu,\xi_1,\xi_2)2\pi \int_{0}^{\lambda}\spec\big(\alpha,-\alpha;  \tau_1, \tau_2, \xi_j\big)d\alpha+O(n^{-1}).
}
Analogously, 
\al{
U_n^{(3)}={}&\sum_{j=1}^2G_j(\lambda ;\tau_1,\tau_2)2\pi \int_0^{\lambda}\spec \big(\alpha,-\alpha ; \xi_1, \xi_2, \tau_j \big)d\alpha+O(n^{-1})
}
and
\al{
U_n^{(4)}={}&\sum_{j_1=1}^2\sum_{j_2=1}^2G_{j_1}(\lambda ;\tau_1,\tau_2)G_{j_2}(\mu,\xi_1,\xi_2)2\pi\spec\big(0;  \tau_j, \xi_j \big)+O(n^{-1}).
}
\hfill $\Box$

\subsubsection{Proof of (\ref{vanishinghighermomentsfidi}) -- convergence of the second-order cumulants }\label{sec::vaninishinghighermomentsfidi}
Observe that
\al{
&\cum\Big(\mathbb{K}_n(\lambda_1,\tau_1^{(1)},\tau_2^{(1)}),\dots,\mathbb{K}_n(\lambda_l,\tau_1^{(l)},\tau_2^{(l)})\Big)\\
={}&\cum\Big(\mathbb{G}_{n,U}(\lambda_1,\tau_1^{(1)},\tau_2^{(1)})+\sqrt{n}\sum_{j=1}^2(\tau_{j}^{(1)}-\hat{F}_{n,U}(\tau_{j}^{(1)}))G_{j}(\lambda_1,\tau_1^{(1)},\tau_2^{(1)}),\dots,\\
&\hspace{2cm} \mathbb{G}_{n,U}(\lambda_l,\tau_1^{(l)},\tau_2^{(l)})+\sqrt{n}\sum_{j=1}^2(\tau_{j}^{(l)}-\hat{F}_{n,U}(\tau_{j}^{(l)}))G_{j}(\lambda_l,\tau_1^{(l)},\tau_2^{(l)})\Big).
}
Let $\cum\Big(A_s,B_t;s\in\mathcal{S},t\in\mathcal{T}\Big):=\cum\Big(A_{s_1},\dots,A_{s_{|\mathcal{S}|}},B_{t_1},\dots,B_{t_{|\mathcal{T}|}}\Big)$ for some finite sets~$\mathcal{S} = \{s_1,\dots, S_{|\mathcal{S}|}\},\mathcal{T} = \{t_1,\dots,t_{|\mathcal{T}|}\}$. Then, by Theorem 2.3.1 (ii) and (iv) in \cite{brillinger75}, with $\mathcal{S}^C:=\{1,\dots,l\}\backslash\mathcal{S}$, we have
\al{
&\cum\Big(\mathbb{K}_n(\lambda_1,\tau_1^{(1)},\tau_2^{(1)}),\dots,\mathbb{K}_n(\lambda_l,\tau_1^{(l)},\tau_2^{(l)})\Big)\\
={}&\sum_{\mathcal{S}\subseteq \{1,\dots,l\}} \cum\Big(\mathbb{G}_{n,U}(\lambda_p,\tau_1^{(p)},\tau_2^{(p)}),\sqrt{n}\sum_{j=1}^2(\tau_{j}^{(q)}-\hat{F}_{n,U}(\tau_{j}^{(q)}))\\
&\hspace{7cm}G_{j}(\lambda_q,\tau_1^{(q)},\tau_2^{(q)});p\in\mathcal{S},q\in\mathcal{S}^C\Big)\\
={}&\sum_{\mathcal{S}\subseteq \{1,\dots,l\}} \cum\Big(\sqrt{n}\frac{2\pi}{n}\sum_{s=1}^{n-1}I\{0\leq2\pi s/n\leq\lambda_p\}\frac{1}{2\pi n}d_{n,U}^{\tau_1^{(p)}}(2\pi s/n)d_{n,U}^{\tau_2^{(p)}}(-2\pi s/n),\\
&\sqrt{n}\sum_{j=1}^2(\tau_{j}^{(q)}-\hat{F}_{n,U}(\tau_{j}^{(q)}))G_{j}(\lambda_q,\tau_1^{(q)},\tau_2^{(q)});p\in\mathcal{S},q\in\mathcal{S}^C\Big)\\
={}&n^{l/2}\sum_{\substack{\mathcal{S}\subseteq \{1,\dots,l\}\\ \mathcal{S}:=\{\xi_1,\dots,\xi_m\}\\ \mathcal{S}^C:=\{\xi_{m+1},\dots,\xi_l\}}}(-1)^{l-m}n^{-2m}\sum_{s_{\xi_1},\dots,s_{\xi_m}=1}^{n-1}\Big(\prod_{p\in\mathcal{S}}I\{0\leq2\pi s_p/n\leq\lambda_p\}\prod_{q\in\mathcal{S}^C}n^{-(l-m)}\\
& \times \sum_{j_q=1}^2G_{j_q}(\lambda_q,\tau_1^{(q)},\tau_2^{(q)})\cum\Big(d_{n,U}^{\tau_1^{(p)}}(2\pi s_p/n)d_{n,U}^{\tau_2^{(p)}}(-2\pi s_p/n),d_{n,U}^{\tau_{j_q}^{(q)}}(0);p\in\mathcal{S},q\in\mathcal{S}^C\Big),\notag 
}

where we have used the convention that $\prod_{p\in\emptyset}a_p:=1$.\\

Hence, since  $\sup\limits_{j=1,2}\sup\limits_{\lambda\in[0,\pi],\tau_1,\tau_2\in[0,1]}|G_j(\lambda ;\tau_1,\tau_2)|<\infty$  by Assumption (D),
\als{\label{estimationcumhigherorder}
\Big|\cum\Big(\mathbb{K}_n(\lambda_1,\tau_1^{(1)},\tau_2^{(1)}),&\dots,\mathbb{K}_n(\lambda_l,\tau_1^{(l)},\tau_2^{(l)})\Big)\Big|\leq  Kn^{-l/2}\\
\sum_{\substack{\mathcal{S}\subseteq \{1,\dots,l\}\\ \mathcal{S}:=\{\xi_1,\dots,\xi_m\}\\ \mathcal{S}^C:=\{\xi_{m+1},\dots,\xi_l\}}}\hspace{-8mm}n^{-m}\sum_{s_{\xi_1},\dots,s_{\xi_m}=1}^{n-1} & \Big|\cum\Big(d_{n,U}^{\tau_1^{(p)}}(2\pi s_p/n)d_{n,U}^{\tau_2^{(p)}}(-2\pi s_p/n),d_{n,U}^{\tau_{j_q}^{(q)}}(0);p\in\mathcal{S},q\in\mathcal{S}^C\Big)\Big|,\notag
} 

for some constant $K$. Put 
\al{
\omega_{k,u}:=\begin{cases}
2\pi s_{u}/n&k=1,u\in\mathcal{S},\\
-2\pi s_{u}/n&k=2,u\in\mathcal{S},\\
0&u\in\mathcal{S}^C.
\end{cases}
}

Then, by Theorem 2.3.2 in \cite{brillinger75},
\als{\label{cumtabledec}
\cum\Big(d_{n,U}^{\tau_1^{(p)}}(2\pi s_p/n)d_{n,U}^{\tau_2^{(p)}}(-2\pi s_p/n)&,d_{n,U}^{\tau_{j_q}^{(q)}}(0);p\in\mathcal{S},q\in\mathcal{S}^C\Big)\notag\\
={}&\sum_{\{\nu_1,\dots,\nu_R\}}\prod_{r=1}^R\cum\Big(d_{n,U}^{\tau_k^{(u)}}(\omega_{k,u});(u,k)\in\nu_r\Big),
}

where the summation is over all indecomposable partitions of the table

\begin{center}
\begin{tabular}{cc}
$(\xi_1,1)$&$(\xi_1,2)$\\
\vdots&\vdots\\
$(\xi_{m},1)$&$(\xi_{m},2)$\\
$(\xi_{m+1},j_{\xi_{m+1}})$&\\
\vdots&\\
$(\xi_{l},j_{\xi_{l}})$&
\end{tabular}
\end{center}

However, all indecomposable partitions of the above table are obtained by adding in the various possible ways the elements $(\xi_{m+1},j_{\xi_{m+1}}),\dots,(\xi_{l},j_{\xi_{l}})$ to the indecomposable partitions of the table

\begin{center}
\begin{tabular}{cc}
$(\xi_1,1)$&$(\xi_1,2)$\\
\vdots&\vdots\\
$(\xi_{m},1)$&$(\xi_{m},2)$
\end{tabular}
\end{center}
Therefore, and since $\E[d_{n,U}^{\tau_k^{(u)}}(\omega)]=0$ for all $\omega\not\equiv 0\mod 2\pi$, the first-order cumulants in~(\ref{cumtabledec}) are zero. Furthermore, the maximum number of sets in an indecomposable decomposition of the above table is $m$.  Hence, neglecting, for notational convenience,  the indices of $\spec$, by (\ref{tobilemma1.1bis1.3}) we obtain, with the convention $\prod_{i\in\emptyset}a_i:=1$,
\als{
&\cum\Big(d_{n,U}^{\tau_1^{(p)}}(2\pi s_p/n)d_{n,U}^{\tau_2^{(p)}}(-2\pi s_p/n),d_{n,U}^{\tau_{j_q}^{(q)}}(0);p\in\mathcal{S},q\in\mathcal{S}^C\Big)\notag\\
={}&\sum_{\substack{\{\nu_1,\dots,\nu_R\}\\ |\nu_r|\geq 2;r=1,\dots,R}}\prod_{r=1}^R\cum\Big(d_{n,U}^{\tau_k^{(u)}}(\omega_{k,u});(u,k)\in\nu_r\Big)\notag\\
={}&\sum_{\substack{\{\nu_1,\dots,\nu_R\}\\ |\nu_r|\geq 2;r=1,\dots,R}}\prod_{r=1}^R\Big[(2\pi)^{|\nu_r|-1}\Delta_n\Big(\sum_{(u,k)\in\nu_r}\omega_{k,u}\Big)\spec\Big(\omega_{k,u};(u,k)\in\nu_r\Big)+O(1)\Big]\notag\\
={}&\sum_{\substack{\{\nu_1,\dots,\nu_R\}\\ |\nu_r|\geq 2;r=1,\dots,R}}\sum_{I\subseteq\{1,\dots,R\}}\prod_{j\in I}\Delta_n\Big(\sum_{(u,k)\in\nu_j}\omega_{k,u}\Big)\spec\Big(\omega_{k,u};(u,k)\in\nu_j\Big)O(1)\label{fidihelpqualcosa}
}


where 
\al{
\Delta_n\Big(\sum_{(u,k)\in\nu_j}\omega_{k,u}\Big)=\Delta_n\Big(\frac{2\pi}{n}\sum_{\substack{(u,k)\in\nu_j\\ u\in\{\xi_1,\dots,\xi_m\}}}(-1)^{k+1}s_{u}\Big)=\begin{cases}
0,&\sum\limits_{\substack{(u,k)\in\nu_j\\ u\in\{\xi_1,\dots,\xi_m\}}}(-1)^{k+1}s_{u}\not\in n\Z\\
n,&\sum\limits_{\substack{(u,k)\in\nu_j\\ u\in\{\xi_1,\dots,\xi_m\}}}(-1)^{k+1}s_{u}\in n\Z.
\end{cases}
}

That is, after substituting \eqref{fidihelpqualcosa} into (\ref{estimationcumhigherorder}), the functions $\Delta_n(\cdot)$ impose linear restrictions on the summation indices, whence
\al{
&\sum_{s_{\xi_1},\dots,s_{\xi_m}=1}^{n-1}\sum_{\substack{\{\nu_1,\dots,\nu_R\}\\ |\nu_r|\geq 2;r=1,\dots,R}}\sum_{I\subseteq\{1,\dots,R\}}\prod_{j\in I}\Delta_n\Big(\frac{2\pi}{n}\sum\limits_{\substack{(u,k)\in\nu_j\\ u\in\{\xi_1,\dots,\xi_m\}}}(-1)^{k+1}s_{u}\Big)O(1)\notag\\
={}&\sum_{\substack{\nu:=\{\nu_1,\dots,\nu_R\}\\ |\nu_r|\geq 2;r=1,\dots,R}}\sum_{I\subseteq\{1,\dots,R\}}\sum_{s_{\xi_1},\dots,s_{\xi_m}\in \mathcal{R}(\nu,I)}n^{|I|}O(1)
}
with
\al{
\mathcal{R}(\nu,I):=\Big\{(s_{\xi_1},\dots,s_{\xi_m})\in\{1,\dots,n-1\}^m\Big|\sum\limits_{\substack{(u,k)\in\nu_j\\ u\in\{\xi_1,\dots,\xi_m\}}}(-1)^{k+1}s_{u}\in n\Z,\forall \nu_j\in\nu,j\in I\Big\}.
}

Note that there are $|I|$ linear constraints on $s_{\xi_1},\dots,s_{\xi_m}$ if $|I|<R$ and $|I|-1$ linear constraints if $|I|=R$, i.e.\ there are $|I|-\floor{|I|/R}$ linear constraints. This follows similarly as in the proof of Lemma A.2 in \cite{DetteEtAl2016}. More precisely, if we define for every~$\nu_j\in\{\nu_1,\dots,\nu_R\}$ a vector $w^{(j)}=(w_1^{(j)},\dots,w_m^{(j)})$ with 
\al{
w_v^{(j)}:=I\{(v,1)\in\nu_j\}-I\{(v,2)\in\nu_j\}\in \{-1,0,1\}^m,
}
we can rewrite the condition $\sum\limits_{\substack{(u,k)\in\nu_j\\ u\in\{\xi_1,\dots,\xi_m\}}}(-1)^{k+1}s_{u}\in n\Z$ as $(s_{\xi_1},\dots,s_{\xi_m})w^{(j)}\in n\Z$. Note that  two at most of the vectors $w^{(1)},\dots,w^{(R)}$ have non-zero entries being one $-1$ and the other $1$ at each position $v=1,\dots,m$. Hence, the linear restrictions corresponding to~$\nu_{j_1},\dots,\nu_{j_k}$ are linearly dependent if and only if $\sum_{a=1}^{k}w^{(j_a)}=0$. However, in the case of indecomposable partitions, $\sum_{a=1}^{k}w^{(j_a)}=0$ if and only if $\{j_1,\dots,j_k\}=\{1,\dots,R\}$.

Therefore,
$\sum_{\substack{\{\nu_1,\dots,\nu_R\}\\ |\nu_r|\geq 2;r=1,\dots,R}}\sum_{I\subseteq\{1,\dots,R\}}\sum_{s_{\xi_1},\dots,s_{\xi_m}\in \mathcal{R}(\nu,I)}n^{|I|}$
is of order 
\al{
\max_{|I|\leq R\leq m}n^{m-(|I|-\floor{|I|/R}}n^{|I|}=\max_{|I|\leq R\leq m}n^{m+\floor{|I|/R}}=n^{m+1}.
}
Thus, in (\ref{estimationcumhigherorder}), we obtain that, for some constant $K^{\prime}$,
\al{
\Big|\cum\Big(\mathbb{K}_n(\lambda_1,\tau_1^{(1)},\tau_2^{(1)}),\dots,\mathbb{K}_n(\lambda_l,\tau_1^{(l)},\tau_2^{(l)})\Big)\Big|\leq{}&K^{\prime}n^{-l/2}\max_{m=1,\dots,l}n^{-m}n^{m+1}\\ 
=&O(n^{-l/2+1}).
}\vspace{-11mm}

\hfill $\Box$

\subsection{Auxiliary results}\label{sec::auxiliaryresultsquantspec}

\begin{lemma}\label{replacetaubytauhatlem}
Under the assumptions of Theorem \ref{weakconvintegspectrum},
\al{
\widehat{\specdis}\phantom{F\!\!\!\!\!}_{n,R}(\lambda ;\tau_1,\tau_2)=\widehat{\specdis}\phantom{F\!\!\!\!\!}_{n,U}(\lambda,\hat{\tau}_1,\hat{\tau}_2)+o_{\Prob}(n^{-1/2}).
} 
\end{lemma}

\begin{proof}
Let $\hat{\tau}_1=\hat{F}_{n,U}^{-1}(\tau_1)$ and $\hat{\tau}_2=\hat{F}_{n,U}^{-1}(\tau_2)$, where $$\hat{F}_{n,U}^{-1}(\tau) := \inf\{q \in \R: \tau \leq \hat{F}_{n,U}(q)\}$$ is the generalized inverse of the empirical distribution function $\hat{F}_{n,U}$. Then, from (\ref{empautocopU}), we have
\al{
\specdis_{n,U}(\lambda,\hat{\tau}_1,\hat{\tau}_2)=\frac{1}{2\pi}\sum_{|k|\leq n-1}w_{n,\lambda}(k)\frac{n-|k|}{n}\hat{\gamma}_k^{U}(\hat{\tau}_1,\hat{\tau}_2).
}
By  the representation (\ref{lagrepresspecdisemp}) of $\hat\specdis_{n,R}(\lambda ;\tau_1,\tau_2)$, we obtain
\al{
&\sqrt{n}\big(\widehat{\specdis}\phantom{F\!\!\!\!\!}_{n,R}(\lambda ;\tau_1,\tau_2)-\widehat{\specdis}\phantom{F\!\!\!\!\!}_{n,U}(\lambda,\hat{\tau}_1,\hat{\tau}_2)\big)\\
={}&\sqrt{n}\frac{1}{2\pi}\sum_{|k|\leq n-1}w_{n,\lambda}(k)\frac{n-|k|}{n}\big(\hat{\gamma}_k^{R}(\tau_1,\tau_2)-\hat{\gamma}_k^{U}(\hat{\tau}_1,\hat{\tau}_2)\big),
} 

where
\al{
&\big|\hat{\gamma}_k^{R}(\tau_1,\tau_2)-\hat{\gamma}_k^{U}(\hat{\tau}_1,\hat{\tau}_2)\big|\\
\leq{}&\frac{1}{n-|k|}\sum_{t\in\mathcal{T}_k}\big|I\{\hat{F}_{n,U}(U_{t+k})\leq \tau_1\}I\{\hat{F}_{n,U}(U_{t})\leq \tau_2\}\\
&\qquad\qquad-I\{U_{t+k}\leq\hat{F}_{n,U}^{-1}(\tau_1)\}I\{U_{t}\leq\hat{F}_{n,U}^{-1}(\tau_2)\}\big|\\
={}&\frac{1}{n-|k|}\sum_{t\in\mathcal{T}_k}\big|I\{\hat{F}_{n,U}(U_{t+k})\leq \tau_1\}\big(I\{\hat{F}_{n,U}(U_{t})\leq \tau_2\}-I\{U_{t}\leq\hat{F}_{n,U}^{-1}(\tau_2)\}\big)\\
&\qquad\qquad+\big(I\{\hat{F}_{n,U}(U_{t+k})\leq \tau_1\}-I\{U_{t+k}\leq\hat{F}_{n,U}^{-1}(\tau_1)\}\big)I\{U_{t}\leq\hat{F}_{n,U}^{-1}(\tau_2)\}\big|\\
\leq{}&\frac{1}{n-|k|}\sum_{t\in\mathcal{T}_k}\Big[\big|I\{\hat{F}_{n,U}(U_{t})\leq \tau_2\}-I\{U_{t}\leq\hat{F}_{n,U}^{-1}(\tau_2)\}\big|\\
&\qquad\qquad+\big|I\{\hat{F}_{n,U}(U_{t+k})\leq \tau_1\}-I\{U_{t+k}\leq\hat{F}_{n,U}^{-1}(\tau_1)\}\big|\Big].
}

Observing that $$I\{\hat{F}_{n,U}(U_{t})< \tau_2\}=I\{U_{t}<\hat{F}_{n,U}^{-1}(\tau_2)\}$$ since $x<F^{-1}(u)$ if and only if~$F(x)<u$ for any distribution function $F$ and that, similarly, $$I\{\hat{F}_{n,U}(U_{t+k})\leq \tau_1\}=I\{U_{t+k}\leq\hat{F}_{n,U}^{-1}(\tau_1)\},$$
we have 
\al{
&\big|I\{\hat{F}_{n,U}(U_{t})\leq\tau_2\}-I\{U_{t}\leq\hat{F}_{n,U}^{-1}(\tau_2)\}\big|=\big|I\{U_{t}=\hat{F}_{n,U}^{-1}(\tau_2)\}-I\{\hat{F}_{n,U}(U_{t})= \tau_2\}\big|
}
and
\al{
&\big|I\{\hat{F}_{n,U}(U_{t+k})\leq\tau_1\}-I\{U_{t+k}\leq\hat{F}_{n,U}^{-1}(\tau_1)\}\big|=\big|I\{U_{t+k}=\hat{F}_{n,U}^{-1}(\tau_1)\}\\
&\hspace{10.3cm}-I\{\hat{F}_{n,U}(U_{t+k})= \tau_1\}\big|.
}

Furthermore, 
$$\begin{array}{rcl} U_{t}=\hat{F}_{n,U}^{-1}(\tau_2) & 
 \text{ if } & I\{\hat{F}_{n,U}(U_{t})= \tau_2\}=1 \\ 
  U_{t+k}=\hat{F}_{n,U}^{-1}(\tau_1) & 
    \text{ if } & I\{\hat{F}_{n,U}(U_{t+k})= \tau_1\}=1.
    \end{array}
    $$
     Hence, the second indicator is never greater than the first one, whence, for any $l\in\N$,
\al{
\frac{n-|k|}{n}\big|\hat{\gamma}_k^{R}(\tau_1,\tau_2)-\hat{\gamma}_k^{U}(\hat{\tau}_1,\hat{\tau}_2)\big|\leq{}&\frac{1}{n}\sum_{t\in\mathcal{T}_k}\big[I\{U_{t}=\hat{F}_{n,U}^{-1}(\tau_2)\}+I\{U_{t+k}=\hat{F}_{n,U}^{-1}(\tau_1)\}\big]\\
\leq{}&\frac{1}{n}\sum_{t=0}^{n-1}\big[I\{U_{t}=\hat{F}_{n,U}^{-1}(\tau_2)\}+I\{U_{t+k}=\hat{F}_{n,U}^{-1}(\tau_1)\}\big]\\
\leq{}&2\sup_{\tau\in[0,1]}\big|\hat{F}_{n,U}(\tau)-\hat{F}_{n,U}(\tau -)\big|\\
={}&O_{\Prob}(n^{-1+1/(2l)})
}
where $\hat{F}_{n,U}(\tau -):=\lim_{\xi\downarrow 0}\hat{F}_{n,U}(\tau-\xi)$ and the above $O_{\Prob}$-bound is a consequence of Lemma~8.6 of \cite{kleysupp16}.

Moreover, by (\ref{estimationweights}),
\al{
\sum_{|k|\leq n-1}|w_{n,\lambda}(k)|=O(\log(n))
}
and thus, altogether, for any $l\in\N$,
\al{
\sqrt{n}\big(\widehat{\specdis}\phantom{F\!\!\!\!\!}_{n,R}(\lambda ;\tau_1,\tau_2)-\widehat{\specdis}\phantom{F\!\!\!\!\!}_{n,U}(\lambda ;\tau_1,\tau_2)\big)=O_{\Prob}(n^{-1/2+1/(2l)}\log(n)).
}

This concludes the proof.
\end{proof}

\begin{lemma}\label{withoutlagzero}
Under the assumptions of Theorem \ref{weakconvintegspectrum}, 
\als{\label{withoutlagzeroranks}
\mathbb{G}_{n,R}(\lambda ;\tau_1,\tau_2)={}&\sqrt{n}\big(\specdisred_{n,R}(\lambda ;\tau_1,\tau_2)-\specdisred (\lambda ;\tau_1,\tau_2)\big)+o_{\Prob}(1)\\
\label{withoutlagzeroU}
\mathbb{G}_{n,U}(\lambda ;\tau_1,\tau_2)={}&\sqrt{n}\big(\specdisred_{n,U}(\lambda ;\tau_1,\tau_2)-\specdisred (\lambda ;\tau_1,\tau_2)\big)+o_{\Prob}(1).
}
\end{lemma}

\begin{proof}
The result follows if we show that 
\al{
\frac{1}{2\pi}\frac{2\pi}{n}\sum_{s=1}^{n-1}I\big\{0\leq \frac{2\pi s}{n}\leq\lambda\big\}\hat{\gamma}_0^R(\tau_1,\tau_2)=\frac{\lambda}{2\pi}(\tau_1\wedge\tau_2-\tau_1\tau_2)+o_{\Prob}(n^{-1/2}).
}

As the indicator is of bounded variation, we have
\al{
\frac{1}{2\pi}\frac{2\pi}{n}\sum_{s=1}^{n-1}I\big\{0\leq \frac{2\pi s}{n}\leq\lambda\big\}=\frac{\lambda}{2\pi}+O(n^{-1}).
}

Furthermore, since $F$ is assumed to be continuous, the ranks of $X_0,\dots,X_{n-1}$ are almost surely the same as the ranks of $F(X_0),\dots,F(X_{n-1})$, i.e.\ we can, without loss of generality, assume the marginals to be uniformly distributed and, letting $a:=\tau_1,\, b:=\tau_2$ in \eqref{gammahatranksdef}, write 
\al{
\hat{\gamma}_0^R(\tau_1,\tau_2)={}&n^{-1}\sum_{t=0}^{n-1}\big(I\{\hat{F}_{n}(X_{t})\leq \tau_1\}-\tau_1\big)\big(I\{\hat{F}_{n}(X_{t})\leq \tau_2\}-\tau_2\big)\\
={}&n^{-1}\sum_{t=0}^{n-1}\big(I\{\hat{F}_{n,U}(U_{t})\leq \tau_1\}-\tau_1\big)\big(I\{\hat{F}_{n,U}(U_{t})\leq \tau_2\}-\tau_2\big)\qquad\text{a.s.}\\
={}&n^{-1}\sum_{t=0}^{n-1}I\{\hat{F}_{n,U}(U_{t})\leq \tau_1\wedge\tau_2\}-n^{-1}\tau_1\sum_{t=0}^{n-1}I\{\hat{F}_{n,U}(U_{t})\leq \tau_2\}\\
&-n^{-1}\tau_2\sum_{t=0}^{n-1}I\{\hat{F}_{n,U}(U_t)\leq \tau_1\}+\tau_1\tau_2.
}
Next,  as in equation (A.4) in \cite{DetteEtAl2016}, for any $l\in\N$, 
\als{\label{diffranksU}
&\sup_{\tau\in[0,1]}\big|n^{-1}\sum_{t=0}^{n-1}I\{\hat{F}_{n,U}(U_{t})\leq \tau\}-n^{-1}\sum_{t=0}^{n-1}I\{U_{t}\leq \hat{F}_{n,U}^{-1}(\tau)\}\big|\notag\\
&\hspace{3cm}\leq{}\sup_{\tau\in[0,1]}\big|\hat{F}_{n,U}(\tau)-\hat{F}_{n,U}(\tau -)\big|=O_{\Prob}(n^{-1+1/(2l)})
}
where $\hat{F}_{n,U}(\tau -):=\lim_{\xi\downarrow 0}\hat{F}_{n,U}(\tau-\xi)$ and
\als{\label{diffFinvtau}
\big|n^{-1}\sum_{t=0}^{n-1}I\{U_{t}\leq \hat{F}_{n,U}^{-1}(\tau)\}-\tau\big|\leq{}&\Big|\frac{\ceil{n\tau}}{n}-\tau\biggr|\leq n^{-1}.
}

The result now follows by applying properties (\ref{diffranksU}) and (\ref{diffFinvtau}) in 
\al{
&\sup_{\tau_1,\tau_2\in[0,1]}\big|\hat{\gamma}_0^R(\tau_1,\tau_2)-(\tau_1\wedge\tau_2-\tau_1\tau_2)\big|\\
\leq{}&\sup_{\tau_1,\tau_2\in[0,1]}\Big\{\big|n^{-1}\sum_{t=0}^{n-1}I\{\hat{F}_{n}(U_{t})\leq \tau_1\wedge\tau_2\}-n^{-1}\tau_1\sum_{t=0}^{n-1}I\{\hat{F}_{n}(U_{t})\leq \tau_2\}\\
&-n^{-1}\tau_2\sum_{t=0}^{n-1}I\{\hat{F}_{n}(U_{t})\leq \tau_1\}-\big[n^{-1}\sum_{t=0}^{n-1}I\{U_{t}\leq \hat{F}_{n,U}^{-1} (\tau_1\wedge\tau_2)\}\\
&-n^{-1}\tau_1\sum_{t=0}^{n-1}I\{U_t\leq \hat{F}_{n,U}^{-1}(\tau_2)\}-n^{-1}\tau_2\sum_{t=0}^{n-1}I\{U_{t}\leq \hat{F}_{n,U}^{-1}(\tau_1)\}\big]\big|\\
&+\big|n^{-1}\sum_{t=0}^{n-1}I\{U_{t}\leq \hat{F}_{n,U}^{-1} (\tau_1\wedge\tau_2)\}-n^{-1}\tau_1\sum_{t=0}^{n-1}I\{U_t\leq \hat{F}_{n,U}^{-1}(\tau_2)\}+\tau_1\tau_2\\
&-n^{-1}\tau_2\sum_{t=0}^{n-1}I\{U_{t}\leq \hat{F}_{n}^{-1}(\tau_1)\}-(\tau_1\wedge\tau_2+\tau_1\tau_2-\tau_1\tau_2-\tau_1\tau_2)\big|\Big\},
}

whence, for any $l\in \N$,
\al{
\frac{1}{2\pi}\frac{2\pi}{n}\sum_{s=1}^{n-1}I\big\{0\leq \frac{2\pi s}{n}\leq\lambda\big\}\hat{\gamma}_0^R(\tau_1,\tau_2)=\frac{\lambda}{2\pi}(\tau_1\wedge\tau_2-\tau_1\tau_2)+O_{\Prob}(n^{-1+1/(lk)}).
}

This concludes the proof of (\ref{withoutlagzeroranks}). Assertion (\ref{withoutlagzeroU}) follows with similar arguments since 
\al{
\big|n^{-1}\sum_{t=0}^{n-1}I\{U_{t}\leq \tau\}-\tau\big|\leq{}&\Big|\frac{\ceil{n\tau}}{n}-\tau\biggr|\leq n^{-1}
}
 and hence
\al{
&\sup_{\tau_1,\tau_2\in[0,1]}\big|\hat{\gamma}_0^U(\tau_1,\tau_2)-(\tau_1\wedge\tau_2-\tau_1\tau_2)\big|=O_{\Prob}(n^{-1}).
}

\vspace{-9.5mm}

\end{proof}\medskip

\begin{lemma}\label{expF-F}
Under Assumption (CS) with $p=2$ and $l\geq 1$, 
\al{\sup_{\substack{\tau_1,\tau_2\in[0,1]\\ \lambda\in[0,\pi]}}\llvert \E[\widehat{\specdis}\phantom{F\!\!\!\!\!}_{n,U}(\lambda ;\tau_1,\tau_2)]-\specdis(\lambda ;\tau_1,\tau_2)\rrvert=O(n^{-1}).
}
\end{lemma}

\begin{proof}
First, for $\omega_{jn}=\frac{2\pi j}{n}$, $j\in \Z$, 
\al{
&\llvert \E[\widehat{\specdis}\phantom{F\!\!\!\!\!}_{n,U}(\lambda ;\tau_1,\tau_2)]-\specdis(\lambda ;\tau_1,\tau_2)\rrvert\\
={}&\llvert\frac{2\pi}{n}\sum_{j=1}^{n-1}I\{0\leq\omega_{jn} \leq\lambda\}\E[\mathcal{I}_{n,U}^{\tau_1,\tau_2}\left(\omega_{jn}\right)]-\specdis(\lambda ;\tau_1,\tau_2)\rrvert.
}
By Lemma 1.4 (or the remark thereafter) in the online appendix of \cite{DetteEtAl2016}, we have, for $j\neq 0 \mod n$, 
\al{
\E[\mathcal{I}_{n,U}^{\tau_1,\tau_2}\left(\omega_{jn}\right)]=\spec(\omega_{jn}; \tau_1,\tau_2)+\epsilon_n^{\tau_1,\tau_2}(\omega_{jn})
}
with $\sup_{\tau_1,\tau_2\in[0,1],\omega\in\R}|\epsilon_n^{\tau_1,\tau_2}(\omega)|=O(n^{-1})$. Therefore,
\al{
\big\vert \E[\widehat{\specdis}\phantom{F\!\!\!\!\!}_{n,U}(\lambda ;\tau_1,\tau_2)]-\specdis(\lambda ;\tau_1,\tau_2)\big\vert={}&\Big\vert \frac{2\pi}{n}\sum_{j=1}^{n-1}I\{0\leq\omega_{jn} \leq\lambda\}\spec(\omega_{jn}; \tau_1,\tau_2)-\specdis(\lambda ;\tau_1,\tau_2)\\
&+\frac{2\pi}{n}\sum_{j=1}^{n-1}I\{0\leq\omega_{jn} \leq\lambda\}\epsilon_n^{\tau_1,\tau_2}(\omega_{jn})\Big\vert.
}

Assumption (CS) implies that $\omega\mapsto\spec(\omega; \tau_1,\tau_2)$ has bounded and uniformly continuous derivatives of order $\leq l$, that is $\omega\mapsto\spec(\omega; \tau_1,\tau_2)$ is of finite total variation on the interval~$[0,2\pi]$. Moreover, the indicator function $\omega\mapsto I\{0\leq\omega\leq\lambda\}$ is also of finite total variation. Then, their product $\omega\mapsto I\{0\leq\omega\leq\lambda\}\spec(\omega; \tau_1,\tau_2)$ is of finite total variation $V$, and we obtain
\al{
&\Big|\int_0^{2\pi}I\{0\leq\omega\leq\lambda\}\spec(\omega; \tau_1,\tau_2)d\omega-\frac{2\pi}{n}\sum_{j=1}^{n-1}I\{0\leq\omega_{jn} \leq\lambda\}\spec(\omega_{jn}; \tau_1,\tau_2)\biggr|\notag\\
\leq{}&\int_0^{\frac{2\pi}{n}}\sum_{j=1}^{n-1}\Big|I\big\{0\leq\omega+\frac{2\pi(j-1)}{n} \leq\lambda\big\}\spec\left(\omega+\frac{2\pi(j-1)}{n}; \tau_1,\tau_2\right)\notag\\
&-I\big\{0\leq \frac{2\pi j}{n}\leq \lambda\big\}\spec\left(\frac{2\pi j}{n}; \tau_1,\tau_2\right)\biggr|d\omega\leq{}\int_0^{\frac{2\pi}{n}}V d\omega=\frac{2\pi}{n}V.
}

%
%
%
%

Hence,
\al{
\sup_{\substack{\tau_1,\tau_2\in[0,1]\\ \lambda\in[0,\pi]}}\llvert \E[\widehat{\specdis}\phantom{F\!\!\!\!\!}_{n,U}(\lambda ;\tau_1,\tau_2)]-\specdis(\lambda ;\tau_1,\tau_2)\rrvert\leq{}& \frac{2\pi}{n}V+\frac{2\pi(n-1)}{n}\sup_{\tau_1,\tau_2\in[0,1],\omega\in\R}|\epsilon_n^{\tau_1,\tau_2}(\omega)|\\
={}& O(n^{-1}),
}
which concludes the proof. 
\end{proof}

\begin{lemma} \label{lemSixthMomIncrementHn}
Let $X_0,\ldots,X_{n-1}$ be the finite realization of a strictly stationary process with $X_0 \sim U[0,1]$, and for $x =(x_1, x_2)$ and $y=(y_1,y_2)$ let 
\al{
\mathbb{H}_n^U(x,y;\beta):= \begin{cases}C_{\beta}\overline{\mathbb{H}}_n^U(x,y;\beta),&\text{ if } \beta\in(0,\pi],\\
0,&\text{ if } \beta=0,
\end{cases}
} 
where 
\al{
\overline{\mathbb{H}}_n^U(x,y;\beta):=\sqrt{nb_{\beta}}(\tilde{\mathbb{H}}_n^U(x,y;\beta)-\E[\tilde{\mathbb{H}}_n^U(x,y;\beta)]),
}
with 
\al{
&\tilde{\mathbb{H}}_n^U(x,y;\beta)=\frac{2\pi}{n}\sum_{s=1}^{n-1}W_{n,\beta}(a_{\beta}-2\pi s/n)\Big\{\mathcal{I}_{n,U}^{x_1,x_2}(2\pi s/n)-\mathcal{I}_{n,U}^{y_1,y_2}(2\pi s/n)\biggr\},\\
&W_{n,\beta}(u)=\sum_{j=-\infty}^{\infty}b_{\beta}^{-1}W(b_{\beta}^{-1}(u+2\pi j)),
}
and the weight function $W(\cdot)$ bounded, real-valued and even, with support $[-\pi,\pi]$.\\

For any Borel set $A$, define
\al{
d^A_n(\omega):= \sum_{t=0}^{n-1}I\{X_t \in A\} e^{-it\omega}.
}
Assume that, for $p=1,\ldots,P$, there exist a constant $C$ and a
function $g\,:\,  \R^+ \rightarrow \R^+$, both independent of $\omega
_1,\ldots,\omega_p \in\R, n$ and $A_1,\ldots,A_p$, such that
\al{
\big| \cum\big(d_n^{A_1}(
\omega_1), \ldots, d_n^{A_p}(
\omega_p)\big) \big| \leq C \Big( \Big|
\Delta_n \Big(\sum_{i=1}^p
\omega_i \Big) \biggr| + 1 \Big) g(\varepsilon)
}
for any Borel sets $A_1, \ldots, A_p$ with $\min_j \Prob(X_0 \in A_j)
\leq\varepsilon$. Then, for $\beta\in (0,\pi]$, there exists a constant $K$ (depending on
$C$, $L$, and $g$ only) such that
\al{
\E\big|\mathbb{H}_n^U(x,y;\beta) \big|
^{2L} \leq K_1 \llVert W\rrVert_{\infty}^{2L}C_{\beta}^{2L}\sum_{l=0}^{L-1}
\frac{g^{L-l}(\llVert x-y\rrVert_1)}{(nb_{\beta})^{l}}
}
for all $x,y$ with $g(\llVert x-y\rrVert_1)<1$.
%
\end{lemma}

%

\begin{proof} First note that, for $\beta\neq 0$,
\al{
\E\big| \mathbb{H}_n^U(x,y;\beta) \big|
^{2L} = C_{\beta}^{2L}\E\big| \overline{\mathbb{H}}_n^U(x,y;\beta) \big|
^{2L}.
}
Repeating the arguments of the proof of Lemma A.2 in \cite{DetteEtAl2016} yields  the representation
\begin{equation}
\label{lemSixthMomIncrementHneqnRepr} \E\big| \overline{\mathbb{H}}_n^U(x,y;\beta) \big|
^{2L} = \mathop{\sum
_{\{\nu_1, \ldots, \nu_R\} }}_{{\llvert  \nu_j\rrvert   \geq2,  j=1,\ldots,R}} \prod_{r=1}^R
\mathcal{D}_{x,y}(\nu_r),
\end{equation}
where the summation runs over all partitions $\{\nu_1, \ldots, \nu
_R\}$ of $\{1,\ldots,2L\}$ such that each set~$\nu_j$ contains at
least two elements, and
\al{
\mathcal{D}_{x,y}(\xi) :={}& \sum_{\ell_{\xi_1}, \ldots, \ell
_{\xi_{q}} \in\{1,2\}}
n^{-3 q/2} b_{\beta}^{q/2} \Big(\prod
_{m \in\xi} \sigma_{\ell_m} \Big)
\\
\times&\sum_{s_{\xi_1}, \ldots, s_{\xi_{q}}=1}^{n-1} \Big(
\prod_{m \in\xi} W_{n,\beta}(a_{\beta}- 2\pi
s_{m} / n) \Big) \cum( D_{\ell_m, (-1)^{m-1} s_m }:  m \in\xi),
}
for any set $\xi:= \{\xi_1, \ldots, \xi_q\} \subset\{1,\ldots,2L\}
$, where $q:=\llvert  \xi\rrvert  $ and
\[
D_{\ell,s}:= d_n^{M_1(\ell)}(2\uppi s/{n})
d_n^{M_2(\ell)}(-2\uppi s/{n}), \qquad \ell=1,2, \qquad s=1,
\ldots,n-1,
\]
with the sets $M_1(1)$, $M_2(2)$, $M_2(1)$, $M_1(2)$ and the signs
$\sigma_{\ell} \in\{-1,1\}$ defined by
\begin{align*}
\sigma_1 &:= 2 I\{x_1 > y_1\} - 1,\qquad
&\sigma_2&:= 2 I\{x_2 > y_2\} - 1,
\nonumber
\\
M_1(1) &:= (x_1 \wedge y_1, x_1
\vee y_1], \qquad &M_2(2)&:= (x_2 \wedge
y_2, x_2 \vee y_2], \notag
\\
M_2(1) &:= %
\begin{cases} [0, x_2], &
\quad y_2 \geq x_2,
\vspace*{3pt}\cr
[0, y_2], &
\quad x_2 > y_2,\end{cases}\qquad &M_1(2)  &:=
\begin{cases} [0, y_1], &\quad y_2 \geq
x_2 ,
\vspace*{3pt}\cr
[0, x_1], &\quad x_2 >
y_2.\end{cases}\notag
\end{align*}

Then, we obtain, similarly as in the proof of Lemma A.2 in \cite{kley14}, 
\al{
\big\llvert \mathcal{D}_{a,b}(\xi)\big\rrvert 
\leq{}&K_{q,g,C}
n^{-3q/2} b_{\beta}^{q/2} 2^q g(\varepsilon) \sum
_{\{
\mu_1, \ldots, \mu_N\}} \sum_{I \subset\{1,\ldots,N\}} \mathop{\sum_{(s_{\xi_1}, \ldots, s_{\xi_q}) \in
S_n(\mu,I)}}\\
&{} \Big( \prod_{m \in\xi} \big\llvert W_{n,\beta}(a_{\beta}- 2\pi
s_{m} / n) \big\rrvert \Big) n^{\llvert  I\rrvert  }\\
\leq{} &K_{q,g,C}n^{-3q/2} b_{\beta}^{q/2} 2^q g(\varepsilon) C_{q}\max_{N\leq q}\max_{|I|\leq N}\big(\sup_{u\in\R}|W_{n,\beta}(u)|\big)^{|I|-\floor{|I|/N}}n^{\llvert  I\rrvert  }\\
&{}  \times \Big(\sum_{s=1}^{n-1}\big\llvert W_{n,\beta}(a_{\beta}- 2\pi s / n) \big\rrvert \Big)^{q-(|I|-\floor{|I|/N})},
}
where summation runs over all indecomposable partitions $\{\mu_1,\dots,\mu_N\}$ of the scheme 
\begin{center}
\begin{tabular}{cc}
$(\xi_1,1)$&$(\xi_1,2)$\\
\vdots&\vdots\\
$(\xi_q,1)$&$(\xi_q,2)$
\end{tabular}
\end{center}
and 
\al{
S_n(\mu,I):={}& \Big\{ (s_{\xi_1}, \ldots,
s_{\xi_q}) \in\{1,\ldots,n-1\}^q \Big|
\sum_{(m,k) \in\mu_j} (-1)^{k+m} s_m
\in n \IZ, \forall\mu_j \in\mu, j \in I \biggr\}.
}

Furthermore, by assumption, the function $W(\cdot)$ has support $[-\pi,\pi]$ and hence, there is at most one $j\in\Z$ such that $W\big(b_{\beta}^{-1}(\alpha+2\pi j)\big)\neq 0$. Denote this integer by $j_{\alpha,b_{\beta}}$. Therefore,

%
\al{
\llvert  W_{n,\beta}(\alpha)\rrvert ={}&\Big|  \sum_{j=-\infty}^{\infty}b_{\beta}^{-1}W\big(b_{\beta}^{-1}(\alpha+2\pi j)\big)\biggr| ={}b_{\beta}^{-1}\llvert  W\big(b_{\beta}^{-1}(\alpha+2\pi j_{\alpha,b_{\beta}})\big)\rrvert\leq{}b_{\beta}^{-1}\llVert  W\rrVert  _{\infty} 
}
and, with similar arguments,
\al{
\sum_{s=1}^{n-1} \big| W_{n,\beta}(a_{\beta}- 2\pi
s / n) \big| ={}&\sum_{s=1}^{n-1} \Big| \sum_{j=-\infty}^{\infty}b_{\beta}^{-1}W\big(b_{\beta}^{-1}(a_{\beta}-2\pi s/n+2\pi j)\big)\biggr|\\
={}&b_{\beta}^{-1}\sum_{s=1}^{n-1} \llvert  W\big(b_{\beta}^{-1}(\alpha_{\beta}-2\pi s/n+2\pi j_{\alpha,b_{\beta}})\big)\rrvert\\
\leq{}&Cn\llVert  W\rrVert  _{\infty},
}
where we have used the fact  that 
\al{
&\sum\limits_{s=1}^{n-1} \llvert W\big(b_{\beta}^{-1}(\alpha_{\beta}-2\pi s/n+2\pi j_{\alpha,b_{\beta}})\big)\rrvert\\
\leq{}&\llVert  W\rrVert  _{\infty}\sum\limits_{s=1}^{n-1}I\{-\pi\leq b_{\beta}^{-1}(\alpha_{\beta}-2\pi s/n+2\pi j_{\alpha,b_{\beta}})\leq\pi\}
}
and the fact that the number of summands that are equal to one is less than $nb_{\beta}$ since $$b_{\beta}^{-1}(\alpha_{\beta}-2\pi s/n+2\pi j_{\alpha,b_{\beta}})$$ lies in the support of $W$ for at most $nb_{\beta}$ values of $s\in\{1,\dots,n-1\}$.\\

Therefore, 
\begin{eqnarray*}
\big\llvert \mathcal{D}_{x,y}(\xi)\big\rrvert &\leq&\tilde{K}_{q,g,C}\llVert  W\rrVert  _{\infty}^qn^{-3q/2} b_{\beta}^{q/2} g(\varepsilon) \max_{N\leq q}\max_{|I|\leq N} n^{q+\floor{|I|/N}} (b_{\beta}^{-1})^{|I|-\floor{|I|/N}}n^{\llvert  I\rrvert  }\\
&\leq&\tilde{K}_{q,g,C}\llVert  W\rrVert  _{\infty}^q(nb_{\beta})^{1-q/2}  g(\varepsilon),
\end{eqnarray*}
and hence
\als{\label{dxy}
\big\llvert \prod_{r=1}^R
\mathcal{D}_{x,y}(\nu_r) \big\rrvert \leq{}&
\tilde K_{L,g,C} \llVert  W\rrVert _{\infty}^{2L}(nb_{\beta})^{R-L}g^R(\varepsilon)
}
as $\sum_{r=1}^R \llvert  \nu_r\rrvert   = 2L$. The proof is complete by combining the estimates (\ref{dxy}) and (\ref{lemSixthMomIncrementHneqnRepr}).
\end{proof}

\begin{lemma}\label{LemmaA.7neu}
Under the assumptions of Theorem \ref{mainprocess}, let $\delta_n$ be a sequence of non-negative real numbers. Assume that there exists $\gamma\in(0,1)$ such that $\delta_n=O(n^{-1/\gamma})$. Then,  as $n\to\infty$,
\al{
\sup_{\substack{(\lambda ;\tau_1,\tau_2),(\lambda^{\prime},\tau_1^{\prime},\tau_2^{\prime})\in [0,\pi]\times[0,1]^2\\ \llVert(\lambda ;\tau_1,\tau_2)-(\lambda^{\prime},\tau_1^{\prime},\tau_2^{\prime})\rrVert_1\leq\delta_n}}|\overline{\mathbb{G}}_{n,U}(\lambda ;\tau_1,\tau_2)-\overline{\mathbb{G}}_{n,U}(\lambda^{\prime},\tau_1^{\prime},\tau_2^{\prime})|=o_{\Prob}(1).
}

\end{lemma}

\begin{proof}
Note that 
\al{
\sup_{\substack{(\lambda ;\tau_1,\tau_2),(\lambda^{\prime},\tau_1^{\prime},\tau_2^{\prime})\in [0,\pi]\times[0,1]^2\\ \llVert(\lambda ;\tau_1,\tau_2)-(\lambda^{\prime},\tau_1^{\prime},\tau_2^{\prime})\rrVert_1\leq\delta_n}}|\overline{\mathbb{G}}_{n,U}(\lambda ;\tau_1,\tau_2)-\overline{\mathbb{G}}_{n,U}(\lambda^{\prime},\tau_1^{\prime},\tau_2^{\prime})| \leq{}&S_n^{(1)}+S_n^{(2)},
}
where 
\al{
&S_n^{(1)}=\sup_{\substack{\lambda,\lambda^{\prime}\in [0,\pi]\\ |\lambda-\lambda^{\prime}|\leq \delta_n}}\sup_{\tau_1,\tau_2\in[0,1]^2}|\overline{\mathbb{G}}_{n,U}(\lambda ;\tau_1,\tau_2)-\overline{\mathbb{G}}_{n,U}(\lambda^{\prime},\tau_1,\tau_2)| \ \text{ and}\\
&S_n^{(2)}=\sup_{\lambda\in [0,\pi]}\sup_{\substack{(\tau_1,\tau_2),(\tau_1^{\prime},\tau_2^{\prime})\in[0,1]^2\\ \llVert (\tau_1,\tau_2)-(\tau_1^{\prime},\tau_2^{\prime})\rrVert_1\leq \delta_n }}|\overline{\mathbb{G}}_{n,U}(\lambda ;\tau_1,\tau_2)-\overline{\mathbb{G}}_{n,U}(\lambda,\tau_1^{\prime},\tau_2^{\prime})|.
}

To bound $S_n^{(1)}$, use (\ref{zuwachslambda}) to obtain 
\al{
&\overline{\mathbb{G}}_{n,U}(\lambda ;\tau_1,\tau_2)-\overline{\mathbb{G}}_{n,U}(\lambda^{\prime},\tau_1,\tau_2)\\
={}&\sqrt{n}\Big(\frac{2\pi}{n}\sum_{s=1}^{n-1}\big(I\{0\leq 2\pi s/n\leq \lambda\}-I\{0\leq 2\pi s/n \leq\lambda^{\prime}\}\big)\big(\mathcal{I}_{n,U}^{\tau_1,\tau_2}(2\pi s/n)\\
&-\E[\mathcal{I}_{n,U}^{\tau_1,\tau_2}(2\pi s/n)]\big)\Big)\\
={}& n^{-3/2}\sum_{s=1}^{n-1}\textrm{sign}(\lambda-\lambda^{\prime})I\{\lambda\wedge \lambda^{\prime}< 2\pi s/n\leq\lambda\vee \lambda^{\prime}\}\big(d_{n,U}^{\tau_1}(2\pi s/n)d_{n,U}^{\tau_2}(-2\pi s/n)\\
&-\E[d_{n,U}^{\tau_1}(2\pi s/n)d_{n,U}^{\tau_2}(-2\pi s/n)]\big)\Big).
}


From Lemmas A.6 and  A.4  in \cite{DetteEtAl2016}, we know that, for any $k\in\N$,
\al{
\sup_{y\in[0,1]}\sup_{\omega\in\R}|d_{n,U}^y(\omega)|=O_{\Prob}(n^{1/2+1/k})
}
and  that,  for $\varepsilon:=\min\{\tau_1,\tau_2\}$ and some constants $C$ and $d$ that do not depend on $s,\tau_1,\tau_2$,  
\al{
|\E[d_{n,U}^{\tau_1}(2\pi s/n)d_{n,U}^{\tau_2}(-2\pi s/n)]|={}&|\cum(d_{n,U}^{\tau_1}(2\pi s/n),d_{n,U}^{\tau_2}(-2\pi s/n))|\\
\leq{}&C\Big(\big|\Delta_n(0)\big|+1\Big)\varepsilon(|\log \varepsilon|+1)^d\\
={}&C(n+1)\varepsilon(|\log \varepsilon|+1)^d
}
for $s=1,\dots,n-1$, i.e.\ 
\al{
\sup_{\tau_1,\tau_2\in[0,1]^2}\sup_{s=1,\dots,n-1}|\E[d_{n,U}^{\tau_1}(2\pi s/n)d_{n,U}^{\tau_2}(-2\pi s/n)]|={}&O(n).
}

Observing that the sum $$\sum_{s=1}^{n-1}I\{\lambda\wedge \lambda^{\prime}< 2\pi s/n\leq\lambda\vee \lambda^{\prime}\}$$ contains at most $\ceil{|\lambda-\lambda^{\prime}|n/(2\pi)}$ non-zero summands, we have, for any $k\in\N$,
\al{
\sup_{\substack{0<\lambda,\lambda^{\prime}\in [0,\pi]\\ |\lambda-\lambda^{\prime}|\leq \delta_n}}\sup_{\tau_1,\tau_2\in[0,1]^2}\Big|&n^{-3/2}\sum_{s=1}^{n-1}\textrm{sign}(\lambda-\lambda^{\prime})I\{\lambda\wedge \lambda^{\prime}< 2\pi s/n\leq\lambda\vee \lambda^{\prime}\}\\
&\times d_{n,U}^{\tau_1}(2\pi s/n)d_{n,U}^{\tau_2}(-2\pi s/n)\Big|={}O_{\Prob}(\delta_n n^{1/2+2/k})
}
and
\al{
\sup_{\substack{0<\lambda,\lambda^{\prime}\in [0,\pi]\\ |\lambda-\lambda^{\prime}|\leq \delta_n}}\sup_{\tau_1,\tau_2\in[0,1]^2}\Big|n^{-3/2}\sum_{s=1}^{n-1}&\textrm{sign}(\lambda-\lambda^{\prime})I\{\lambda\wedge \lambda^{\prime}< 2\pi s/n\leq\lambda\vee \lambda^{\prime}\}\\
&\E\Big[d_{n,U}^{\tau_1}(2\pi s/n)d_{n,U}^{\tau_2}(-2\pi s/n)\Big]\Big| ={}O(\delta_n n^{1/2}).
}

Hence, 
\als{\label{abschzuwachslambda}
\sup_{\substack{\lambda,\lambda^{\prime}\in [0,\pi]\\ |\lambda-\lambda^{\prime}|\leq \delta_n}}\sup_{\tau_1,\tau_2\in[0,1]^2}&|\overline{\mathbb{G}}_{n,U}(\lambda ;\tau_1,\tau_2)-\overline{\mathbb{G}}_{n,U}(\lambda^{\prime},\tau_1,\tau_2)|=O_{\Prob}(\delta_n n^{1/2+2/k})=o_{\Prob}(1)
}
for $k$ sufficiently large.\\

Turning to $S_n^{(2)}$, observe that for $\lambda\in[0,\pi]$, we have
\al{
\widehat{\specdis}\phantom{F\!\!\!\!\!}_{n,U}(\lambda ;\tau_1,\tau_2)={}&\frac{2\pi}{n}\sum_{s=1}^{n-1}I\big\{0\leq \frac{2\pi s}{n}\leq\lambda\big\}\mathcal{I}_{n,U}^{\tau_1,\tau_2}\big(\frac{2\pi s}{n}\big)\\
={}&\frac{2\pi}{n}\sum_{s=1}^{n-1}I\big\{0\leq \frac{2\pi s}{n}\leq\lambda\big\}\frac{1}{2\pi n}d_{n,U}^{\tau_1}\big(\frac{2\pi s}{n}\big)d_{n,U}^{\tau_2}\big(-\frac{2\pi s}{n}\big)\\
={}&\frac{1}{n^2}\sum_{s=1}^{n-1}I\big\{0\leq \frac{2\pi s}{n}\leq\lambda\big\}\sum_{t_1=0}^{n-1}I\{U_{t_1}\leq \tau_1\}e^{-it_1\frac{2\pi s}{n}}\sum_{t_2=0}^{n-1}I\{U_{t_2}\leq \tau_2\}e^{it_2\frac{2\pi s}{n}}\\
={}&\frac{1}{2\pi}\sum_{|k|\leq n-1}\frac{2\pi}{n}\sum_{s=1}^{n-1}I\big\{0\leq \frac{2\pi s}{n}\leq\lambda\big\}e^{-ik\frac{2\pi s}{n}}\frac{1}{n}\sum_{t\in\mathcal{T}_k}I\{U_{t+k}\leq \tau_1\}I\{U_{t}\leq \tau_2\}\\
=:{}&\frac{1}{2\pi}\sum_{|k|\leq n-1}w_{n,k}(\lambda)\frac{n-|k|}{n}C_{n,k}(\tau_1,\tau_2),
}
%
where $\mathcal{T}_k:=\{t\in\{0,\dots,n-1\}|t,t+k\in\{0,\dots,n-1\}\}$, $k\in\{-(n-1),\dots,n-1\}$, $w_{n,\lambda}(k)$ as defined in (\ref{defweightsquantspec}), and 
%

\al{
&C_{n,k}(\tau_1,\tau_2):=\frac{1}{n-|k|}\sum_{t\in\mathcal{T}_k}I\{U_{t+k}\leq \tau_1\}I\{U_{t}\leq \tau_2\}.
}
Note that  $C_{n,k}$ is equivalent to $\gamma_k^{U}$ defined in (\ref{gammahatranksdef}) with $a,b:=0$.

%

Furthermore,
\al{
w_{n,\lambda}(0)=\frac{2\pi}{n}\sum_{s=1}^{n-1}I\big\{0\leq \frac{2\pi s}{n}\leq\lambda\big\}=\frac{2\pi M}{n}\leq \lambda\leq \pi,
}
where $M\in\{1,\dots,\floor{\frac{n}{2}}\}$ is the integer such that $\frac{2\pi M}{n}\leq \lambda<\frac{2\pi (M+1)}{n}$. With this notation, for $|k|=1,\dots,n-1$,
\al{
w_{n,\lambda}(k)={}&\frac{2\pi}{n}\sum_{s=1}^Me^{-ik\frac{2\pi s}{n}}
={}\frac{2\pi}{n}e^{-i\frac{\pi (M+1) k}{n}}\frac{\sin\left(\frac{\pi k M}{n}\right)}{\sin\left(\frac{\pi k }{n}\right)}.
}
Hence, for $|k|=1,\dots,\floor{\frac{n}{2}}$,
\al{
|w_{n,\lambda}(k)|={}&\frac{2\pi}{n}\frac{\big|\sin\left(\frac{\pi k M}{n}\right)\big|}{\big|\sin\left(\frac{\pi k }{n}\right)\big|}=\frac{2\pi}{n}\frac{\big|\sin\left(\frac{\pi k M}{n}\right)\big|}{\sin\left(\frac{\pi |k| }{n}\right)}\leq{}\frac{2\pi}{n}\frac{\big|\sin\left(\frac{\pi k M}{n}\right)\big|}{\frac{|k|}{n}}\leq{}\frac{2\pi}{|k|}
}
where we have used the fact  that $\sup\limits_{\omega\in\R}|\sin(\omega)|\leq 1$ and $\sin(\pi x)\geq x$ for $x\in[0,1/2]$. Similarly, for~$|k|=\floor{\frac{n}{2}}+1,\dots,n-1$,
\al{
|w_{n,\lambda}(k)|\leq{}&\frac{2\pi}{n}\frac{\big|\sin\left(\frac{\pi k M}{n}\right)\big|}{1-\frac{|k|}{n}}\leq{}\frac{2\pi}{n-|k|}
}
as $\sup_{\omega\in\R}|\sin(\omega)|\leq 1$ and $\sin(\pi x)\geq 1-x$ for $x\in[1/2,1]$.\\

Summing up,
\als{\label{estimationweights}
|w_{n,\lambda}(k)|\leq{}&\begin{cases}\pi,&\text{ if }k=0,\\
\frac{2\pi}{|k|},&\text{ if }0<|k|\leq \floor{\frac{n}{2}},\\
\frac{2\pi}{n-|k|},&\text{ if }\floor{\frac{n}{2}}<|k|\leq n-1.\end{cases}
}
Next, note that, letting $C_k:=\E[I\{U_{t+k}\leq \tau_1,U_t\leq \tau_2\}]$, we have 
\als{\label{diffGnweighted}
&\sup_{\lambda\in [0,\pi]}\sup_{\substack{(\tau_1,\tau_2),(\tau_1^{\prime},\tau_2^{\prime})\in[0,1]^2\\ \llVert (\tau_1,\tau_2)-(\tau_1^{\prime},\tau_2^{\prime})\rrVert_1\leq \delta_n }}|\overline{\mathbb{G}}_{n,U}(\lambda ;\tau_1,\tau_2)-\overline{\mathbb{G}}_{n,U}(\lambda,\tau_1^{\prime},\tau_2^{\prime})|\notag\\
={}&\sqrt{n}\sup_{\lambda\in [0,\pi]}\sup_{\substack{(\tau_1,\tau_2),(\tau_1^{\prime},\tau_2^{\prime})\in[0,1]^2\\ \llVert (\tau_1,\tau_2)-(\tau_1^{\prime},\tau_2^{\prime})\rrVert_1\leq \delta_n }}\big|\widehat{\specdis}\phantom{F\!\!\!\!\!}_{n,U}(\lambda ;\tau_1,\tau_2)-\widehat{\specdis}\phantom{F\!\!\!\!\!}_{n,U}(\lambda,\tau_1^{\prime},\tau_2^{\prime})\notag\\
&\hspace{5cm}-(\E[\widehat{\specdis}\phantom{F\!\!\!\!\!}_{n,U}(\lambda ;\tau_1,\tau_2)]-\E[\widehat{\specdis}\phantom{F\!\!\!\!\!}_{n,U}(\lambda,\tau_1^{\prime},\tau_2^{\prime})])\big|\notag\\
\leq{}&n^{-1/2}\frac{1}{2\pi}\sum_{|k|\leq n-1}\sup_{\lambda\in [0,\pi]}|w_{n,\lambda}(k)|(n-|k|)\sup_{\substack{(\tau_1,\tau_2),(\tau_1^{\prime},\tau_2^{\prime})\in[0,1]^2\\ \llVert (\tau_1,\tau_2)-(\tau_1^{\prime},\tau_2^{\prime})\rrVert_1\leq \delta_n }}
\Big|C_{n,k}(\tau_1,\tau_2)-C_{n,k}(\tau_1^{\prime},\tau_2^{\prime})\notag\\
&\hspace{75mm}-\big(C_k(\tau_1,\tau_2)-C_k(\tau_1^{\prime},\tau_2^{\prime})\big)\Big|.
}
In a first step, let us show that, for any $L\in\N$, there exists a constant $d_L$ such that 
%
\als{\label{Cnko_p(1)}
&(n-|k|)\sup_{\substack{(\tau_1,\tau_2),(\tau_1^{\prime},\tau_2^{\prime})\in[0,1]^2\\ \llVert (\tau_1,\tau_2)-(\tau_1^{\prime},\tau_2^{\prime})\rrVert_1\leq \delta_n }}
\Big|C_{n,k}(\tau_1,\tau_2)-C_{n,k}(\tau_1^{\prime},\tau_2^{\prime})-\big(C_k(\tau_1,\tau_2)-C_k(\tau_1^{\prime},\tau_2^{\prime})\big)\Big|\notag \\
&\hspace{8cm}=O_{\Prob}\big(n^{2/L}(\log(n))^{d_L/2}\big).
}
For this, for $\tau:=(\tau_1,\tau_2),\tau^{\prime}:=(\tau^{\prime}_1,\tau^{\prime}_2)\in[0,1]^2$ and fixed $k\in\{-(n-1),\dots,n-1\}$, let 
\al{
|C_{n,k}(\tau_1,\tau_2)-C_{n,k}(\tau^{\prime}_1,\tau^{\prime}_2)-\big(C_k(\tau_1,\tau_2)-C_k(\tau^{\prime}_1,\tau^{\prime}_2)\big)|\leq{}T_n^{(1)}
+T_n^{(2)}+T_n^{(3)}, 
}
where
\al{
T_n^{(1)}={}&\Big|C_{n,k}(\tau_1,\tau_2)-C_{n,k}\Big(\flt{1},\flt{2}\Big)\\
&\qquad-\Big(C_k(\tau_1,\tau_2)-C_k\Big(\flt{1},\flt{2}\Big)\Big)\Big|\\
T_n^{(2)}={}&\Big|C_{n,k}(\tau^{\prime}_1,\tau^{\prime}_2)-C_{n,k}\Big(\fltp{1},\fltp{2}\Big)\\
&\qquad-(C_k(\tau^{\prime}_1,\tau^{\prime}_2)-C_k\Big(\fltp{1},\fltp{2}\Big)\Big)\Big|\\
T_n^{(3)}={}&\Big|C_{n,k}\Big(\flt{1},\flt{2}\Big)-C_{n,k}\Big(\fltp{1},\fltp{2}\Big)\\
&\qquad-\Big(C_k\Big(\flt{1},\flt{2}\Big)-C_k\Big(\fltp{1},\fltp{2}\Big)\Big)\Big|.
}
By Theorem 2.2.4 in \cite{nelsen06} and since $\Big|\tau_i-\flt{i}\Big|\leq\frac{1}{n-|k|}$; $i=1,2$, we have 
\al{
T_n^{(1)}\leq{}&\Big|C_{n,k}(\tau_1,\tau_2)-C_{n,k}\Big(\flt{1},\flt{2}\Big)\Big|\\
&\qquad+\Big|\tau_1-\flt{1}\Big|+\Big|\tau_2-\flt{2}\Big|\\
\leq{}&\Big|C_{n,k}(\tau_1,\tau_2)-C_{n,k}\Big(\flt{1},\flt{2}\Big)\Big|+\frac{2}{n-|k|}.
}
As, for $a_1\geq a_2$ and $b_1\geq b_2$, 
\al{
C_{n,k}(a_1,b_1)-C_{n,k}(a_2,b_2)={}&\frac{1}{n-|k|}\sum_{t\in\mathcal{T}_k}(I\{U_{t+k}\in (a_2,a_1],U_{t}\in [0,b_1]\}\\
&+I\{U_{t+k}\in [0,a_2],U_{t}\in (b_2,b_1]\})\geq 0,
}
we have, since $\frac{\floor{(n-|k|)\tau}}{n-|k|}\leq\tau\leq\frac{\floor{(n-|k|)\tau}+1}{n-|k}$, 
\al{
T_n^{(1)}\leq{}&\Big|C_{n,k}\Big(\cet{1},\cet{2}\Big)-C_{n,k}\Big(\flt{1},\flt{2}\Big)\Big|\\
&\hspace{90mm} + \frac{2}{n-|k|}\\
\leq{}&\Big|C_{n,k}\Big(\cet{1},\cet{2}\Big)-C_{n,k}\Big(\flt{1},\flt{2}\Big)\\
-&\Big(C_{k}\Big(\cet{1},\cet{2}\Big)-C_{k}\Big(\flt{1},\flt{2}\Big)\Big)\Big|\\
&\hspace{90mm} +\frac{4}{n-|k|}.
} 
Analogously, 
\al{
T_n^{(2)}
\leq{}&\Big|C_{n,k}\Big(\cetp{1},\cetp{2}\Big)-C_{n,k}\Big(\fltp{1},\fltp{2}\Big)\\
-&\Big(C_{k}\Big(\cetp{1},\cetp{2}\Big)-C_{k}\Big(\fltp{1},\fltp{2}\Big)\Big)\Big|\\
&\hspace{90mm}+\frac{4}{n-|k|}.
}
By bounding $T_n^{(1)}$ and $T_n^{(2)}$, we have bounded the error made by evaluating the copulas on the points of the grid
\al{
M_{n,k}:=\Big\{\Big(\frac{i}{n-|k|},\frac{j}{n-|k|}\Big):i,j=0,\dots,n-|k|\Big\},
}
whereas the copulas in $T_n^{(3)}$ are already evaluated on the grid $M_{n,k}$ and, thus, do not have to be treated separately. \\

%
%

The cardinality of the set 
\al{
\mathcal{M}_{n,k}:=\{(\tau_1,\tau_2),(\tau_1^{\prime},\tau_2^{\prime})\in M_{n,k}:\llVert(\tau_1,\tau_2)-(\tau_1^{\prime},\tau_2^{\prime})\rrVert_1\leq\delta_n+\frac{4}{n-|k|}\}
}
is of the order $O\big((n-|k|)^4(\delta_n+(n-|k|)^{-1})\big)$. Hence, by Lemma 2.2.2 in \cite{vandervaart96} using $\Psi(x)=x^{2L}$ and the upper bounds on $T_n^{(1)}$ and $T_n^{(2)}$, 
\al{
&\E\Big[(n-|k|)\sup_{\substack{(\tau_1,\tau_2),(\tau^{\prime}_1,\tau^{\prime}_2)\in[0,1]^2\\ \lVert(\tau_1,\tau_2)-(\tau^{\prime}_1,\tau^{\prime}_2)\rVert_1\leq \delta_n}}\big|C_{n,k}(\tau_1,\tau_2)-C_{n,k}(\tau^{\prime}_1,\tau^{\prime}_2)-\big(C_k(\tau_1,\tau_2)-C_k(\tau^{\prime}_1,\tau^{\prime}_2)\big)\big|\Big]\\
\leq{}&3\E\Big[(n-|k|)^{1/2}\max_{\substack{(\tau_1,\tau_2),(\tau^{\prime}_1,\tau^{\prime}_2)\in M_{n,k}\\ \lVert(\tau_1,\tau_2)-(\tau^{\prime}_1,\tau^{\prime}_2)\rVert_1\leq \frac{4}{n-|k|}+\delta_n}}\sqrt{(n-|k|)}\big|C_{n,k}(\tau_1,\tau_2)-C_{n,k}(\tau^{\prime}_1,\tau^{\prime}_2)\\
&\hspace{70mm}-\big(C_k(\tau_1,\tau_2)-C_k(\tau^{\prime}_1,\tau^{\prime}_2)\big)\big|\Big]+8\\
\leq{}&\Lambda (n-|k|)^{1/2}\Big\{(n-|k|)^4(\delta_n+(n-|k|)^{-1})\Big\}^{1/(2L)}\max_{\substack{(\tau_1,\tau_2),(\tau^{\prime}_1,\tau^{\prime}_2)\in M_{n,k}\\ \lVert(\tau_1,\tau_2)-(\tau^{\prime}_1,\tau^{\prime}_2)\rVert_1\leq \frac{4}{n-|k|}+\delta_n}}\\
&\quad\Big(\E\Big[\sqrt{n-|k|}\big|C_{n,k}(\tau_1,\tau_2)-C_{n,k}(\tau^{\prime}_1,\tau^{\prime}_2)-\big(C_k(\tau_1,\tau_2)-C_k(\tau^{\prime}_1,\tau^{\prime}_2)\big)\big|^{2L}\Big]\Big)^{\frac{1}{2L}}+8
}
where $\Lambda <\infty$ is some adequate constant.

Then, from Lemma \ref{uniformlyboundedmomentsCnk}, it follows that for any $L\in\N$ there exist $C_L$ and $d_L$ such that 
\al{
&\max_{\substack{(\tau_1,\tau_2),(\tau^{\prime}_1,\tau^{\prime}_2)\in M_{n,k}\\ \lVert(\tau_1,\tau_2)-(\tau^{\prime}_1,\tau^{\prime}_2)\rVert_1\leq \frac{4}{n-|k|}+\delta_n}}\Big(\E\Big[\sqrt{n-|k|}\big|C_{n,k}(\tau_1,\tau_2)-C_{n,k}(\tau^{\prime}_1,\tau^{\prime}_2)
\\
&\hspace{5cm}
-\big(C_k(\tau_1,\tau_2)-C_k(\tau^{\prime}_1,\tau^{\prime}_2)\big)\big|^{2L}\Big]\Big)^{1/(2L)}\\
&\hspace{0cm}\leq{}C_L\Big(\big((\delta_n+(n-|k|)^{-1})(1+|\log (\delta_n+(n-|k|)^{-1})|)^{d_L}\big)\vee (n-|k|)^{-1}\Big)^{1/2}
}
and since, by assumption, $\delta_n=O(n^{-1/\gamma})=o(n^{-1})$ for $\gamma\in(0,1)$, 
\al{
&(\delta_n+(n-|k|)^{-1})(1+|\log (\delta_n+(n-|k|)^{-1})|)^{d_L}\vee (n-|k|)^{-1}
\leq{}cst(n-|k|)^{-1}(\log (n))^{d_L},
}
whence
\al{
&\max_{|k|\leq n-1}\E\Big[(n-|k|)\sup_{\substack{(\tau_1,\tau_2),(\tau^{\prime}_1,\tau^{\prime}_2)\in[0,1]^2\\ \lVert(\tau_1,\tau_2)-(\tau^{\prime}_1,\tau^{\prime}_2)\rVert_1\leq \delta_n}}\big|C_{n,k}(\tau_1,\tau_2)-C_{n,k}(\tau^{\prime}_1,\tau^{\prime}_2)\\
&\hspace{7cm}-\big(C_k(\tau_1,\tau_2)-C_k(\tau^{\prime}_1,\tau^{\prime}_2)\big)\big|\Big]\\
\leq{}&cst_L\max_{|k|\leq n-1}(n-|k|)^{1/2}\Big\{(n-|k|)^4(\delta_n+(n-|k|)^{-1})\Big\}^{1/(2L)}(n-|k|)^{-1/2}(\log (n))^{d_L/2}\\
\leq{}&cst_Ln^{2/L}(\log (n))^{d_L/2}.
}
Thus, (\ref{Cnko_p(1)}) holds for any $L\in\N$. Next, observe that for a constant $K$ that does not depend on~$\lambda$, 
\als{\label{sumoverweightslog}
\sum_{|k|\leq n-1}\sup_{\lambda\in [0,\pi]}|w_{n,\lambda}(k)|\leq K\Big(1+\sum_{0<|k|\leq n-1}\max\big\{\frac{1}{|k|},\frac{1}{n-|k|}\big\}\Big)=O\big(\log(n)\big).
}

Finally, choose $L=5$. Then, plugging (\ref{Cnko_p(1)}) and (\ref{sumoverweightslog}) into (\ref{diffGnweighted}) yields 
\als{\label{abschzuwachstau}
\sup_{\lambda\in [0,\pi]}\sup_{\substack{(\tau_1,\tau_2),(\tau_1^{\prime},\tau_2^{\prime})\in[0,1]^2\\ \llVert (\tau_1,\tau_2)-(\tau_1^{\prime},\tau_2^{\prime})\rrVert_1\leq \delta_n }}|\overline{\mathbb{G}}_{n,U}(\lambda ;\tau_1,\tau_2)-\overline{\mathbb{G}}_{n,U}(\lambda,\tau_1^{\prime},\tau_2^{\prime})|={}&O_{\Prob}\big(n^{-1/10}(\log(n))^{1+d_5/2}\big)\notag\\
={}&o_{\Prob}(1).
}

Equations (\ref{abschzuwachslambda}) and (\ref{abschzuwachstau}) together yield the desired result.
\end{proof}

\begin{lemma}\label{uniformlyboundedmomentsCnk}
For any $L\in\N$ there exist constants $K_L$ and $d_L$ such that 
\al{
&\sup_{\substack{(x_1,x_2)\in[0,1]^2,\\(y_1,y_2)\in[0,1]^2\\ \lVert(x_1,x_2)-\\(y_1,y_2)\rVert_1\leq \delta}}\E\Big[\Big(\sqrt{n-|k|}\big|C_{n,k}(x_1,x_2)-C_{n,k}(y_1,y_2)-\big(C_k(x_1,x_2)-C_k(y_1,y_2)\big)\big|\\
&\hspace{7cm}\times\big(\delta(1+|\log \delta|)^{d_L}\vee (n-|k|)^{-1}\big)^{-1/2}\Big)^{2L}\Big]\leq K_L.
}
\end{lemma}

\begin{proof}

First note that 
\al{
&\sup_{\substack{(x_1,x_2),(y_1,y_2)\in[0,1]^2\\ \lVert(x_1,x_2)-(y_1,y_2)\rVert_1\leq \delta}}\E[\sqrt{n-|k|}\big|C_{n,k}(x_1,x_2)-C_{n,k}(y_1,y_2)\\
&\hspace{6cm}-\big(C_k(x_1,x_2)-C_k(y_1,y_2)\big)\big|^{2L}]\\
\leq{}& 2^{2L-1}\Big(\sup_{\substack{x_1,y_1\in[0,1]\\ |x_1-y_1|\leq \delta}}\sup_{x_2\in[0,1]}\E\big[\sqrt{n-|k|}\big|C_{n,k}(x_1,x_2)-C_{n,k}(y_1,x_2)\\
&\hspace{6cm}-\big(C_k(x_1,x_2)-C_k(y_1,x_2)\big)\big|^{2L}\big]\\
&+\sup_{\substack{x_2,y_2\in[0,1]\\ |x_2-y_2|\leq \delta}}\sup_{y_1\in[0,1]}\E\big[\sqrt{n-|k|}\big|C_{n,k}(y_1,x_2)-C_{n,k}(y_1,y_2)\\
&\hspace{6cm}-\big(C_k(y_1,x_2)-C_k(y_1,y_2)\big)\big|^{2L}\big]\Big)\\
=:{}&T_{1,n,k}+T_{2,n,k},
} 
where the terms $T_{1,n,k}$ and $T_{2,n,k}$ can be handled similarly. Concentrating on the first one,  let us prove that for any $L\in \N$ there exist $K_L$ and $d_L$ depending only on $L$ such that 
\al{
&\sup_{\substack{x_1,y_1\in[0,1]\\ |x_1-y_1|\leq \delta}}\sup_{x_2\in[0,1]}\E\Big[\Big(\sqrt{n-|k|}\big|C_{n,k}(x_1,x_2)-C_{n,k}(y_1,x_2)-\big(C_k(x_1,x_2)-C_k(y_1,x_2)\big)\\
&\hspace{6.5cm}\times\big(\delta(1+|\log \delta|)^{d_L}\vee (n-|k|)^{-1}\big)^{-1/2}\Big)^{2L}\biggr]\leq K_L
} 

Let $\mathcal{T}_k:=\big\{t\in\{0,\dots,n-1\}|t,t+k\in\{0,\dots,n-1\}\big\}$. Observe that with 
$$\sigma:=2I\{x_1>y_1\}-1,\quad  M_1:=(x_1\wedge y_1,x_1\vee y_1], \text{ and }  M_2:=[0,x_2],$$ 
we have 
\al{
&C_{n,k}(x_1,x_2)-C_{n,k}(y_1,x_2)-\big(C_k(x_1,x_2)-C_k(y_1,x_2)\big)\\
={}&\frac{1}{n-|k|}\sum_{t\in\mathcal{T}_k}\Big( I\{U_{t+k}\in M_1,U_t\in M_2\}-\E[I\{U_{t+k}\in M_1,U_t\in M_2 \}]\Big) \sigma.
}
Since  $\Big( I\{U_{t+k}\in M_1,U_t\in M_2\}-\E[I\{U_{t+k}\in M_1,U_t\in M_2 \}]\Big)$ are centered, Theorem~2.3.2 of \cite{brillinger75} yields 
\al{
&\E\big[\big(C_{n,k}(x_1,x_2)-C_{n,k}(y_1,x_2)-\big(C_k(x_1,x_2)-C_k(y_1,x_2)\big)^{2L}\big]\\
={}&\frac{1}{(n-|k|)^{2L}}\sum_{\substack{\{\nu_1,\dots,\nu_R\}\\|\nu_j|\geq 2;\,j=1,\dots,R}}\prod_{r=1}^R\cum\Big(\sum_{t_{\xi}\in\mathcal{T}_k} I\{U_{t_{\xi}+k}\in M_1,U_{t_{\xi}}\in M_2\};\xi\in\nu_r\Big),
} 

where the sum runs over all partitions $\{\nu_1,\dots,\nu_R\}$ of $\{1,\dots,2L\}$.\\

Next, for any set $\nu_r$ 
with $|\nu_r|=q$ of a partition $\{\nu_1,\dots,\nu_R\}$, we have by Theorems 2.3.1 and 2.3.2 of \cite{brillinger75} 
\al{
&\cum\big(\sum_{t_{1}\in\mathcal{T}_k}I\{U_{t_{1}+k}\in M_1,U_{t_{1}}\in M_2\},\dots,\sum_{t_{q}\in\mathcal{T}_k}I\{U_{t_{q}+k}\in M_1,U_{t_{q}}\in M_2\}\big)\\
={}&\sum_{t_{1},\dots,t_{q}\in\mathcal{T}_k}\cum\big(I\{U_{t_{1}+k}\in M_1,U_{t_{1}}\in M_2\},\dots,I\{U_{t_{q}+k}\in M_1,U_{t_{q}}\in M_2\}\big)\\
={}&\sum_{t_{1},\dots,t_{q}\in\mathcal{T}_k}\sum_{\{\mu_1,\dots,\mu_N\}}\prod_{i=1}^N\cum\big(I\{U_u\in M_v\}; (u,v)\in\mu_i\big)
}

where the sum runs over all indecomposable partitions of the table
\begin{center}
\begin{tabular}{cc}
$(t_1+k,1)$\qquad &$(t_1,2)$\\
\vdots\qquad &\vdots\\
$(t_q+k,1)$\qquad &$(t_q,2)$.
\end{tabular}
\end{center}

Note that in $\prod_{i=1}^N\cum\big(I\{U_u\in M_v\}; (u,v)\in\mu_i\big)$ there are at most $q$ cumulants of order one for which we have 
\al{
\cum(I\{U_u\in M_v\})=\E[I\{U_u\in M_v\}]=\lambda(M_v),
}
where $\lambda$ denotes the Lebesgue measure. Moreover, we will never encounter the case of the first-order cumulants $\cum(I\{U_{t_{s}}\in M_2\})$ and $\cum(I\{U_{t_{s}+k}\in M_1\})$, for some $s=1,\dots,q$, both  appear in the product since  the partition then would be decomposable.

The other cumulants in $\prod_{i=1}^N\cum\big(I\{U_u\in M_v\}; (u,v)\in\mu_i\big)$ are at least of second-order and, as can be seen from the definition of a cumulant and the triangle inequality, 
\al{
\big|\cum\big(I\{U_u\in M_v\};(u,v)\in\mu_i:|\mu_i|\geq 2\big)|\leq{}&C\min\{\lambda(M_v):v\in\mu_i\}.
}
Furthermore, if we let $\mu_i:=\{(u_1,v_1),\dots,(u_p,v_p)\}$ and define 
$$m_{\mu_i}:=\max\{|u_k-u_l|: (u_k,v_k),(u_l,v_l)\in\mu_i,\,k,l=1,\dots,p\},$$ by Assumtion (C), we have 
\al{
\big|\cum\big(I\{U_u\in M_v\};(u,v)\in\mu_i:|\mu_i|\geq 2\big)\big|\leq{}&K_{p}(\rho^{1/|\mu_i|})^{m_{\mu_i}}.
}
Hence, for all cumulants of order greater than two, we have the bound 
\al{
\big|\cum\big(I\{U_u\in M_v\};(u,v)\in\mu_i:|\mu_i|\geq 2\big)\big|\leq{}&(C+K_p)\big(\min_{v\in\mu_i}\{\lambda(M_v)\}\wedge(\rho^{1/|\mu_i|})^{m_{\mu_i}}\big).
}
Thus, for one partition $\{\mu_1,\dots,\mu_N\}$ we obtain 
\al{
&\sum_{t_{1},\dots,t_{q}\in\mathcal{T}_k}\prod_{i=1}^N\cum\big(I\{U_u\in M_v\}; (u,v)\in\mu_i\big)\\
\leq{}&K_q\sum_{t_{1},\dots,t_{q}\in\mathcal{T}_k}\prod\limits_{\{i:|\mu_i=\{(u,v)\}|=1\}}\lambda(M_{v})\prod\limits_{\{j:|\mu_j|\geq 2\}}\big(\min\{\lambda(M_{v});(u,v)\in\mu_j\}\wedge (\rho^{1/|\mu_j|})^{m_{\mu_j}}\big).
}
Since $\lambda(M_1)=|x_1-y_1|\leq 1$ and $\lambda(M_2)=x_2\leq 1$, 
\al{
\prod\limits_{\{i:|\mu_i=\{(u,v)\}|=1\}}\lambda(M_{v})\leq \min\{\lambda(M_{v});(u,v)\in\mu_i,|\mu_i|=1\};
}
since $(a\wedge b)\cdot(c\wedge d)\leq (ac\wedge bd)$ for $a,b,c,d>0$, 
\al{
&\prod\limits_{\{j:|\mu_j|\geq 2\}}\big(\min\{\lambda(M_{v});(u,v)\in\mu_j\}\wedge (\rho^{1/|\mu_j|})^{m_{\mu_j}}\big)\\
\leq{}&\prod\limits_{\{j:|\mu_j|\geq 2\}}\big(\min\{\lambda(M_{v});(u,v)\in\mu_j\}\wedge (\rho^{1/|\mu_j|})^{\sum\limits_{\{j:|\mu_j|\geq 2\}}m_{\mu_j}}\big);
}
and since $\rho<1$, 
\al{
(\rho^{1/|\mu_j|})^{\sum\limits_{\{j:|\mu_j|\geq 2\}}m_{\mu_j}}\leq{}& (\rho^{1/(\max\{|\mu_1|,\dots,|\mu_N|\})})^{\max\{m_{\mu_j};|\mu_j|\geq 2, j=1,\dots,N\}}.
}

Thus, if we let $\tilde{\rho}:=\rho^{1/(\max\{|\mu_1|,\dots,|\mu_N|\})}$ and $$m_{\mu_1,\dots,\mu_N}:=\max\limits_{j=1,\dots,N}\{\max\{|u_i-u_{i^{\prime}}|:(u_i,v_i),(u_{i^{\prime}},v_{i^{\prime}})\in\mu_j,|\mu_j|\geq 2\}\},$$ we have 
\al{
&\sum_{t_{1},\dots,t_{q}\in\mathcal{T}_k}\prod_{i=1}^N\cum\big(I\{U_u\in M_v\}; (u,v)\in\mu_i\big)\\
\leq{}&K_q\sum_{t_{1},\dots,t_{q}\in\mathcal{T}_k}\big(\min\{\lambda(M_{v});(u,v)\in\mu_1,\dots,\mu_N\}\wedge \tilde{\rho}^{m_{\mu_1,\dots,\mu_N}}\big)\\
\leq{}&K_q\sum_{t_{1},\dots,t_{q}\in\mathcal{T}_k}\big(|x_1-y_1|\wedge \tilde{\rho}^{m_{\mu_1,\dots,\mu_N}}\big).
}

Next, 
\al{
K_q\sum_{t_{1},\dots,t_{q}\in\mathcal{T}_k}&\big(|x_1-y_1|\wedge  \tilde{\rho}^{m_{\mu_1,\dots,\mu_N}}\big)\leq{}\sum_{m=0}^{\infty}\sum_{\substack{t_1,\dots,t_q\in\mathcal{T}_k\\ m_{\mu_1,\dots,\mu_N}=m}}|x_1-y_1|\wedge \tilde{\rho}^m\\
\leq{}&\sum_{m=0}^{\infty}\#\{t_1,\dots,t_q\in\mathcal{T}_k: m_{\mu_1,\dots,\mu_N}=m\}|x_1-y_1|\wedge \tilde{\rho}^m.
}


In order to estimate the cardinality  of the set $\{t_1,\dots,t_q\in\mathcal{T}_k: m_{\mu_1,\dots,\mu_N}=m\}$, consider first the case $N=1$. We have
\al{
&\Big|\sum_{t_{1},\dots,t_{q}\in\mathcal{T}_k}\prod_{i=1}^N\cum\big(I\{U_u\in M_v\}; (u,v)\in\mu_i\big)\Big|\\
={}&\Big|\sum_{t_{1},\dots,t_{q}\in\mathcal{T}_k}\cum\big(I\{U_{t_1+k}\in M_1\},I\{U_{t_1}\in M_2\},\dots,I\{U_{t_q+k}\in M_1\},I\{U_{t_q}\in M_2\}\big)\Big|\\
\leq{}&\sum_{t_{1},\dots,t_{q}\in\mathcal{T}_k}K_{2q}\big(\rho^{1/(2q)}\big)^{\max\{|a-b|:a,b\in\{t_1,t_1+k,\dots ,t_q,t_q+k\}\}}\\
\leq{}&K_{2q}\sum_{m=0}^{\infty}\#\{t_1,\dots,t_q\in\mathcal{T}_k \max\{|a-b|:a,b\in\{t_1,t_1+k,\dots ,t_q,t_q+k\}=m\}\tilde{\rho}^{m}
}
where (since there are $n-|k|$ possibilities to fix one element $t_{j_0}$ of $\{t_1,\dots,t_q\in\mathcal{T}_k\}$ and   at most $m$ possible values for the remaining $t_j$, $j=1\dots,q,\,j\neq j_0$)
$$\#\{t_1,\dots,t_q\in\mathcal{T}_k \max\{|a-b|:a,b\in\{t_1,t_1+k,\dots ,t_q,t_q+k\}=m\}\leq c_q(n-|k|)m^{q-1}.$$ 
For the case $N\geq 2$, 
\al{
&\sum_{t_{1},\dots,t_{q}\in\mathcal{T}_k}\prod_{i=1}^N\cum_{|\mu_i|}\big(I\{U_u\in M_v\}; (u,v)\in\mu_i\big)\\
\leq{}&\sum_{m=0}^{\infty}\#\{t_1,\dots,t_q\in\mathcal{T}_k \max\{m_{\mu_j};|\mu_j|\geq 2, j=1,\dots,N\}=m\}|x_1-y_1|\wedge \tilde{\rho}^m,
}
where 
\begin{equation}\label{hd2}
\#\{t_1,\dots,t_q\in\mathcal{T}_k: \max_{j:|\mu_j|\geq 2} m_{\mu_j} = m\}\leq c_q(n-|k|)m^{q-1}.
\end{equation} 
In order to prove this, start by considering the set $\mu_{i_0}$ of one partition $\{\mu_1,\dots,\mu_N\}$ which contains either $t_1$ or $t_1+k$ or both. By indecomposability of the partition there exists at least one other $t_s$ or $t_s+k$ in $\mu_{i_0}$ such that $t_s+k$ or $t_s$ 
 are not contained in $\mu_{i_0}$. Hence, there are~$n-|k|$ possible values for $t_1$ and at most $m$ possible values for any other $t_s$ so that either~$t_s$ or exclusively $t_s+k$ is contained in $\mu_{i_0}$ since $k$ is fixed. Next, observe that by indecomposability of the partition, all sets $\mu_j$ hook [for a precise definition see page 20 in \cite{brillinger75}] and thus there exists a $\mu_{j_0}$ such that $t_s$ is contained in $\mu_{i_0}$ and $t_s+k$ in $\mu_{j_0}$ or vice versa. Again by indecomposability we find another $t_r$ or exclusively $t_r+k$ in $\mu_{j_0}$ for which we have at most $m$ choices so that $\max\{m_{\mu_j};|\mu_j|\geq 2, j=1,\dots,N\}=m$. Continuing this argumentation until the maximum over all sets $\mu_j$ have been taken into consideration, we see that $\#\{t_1,\dots,t_q\in\mathcal{T}_k: \max\{m_{\mu_j};|\mu_j|\geq 2, j=1,\dots,N\}=m\}$ is at most of the order $(n-|k|)m^{q-1}$, since the indecomposable partitions $\{\mu_1,\dots,\mu_N\}$ yielding the highest order are those where each set $\mu_j$ is of size $2$ and contains $t_s$ or~$t_s+k$ and $t_r$ or~$t_r+k$.
%
%
%
%

Therefore, \eqref{hd2} follows and 
\al{
&\sum_{m=0}^{\infty}\#\{t_1,\dots,t_q\in\mathcal{T}_k :\max\{m_{\mu_j};|\mu_j|\geq 2, j=1,\dots,N\}=m\}|x_1-y_1|\wedge \tilde{\rho}^m\\
&\hspace{5cm}\leq{}c_q (n-|k|)\sum_{m=0}^{\infty}  m^{q-1}|x_1-y_1|\wedge \tilde{\rho}^m.
}
Observe that for some constant $K$, 
\al{
\sum_{m=0}^{\infty}  m^{q-1}(\varepsilon\wedge \rho^m)\leq K\varepsilon(1+|\log \varepsilon|)^q
}
because
\bi
\item[(1)] if $\varepsilon\geq \tilde{\rho}$,  then 
$
\sum_{m=0}^{\infty}  m^{q-1}(\varepsilon\wedge \tilde{\rho}^m)=\sum_{m=0}^{\infty}  m^{q-1} \tilde{\rho}^m<\infty
$;\vspace{1mm}
\item[(2)] if $\varepsilon< \tilde{\rho}$, setting $m_{\varepsilon}:=\log\varepsilon/\log \tilde{\rho}$ (so that $m_{\varepsilon}$ is   such that $\tilde{\rho}^m<\varepsilon$ for any $m>m_{\varepsilon}$), then $\tilde{\rho}^{m_{\varepsilon}}=\varepsilon$ and 
\al{
\sum_{m=0}^{\infty}  m^{q-1}(\varepsilon\wedge \tilde{\rho}^m)\leq{}&\sum_{m\leq m_{\varepsilon}}\varepsilon m^{q-1}+\sum_{m> m_{\varepsilon}}\tilde{\rho}^{m} m^{q-1}\\
\leq{}&m_{\varepsilon}m_{\varepsilon}^{q-1}\varepsilon+\tilde{\rho}^{m_{\varepsilon}}\sum_{m=0}^{\infty}(m+m_{\varepsilon})^{q-1}\tilde{\rho}^m\\
\leq{}&m_{\varepsilon}^q\varepsilon+\varepsilon m_{\varepsilon}^q\sum_{m=0}^{\infty}(m+1)^{q-1}\tilde{\rho}^m\\
\leq{}&C_q^{\prime}\varepsilon\Big(1+\Big|\frac{\log \varepsilon}{\log \tilde{\rho}}\Big|^q\Big)
\leq{}
C_{q,\tilde{\rho}}\varepsilon(1+|\log \varepsilon|)^q.
}
\ei

Hence, in total, for an indecomposable decomposition $\{\mu_1,\dots,\mu_N\}$, we have 
\al{
\sum_{t_{1},\dots,t_{q}\in\mathcal{T}_k}\prod_{i=1}^N\cum\big(I\{U_u\in M_v\}; (u,v)\in\mu_i\big)\leq C_q(n-|k|)|x_1-y_1|(1+|\log |x_1-y_1||)^q
}
and therefore, for one set $\nu_r$ of $q$ elements, 
\al{
\cum\big( I\{U_{t_{\xi}+k}\in M_1,U_{t_{\xi}}\in M_2\};\xi\in\nu_r\big)\leq \tilde{C}_{q}(n-|k|)|x_1-y_1|(1+|\log |x_1-y_1||)^q.
}


Thus, for any partition $\{\nu_1,\dots,\nu_R\}$ with $|\nu_j|\geq 2; j=1,\dots,R$ of $\{1,\dots,2L\}$, 
\al{
&\prod_{r=1}^R\cum\big( \sum_{t_{\xi}\in\mathcal{T}_k}I\{U_{t_{\xi}+k}\in M_1,U_{t_{\xi}}\in M_2\};\xi\in\nu_r\big)\\
\leq{}&\tilde{C}_R(n-|k|)^{R}\big(|x_1-y_1|(1+|\log |x_1-y_1||)^{\max\{|\nu_j|;j=1,\dots,R\}}\big)^R
}
and, if we let $d_R:=\max\{|\nu_j|;j=1,\dots,R\}$ and $d:=\max\{d_1,\dots,d_L\}$, we obtain 
\al{
&\E\big[\big(\sqrt{n-|k|}\big|C_{n,k}(x_1,x_2)-C_{n,k}(y_1,x_2)-\big(C_k(x_1,x_2)-C_k(y_1,x_2)\big|\big)^{2L}\big]\\
\leq{}&\tilde{K}_{1,L}\sum_{R=1}^L(n-|k|)^{R-L}(|x_1-y_1|(1+|\log |x_1-y_1||)^d)^R\\
\leq{}&L\tilde{K}_{1,L}((n-|k|)^{-1}\vee|x_1-y_1|(1+|\log |x_1-y_1||)^d)^L\\
\leq{}&K_{1,L}((n-|k|)^{-1}\vee|x_1-y_1|(1+|\log |x_1-y_1||)^d)^L,
}
that is, 
\al{
&\sup_{\substack{x_1,y_1\in[0,1]\\ |x_1-y_1|\leq \delta}}\sup_{x_2\in[0,1]}\E\big[\big(\sqrt{n-|k|}\big|C_{n,k}(x_1,x_2)-C_{n,k}(y_1,x_2)-\big(C_k(x_1,x_2)-C_k(y_1,x_2)\big|\big)^{2L}\big]\\
&\hspace{3cm}\leq K_{1,L}((n-|k|)^{-1}\vee\delta(1+|\log \delta|)^d)^L.
}

Analogously, 
\al{
&\sup_{\substack{x_2,y_2\in[0,1]\\ |x_2-y_2|\leq \delta}}\sup_{y_1\in[0,1]}\E\big[\sqrt{n-|k|}\big|C_{n,k}(y_1,x_2)-C_{n,k}(y_1,y_2)-\big(C_k(y_1,x_2)-C_k(y_1,y_2)\big)\big|^{2L}\big]\Big)\\
&\hspace{3cm}\leq K_{2,L}((n-|k|)^{-1}\vee\delta(1+|\log \delta|)^{d_L})^L,
}
and hence
\al{
&\sup_{\substack{(x_1,x_2),(y_1,y_2)\in[0,1]^2\\ \lVert(x_1,x_2)-(y_1,y_2)\rVert_1\leq \delta}}\E[\sqrt{n-|k|}\big|C_{n,k}(x_1,x_2)-C_{n,k}(y_1,y_2)-\big(C_k(x_1,x_2)-C_k(y_1,y_2)\big)\big|^{2L}]\\
&\hspace{3cm}\leq K_L((n-|k|)^{-1}\vee\delta(1+|\log \delta|)^d)^L,
}
which completes the proof.
\end{proof}

\section{Additional simulation results}



\subsection{Additional simulation results for the test for time-reversibility}

We show additional simulation results for the test for time-reversibility introduced in Section \ref{sec:timerev}. We set the sample size $n\in\{100, 128, 150, 200, 256,$ $400, 512, 700, 1024\}$ and the block size, for a given~$n$, to $b \in B(n) := \{2^4, 2^5, \ldots, n/2\}$, the range of maximum for frequency as $\{2\pi \ell/32; \linebreak\ell = 0,1,\ldots,16\}$, and the range of maxima for quantiles as $\{\tau_1,\tau_2=k/8; k = 1,\ldots,7\}$. The weight functions $s_1,\ldots,s_5$ defined in the Appendix and the significance level as $\alpha=0.05$ are employed.

The simulation procedure is as follows: generate time series and calculate the $p$-values based on $T_{\rm TR1}^{(n,b,t)}$ and $T_{\rm TR1\_fpc}^{(n,b,t)}$, which is defined as $T_{\rm TR1\_fpc}^{(n,b,t)}:=(1 - b/n)^{-1/2}T_{\rm TR1}^{(n,b,t)}$. Then, iterate $R=1000$ times and compute empirical size or power. In the figures, $b$ is chosen by the rule of thumb defined by \eqref{eqn:rt_bw}.

Figures \ref{fig:Appx:TR_nofpc_M8_M9}--\ref{fig:Appx:TR_fpc_M10_M11} illustrate tha fact  that the power of the tests increases as the degree of time-irreversibility increases and as the sample size increases.
The weight functions $s_1$, $s_2$, and~$s_4$ provide better power among $s_1,\ldots,s_5$ since the tests based on $s_3$ and $s_5$ have low power for many models and M10, respectively.

Figures \ref{fig:TR_nofpc_weights}--\ref{fig:TR_fpc_weights} display results with the same settings as Figures \ref{fig:TR_nofpc}--\ref{fig:TR_fpc} in the main manuscript but with weight functions $s_1,s_2,s_3,s_5$.

\begin{figure}[H]
	\begin{center}
\includegraphics[width = \linewidth]{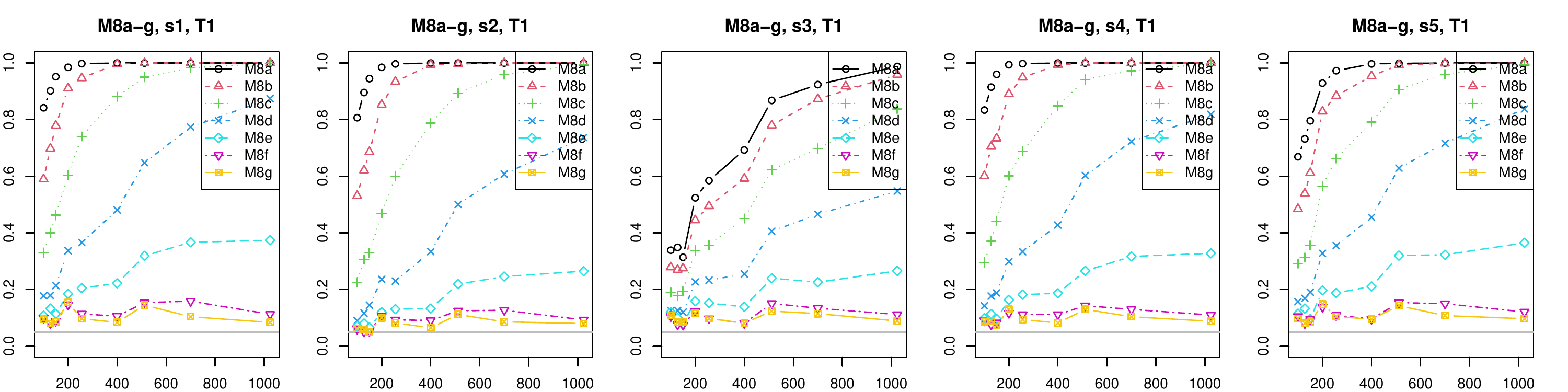}
\includegraphics[width = \linewidth]{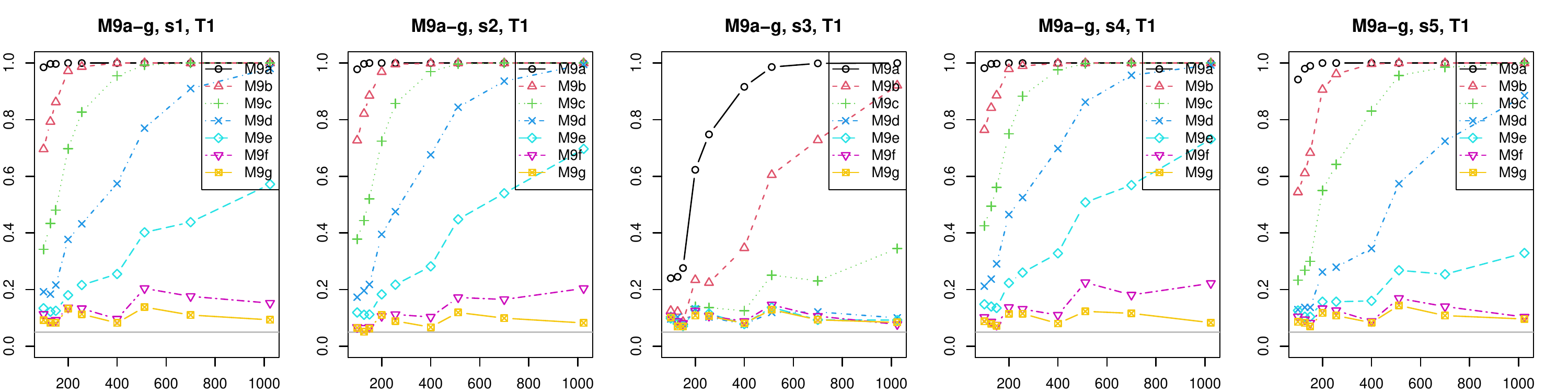}
	\end{center}
	\vspace{-0.4cm}
	\caption{Empirical size (top) and power (bottom) of the tests for time-reversibility based on $T_{\rm TR1}^{(n,b,t)}$ described in Section   \ref{sec:timerev}.
The upper plots and lower plots correspond to M8$a$--$g$ and M9$a$--$g$, respectively. 
Columns correspond to the weight functions $s_1,\ldots,s_5$ from left to right, respectively.
The horizontal axis of the plots corresponds to the parameters of models ($\lambda_i$ or $\gamma_i^{-1}$) and the vertical axis corresponds to empirical power.
}
	\label{fig:Appx:TR_nofpc_M8_M9}
\end{figure}

\begin{figure}[H]
	\begin{center}
\includegraphics[width = \linewidth]{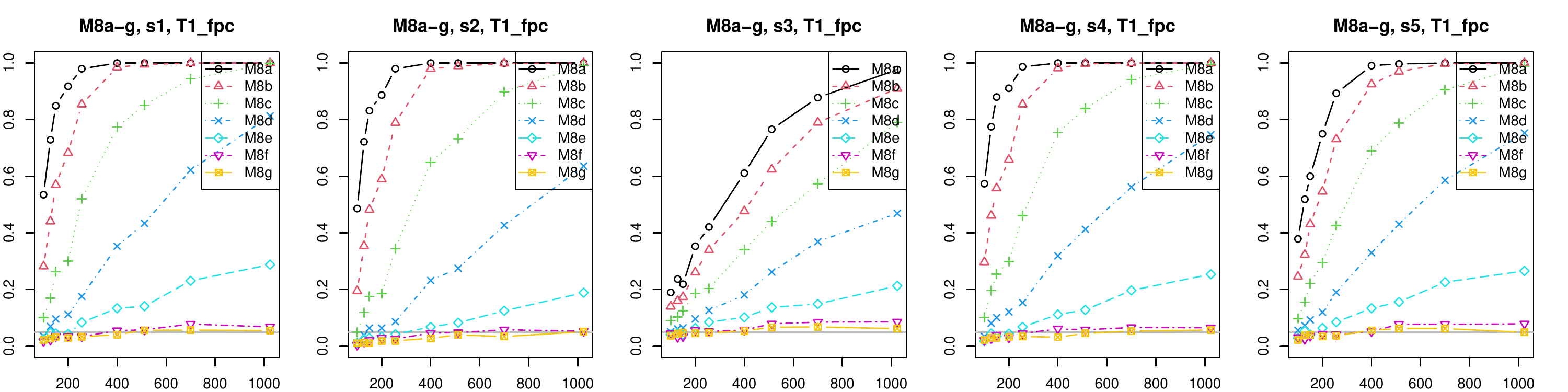}
\includegraphics[width = \linewidth]{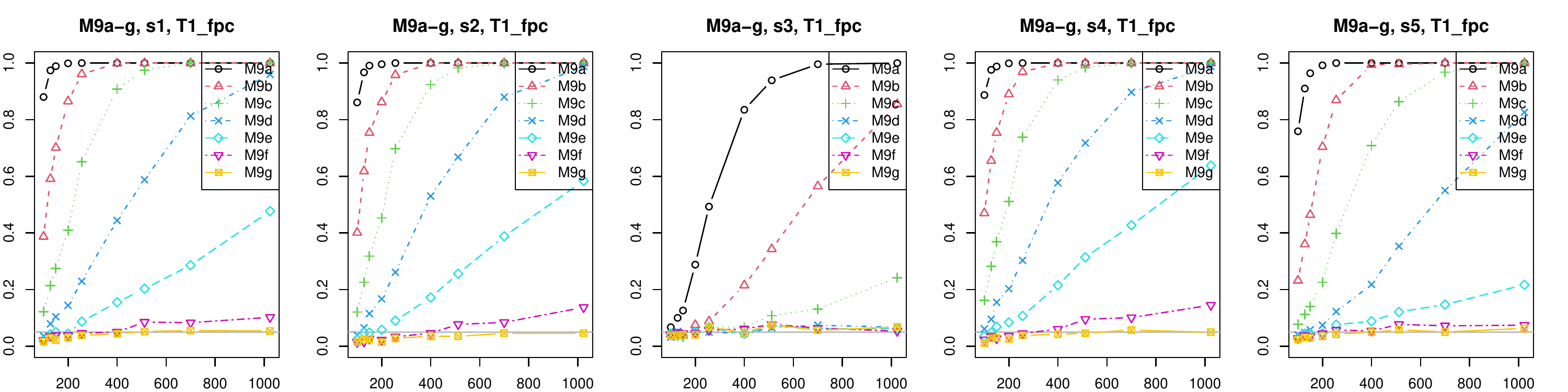}
	\end{center}
	\vspace{-0.4cm}
	\caption{Empirical size (top) and power (bottom) of the tests for time-reversibility based on $T_{\rm TR1\_fpc}^{(n,b,t)}$ described in Section   \ref{sec:timerev}.
The upper plots and lower plots correspond to M8$a$--$g$ and M9$a$--$g$, respectively. 
Columns correspond to the weight functions $s_1,\ldots,s_5$ from left to right, respectively.
The horizontal axis of the plots corresponds to the parameters of models ($\lambda_i$ or $\gamma_i^{-1}$) and the vertical axis corresponds to empirical power.}
	\label{fig:Appx:TR_fpc_M8_M9}
\end{figure}

\begin{figure}[H]
	\begin{center}
\includegraphics[width = \linewidth]{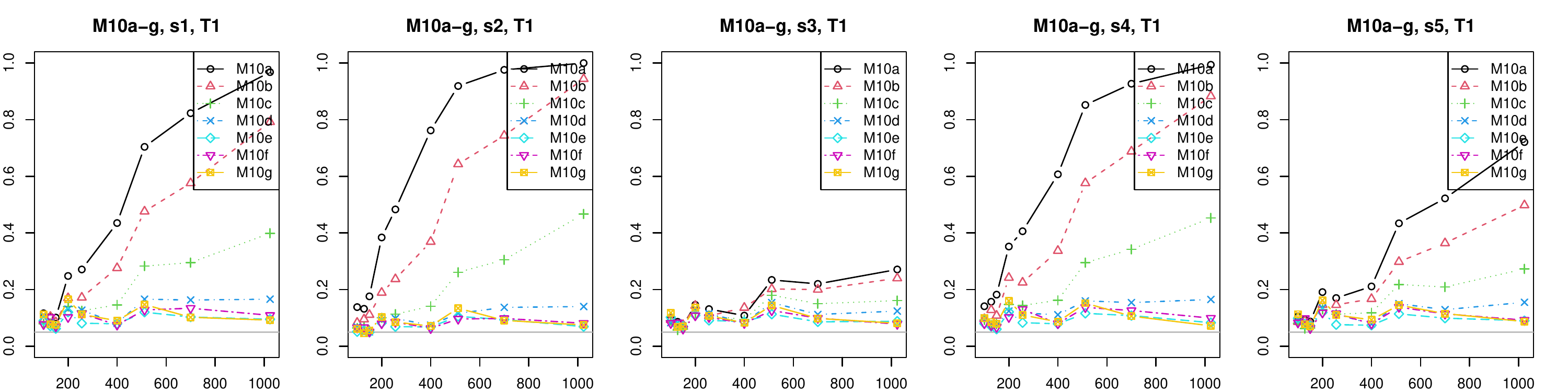}
\includegraphics[width = \linewidth]{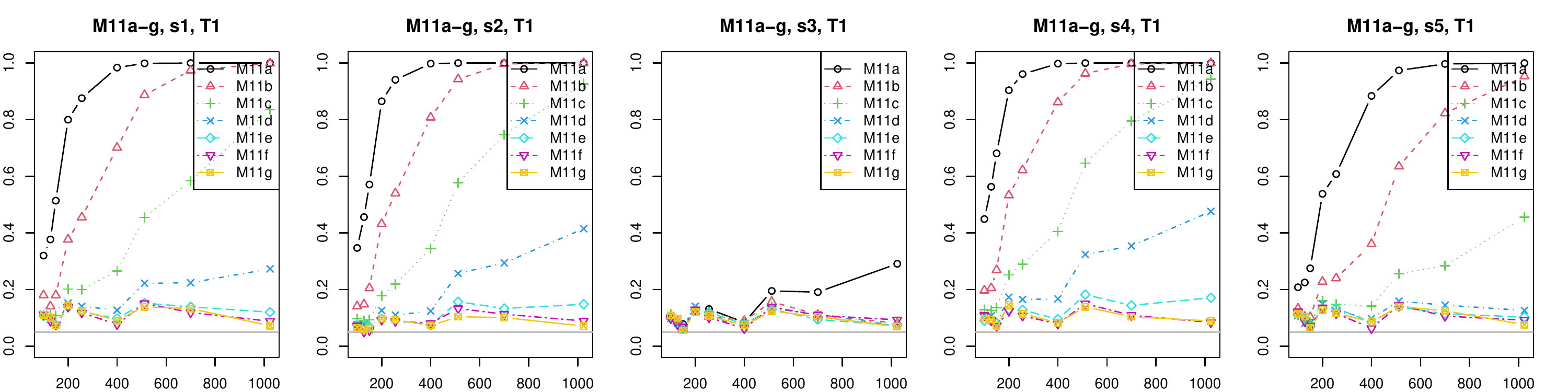}
	\end{center}
	\vspace{-0.4cm}
	\caption{Empirical size (top) and power (bottom) of the tests for time-reversibility based on $T_{\rm TR1}^{(n,b,t)}$ described in Section   \ref{sec:timerev}.
The upper plots and lower plots correspond to M10$a$--$g$ and M11$a$--$g$, respectively. 
Columns correspond to the weight functions $s_1,\ldots,s_5$ from left to right, respectively.
The horizontal axis of the plots corresponds to the parameters of models ($\lambda_i$ or $\gamma_i^{-1}$) and the vertical axis corresponds to empirical power.}
	\label{fig:Appx:TR_nofpc_M10_M11}
\end{figure}

\begin{figure}[H]
	\begin{center}
\includegraphics[width = \linewidth]{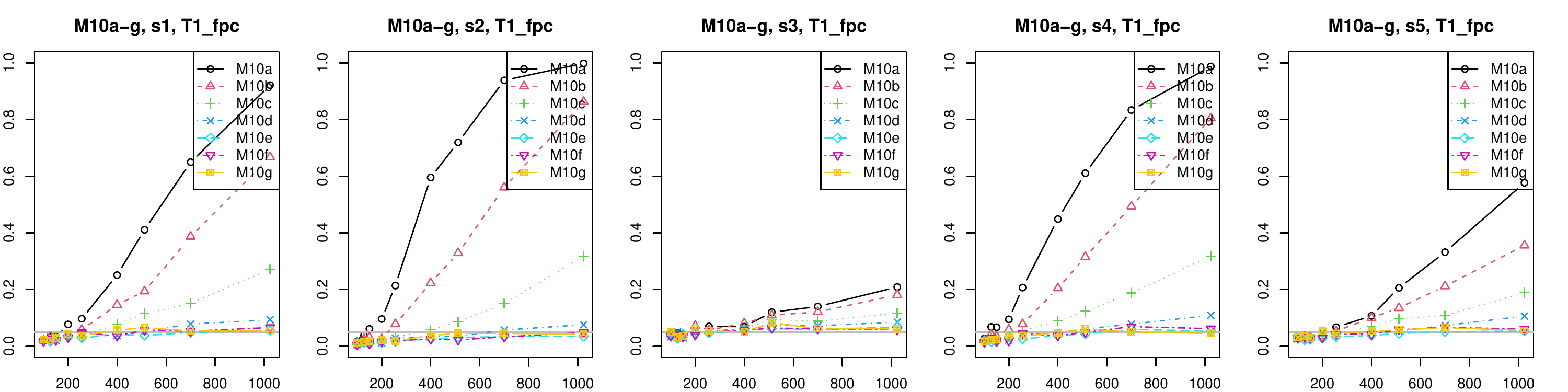}
\includegraphics[width = \linewidth]{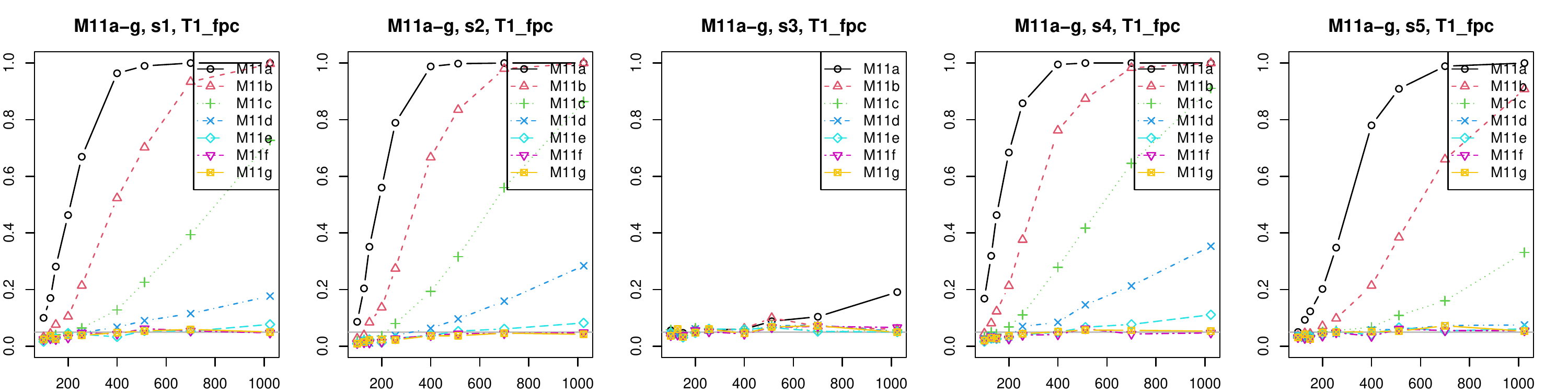}
	\end{center}
	\vspace{-0.4cm}
	\caption{Empirical size (top) and power (bottom) of the tests for time-reversibility based on $T_{\rm TR1\_fpc}^{(n,b,t)}$ described in Section   \ref{sec:timerev}.
The upper plots and lower plots correspond to M10$a$--$g$ and M11$a$--$g$, respectively. 
Columns correspond to the weight functions $s_1,\ldots,s_5$ from left to right, respectively.
The horizontal axis of the plots corresponds to the parameters of models ($\lambda_i$ or $\gamma_i^{-1}$) and the vertical axis corresponds to empirical power.}
	\label{fig:Appx:TR_fpc_M10_M11}
\end{figure}

\begin{figure}[ht]
	\begin{center}
\includegraphics[width = \linewidth]{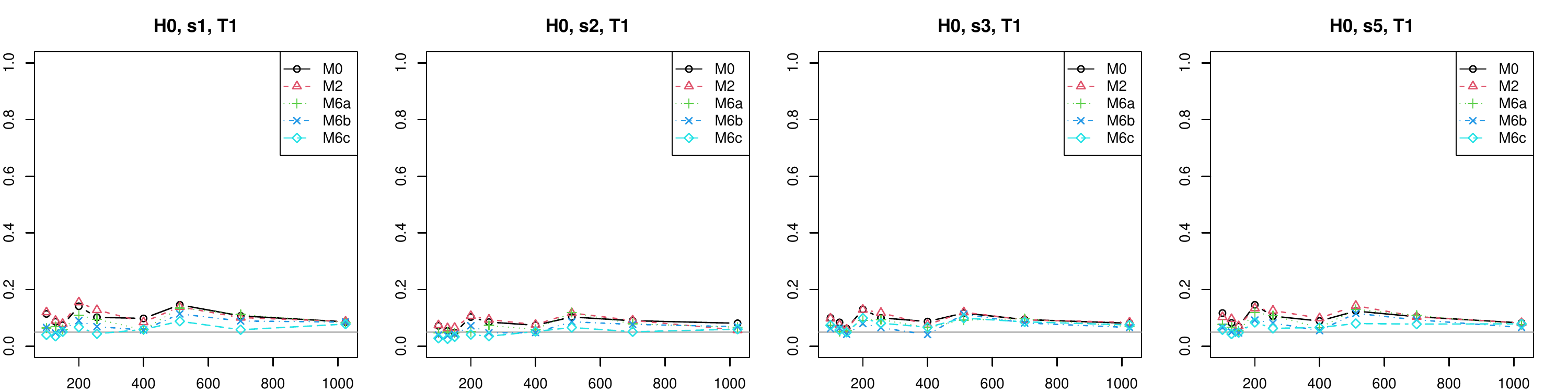}
\includegraphics[width = \linewidth]{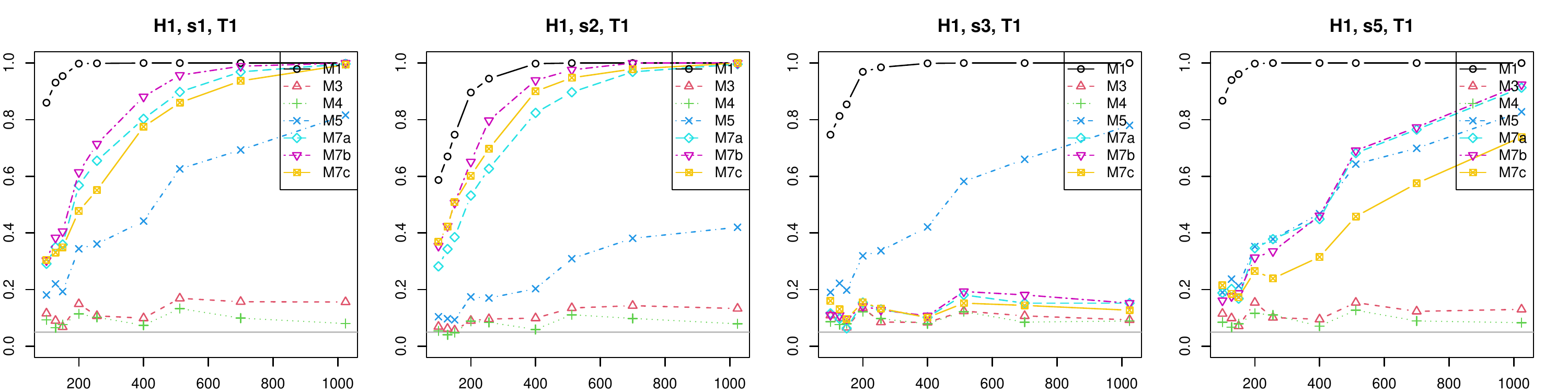}
	\end{center}
	\caption{\small Empirical sizes (top, time-reversible models M0, M2, and M6a-c) and powers (bottom, time-irreversible models M1, M3, M4, M5, and M7a-c) as functions of $n$, of the tests for time-reversibility based on $T_{\rm TR1\_fpc}^{(n,b,t)}$ (without finite-population correction). Columns correspond to weight functions $s_1,s_2,s_3,s_5$, respectively.}
	\label{fig:TR_nofpc_weights}
\end{figure}

\begin{figure}[ht]
	\begin{center}
		\includegraphics[width = \linewidth]{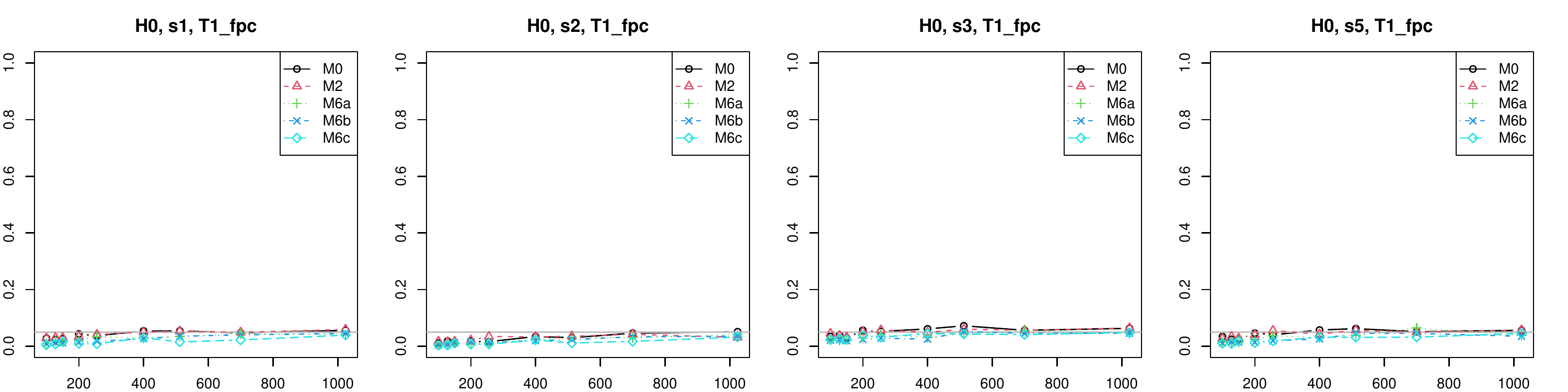}
		\includegraphics[width = \linewidth]{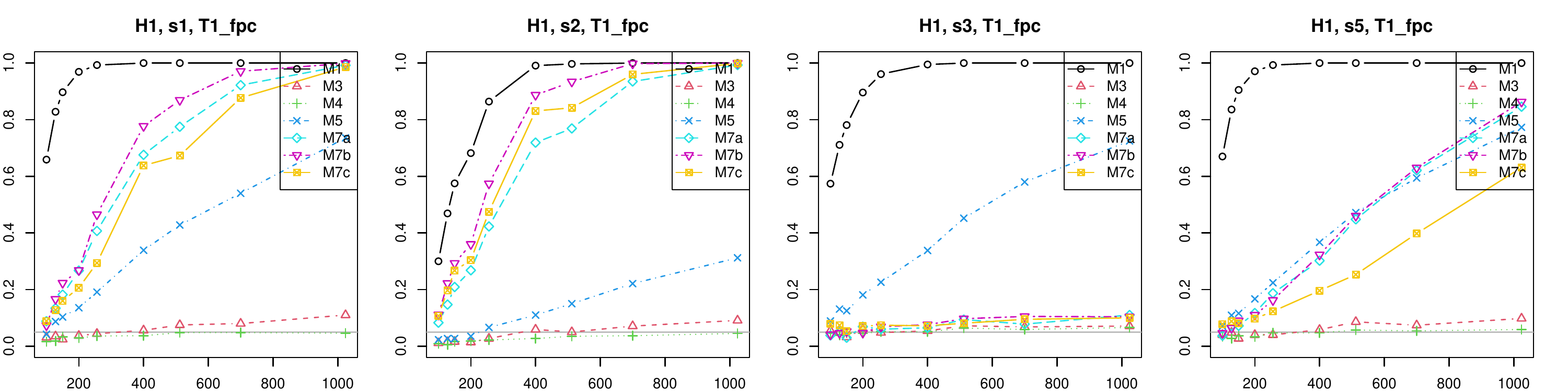}
	\end{center}
	\caption{\small Empirical sizes (top, time-reversible models M0, M2, and M6a-c) and powers (bottom, time-irreversible models M1, M3, M4, M5, and M7a-c) as functions of $n$, of the tests for time-reversibility based on $T_{\rm TR1\_fpc}^{(n,b,t)}$ (with  finite-population correction). Columns correspond to weight functions $s_1,s_2,s_3,s_5$, respectively.}
	\label{fig:TR_fpc_weights}
\end{figure}

\clearpage

\subsection{Additional simulation results for the test for asymmetry in tail dynamics}

Here we provide additional simulation results with the same settings as in Figures \ref{fig:EQ_nofpc} -- \ref{fig:EQ_fpc} in the main manuscript but with weight functions $s_1,s_2,s_3,s_5$.

\begin{figure}[h!]
\begin{center}
\includegraphics[width = \linewidth]{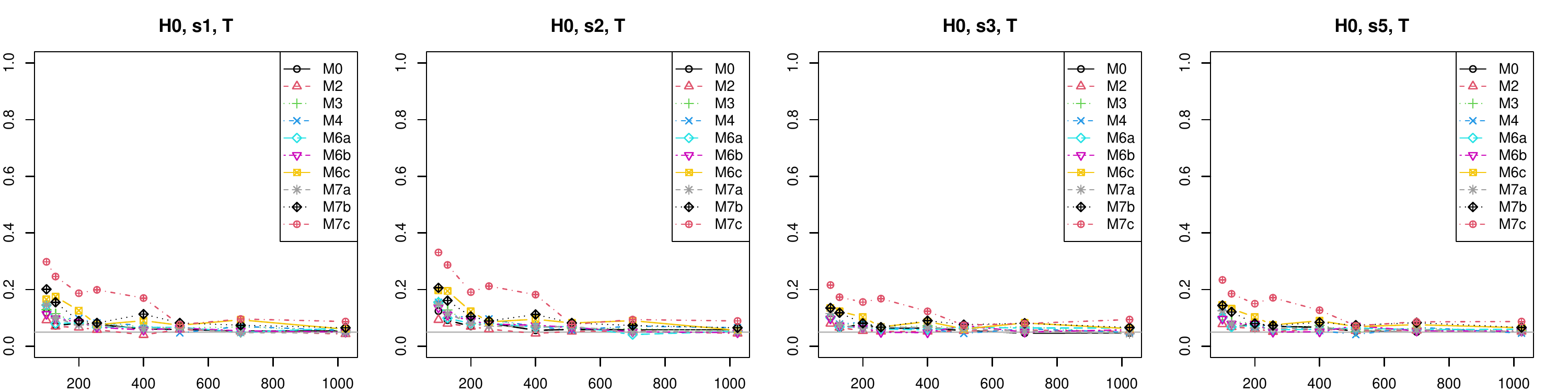}
\includegraphics[width = \linewidth]{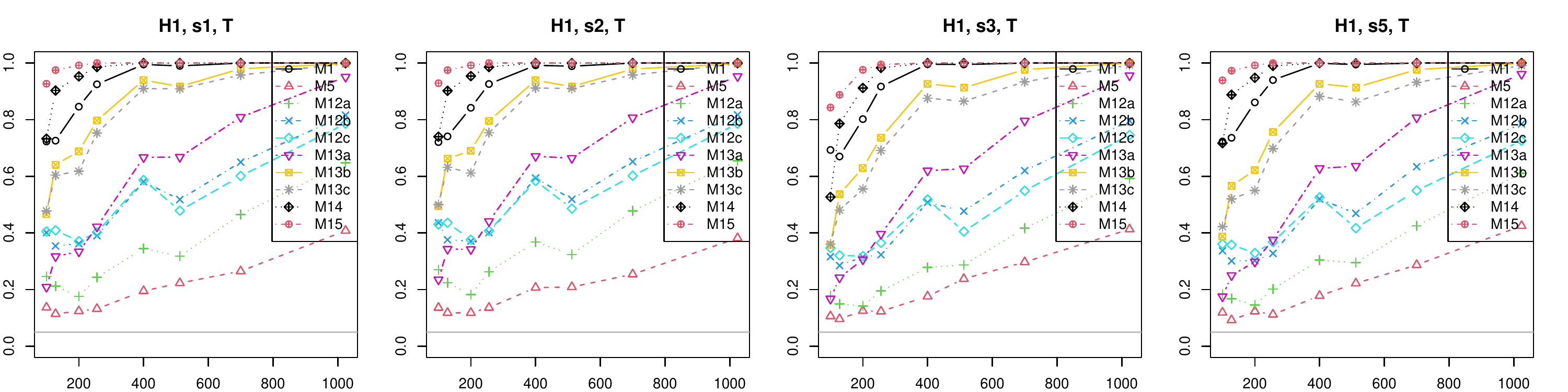}
\end{center}
\vspace{-0.4cm}
\caption{\small 
Empirical sizes (top) and powers (bottom), as functions of $n$, of the tests for tail symmetry based on $T_{\rm EQ}^{(n,b,t)}$ under various models. Columns correspond to weight functions $s_1,s_2,s_3,s_5$, respectively.
\vspace{-4mm}}
\label{fig:EQ_nofpc_weights}
\end{figure}

\begin{figure}[h!]
\begin{center}
\includegraphics[width = \linewidth]{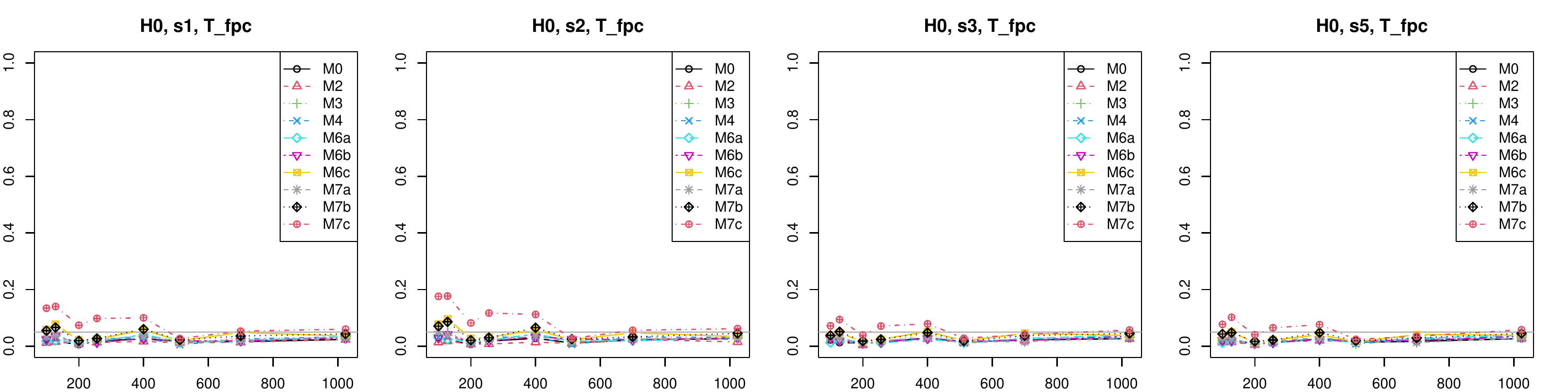}
\includegraphics[width = \linewidth]{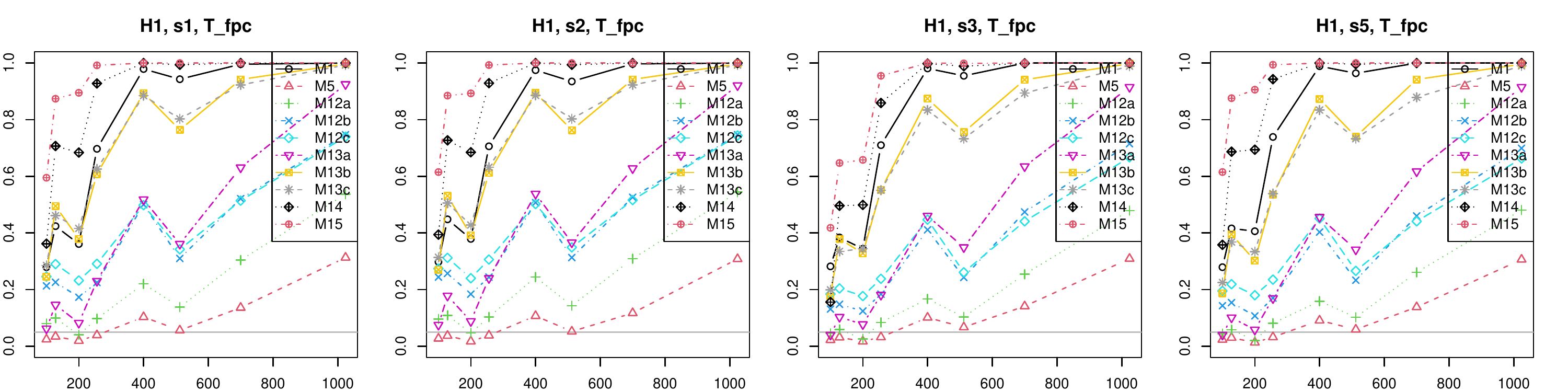}
\end{center}
\vspace{-0.4cm}
\caption{\small 
Empirical sizes (top) and powers (bottom), as functions of $n$, of the tests for tail symmetry based on $T_{\rm EQ\_fpc}^{(n,b,t)}$ (with finite--population correction) under various models. Columns correspond to weight functions $s_1,s_2,s_3,s_5$, respectively.
\vspace{-4mm}}
	\label{fig:EQ_fpc_weights}
\end{figure}

\end{document}